\documentclass[a4paper, 11pt, reqno]{amsart}
\usepackage[T1]{fontenc}
\usepackage[utf8]{inputenc}
\usepackage{amsmath}
\usepackage{amsfonts}
\usepackage{amssymb}
\usepackage{amsthm}
\usepackage[english]{babel}
\usepackage{hyperref}
\usepackage{xcolor}
\usepackage{soul}

\theoremstyle{definition}
\newtheorem{defin}{Definition}[section]
\newtheorem{ex}[defin]{Example}
\theoremstyle{plain}
\newtheorem{theo}[defin]{Theorem}
\newtheorem{lemma}[defin]{Lemma}
\newtheorem{obs}[defin]{Remark}
\newtheorem{prop}[defin]{Proposition}
\newtheorem{cor}[defin]{Corollary}

\renewenvironment{abstract}
{\par\noindent\textbf{\abstractname.}\ \ignorespaces}
{\par\medskip}

 \title[Discrete groups of hyperbolic spaces]{Discrete groups of packed, non-positively curved, Gromov hyperbolic  metric spaces}

\author{Nicola Cavallucci (corresponding author)}
\thanks{N.Cavallucci is partially supported by the SFB/TRR 191, funded by the DFG}
\author{Andrea Sambusetti}
\thanks{A. Sambusetti is member of the Differential Geometry section of the GNSAGA-INdAM}

\address{Andrea Sambusetti, Dipartimento di Matematica “Guido
	Castelnuovo”, SAPIENZA Universita di Roma, Piazzale Aldo Moro 5, I-00185 `
	Roma.}
\address{Nicola Cavallucci, Karlsruhe Institute of Technology, Engelstrasse 2, D-76128 Karlsruhe}
\email{n.cavallucci23@gmail.com, sambuset@mat.uniroma1.it}
\date{}

\begin{document}
\maketitle

\footnotesize
\begin{abstract}
	We prove a quantitative version of the classical Tits' alternative for discrete groups acting on packed Gromov-hyperbolic spaces supporting a convex geodesic bicombing. Some geometric consequences, as uniform estimates on systole, diastole, algebraic entropy and critical exponent of the groups, will be presented. Finally we will study the behaviour of these group actions under limits, providing new examples of compact classes of metric spaces. 
\end{abstract}
\normalsize
\tableofcontents
\pagebreak

\section{Introduction}
The aim of this paper is   to extend the results of  \cite{BCGS17} to {\em non-cocompact} actions of discrete groups of  Gromov hyperbolic   spaces.
 In fact, while non-cocompact actions of  groups $\Gamma$ on general metric spaces $X$ are  considered to some extent in  \cite{BCGS17} (especially in Chapter 6),  the underlying assumption in that work is that  the group  $\Gamma$ admits {\em also} a cocompact action  or a ``well-behaved'' action on some  Gromov $\delta_0$-hyperbolic space $X_0$,   namely   with some prescribed    lower bound $\vert g \vert \geq \ell_0$ for the  minimal or asymptotic displacement of every $g \in \Gamma$ on $X_0$ (such groups are called {\em $(\delta_0, \ell_0)$-thick} by the authors).

On the other hand,   we will be interested here to general  discrete group actions on Gromov-hyperbolic metric spaces $X$ satisfying a {\em packing condition}: \linebreak that is, an upper bound of the cardinality of any  $2r$-separated set   inside any ball of radius $3r$. This condition  bounds  the metric complexity of the space $X$ \linebreak at some {\em fixed scale}, and  should be thought as a macroscopic, weak replacement of a  lower bound on the curvature: for instance, in the case of Riemannian manifolds, it is strictly weaker than a lower Ricci curvature bound. It also seems, a priori,  a much  stronger condition  than the {\em bounded entropy} assumption adopted  in \cite{BCGS17} (see Lemma \ref{entropybound}); but, actually,  most of the work in \cite{BCGS17} is  proving  that  a group $\Gamma$  acting {\em cocompactly}    on  a $\delta$-hyperbolic space $X$  with  bounded entropy  does satisfies an inequality \`a-la  Bishop-Gromov, hence a uniform packing condition (see Theorem 5.1 of \cite{BCGS17}, and Sec.2  therein for a complete comparison of the packing  assumptions with curvature or entropy bounds).

We will restrict our attention   to actions of groups on Gromov-hyperbolic spaces which  
are  complete and possess a   {\em convex geodesic bicombing}
\linebreak (see Sec. \ref{sub-bicombing} for precise definitions and properties).  
The notion of metric space with  a convex geodesic bicombing is one of the weakest forms   of non-positive curvature,  and  includes many geometric classes like CAT$(0)$  and  Busemann convex metric spaces, Banach spaces, injective metric spaces, etc. For instance, it is known that any Gromov hyperbolic group acts geometrically on a proper metric space supporting a convex geodesic bicombing (cp. \cite{Des15}), while the existence of a geometric action on a CAT(0)-space for these groups is a well-known open problem.
Moreover this class of spaces is closed under limit operations (as ultralimits), while Busemann convex spaces are not; a property that makes it preferable to work in this setting.

We will require an additional property to our classes of metric spaces: the {\em geodesically completeness} of the bicombing, which is equivalent to the standard geodesically completeness assumption in case of CAT$(0)$ or Busemann convex spaces. This is an assumption which  is usually required, even in case of CAT$(0)$-spaces, in order to control much better the local and asymptotic geometry (see the foundational works of B. Kleiner \cite{Kle99} and of A. Lytchak and K. Nagano \cite{Nag18}, \cite{LN19}, \cite{LN20}). 
For instance in \cite{CavS20} we proved that, for complete and geodesically complete CAT$(0)$-spaces, a packing condition at some scale $r$ yields explicit control of the packing condition at \emph{any} scale. With similar techniques the same is true for spaces supporting a geodesically complete, convex, geodesic bicombing (see Proposition \ref{packingsmallscales}). 
 
\vspace{2mm}
 A couple $(X,\sigma)$, where $X$ is a complete metric space and $\sigma$ a geodesically complete, convex geodesic bicombing, will be called a \emph{GCB-space} for short. \linebreak
 The set up of  
{\em  $\delta$-hyperbolic GCB-spaces satisfying a uniform $P_0$-packing condition at some fixed scale $r_0$} establishes a solid  framework where many  tools are available, and which is  large enough to contain many interesting examples besides Riemannian manifolds  (see Sec.\ref{sub-packing} and \cite{CavS20} for  examples). \\
First of all, due to the propagation of the packing property at scales smaller than the hyperbolicity constant, it is possible to obtain precise estimates of the {\em $\varepsilon$-Margulis' domain}  of an isometry on such a space $X$,  for values of $\varepsilon$ smaller than $\delta$, a tool which will be extensively used in the paper. \\
Secondly the $P_0$-packing condition at some fixed scale $r_0$ on $X$ implies, by   Breuillard-Green-Tao's work, an analogue for metric spaces of the celebrated {\em Margulis Lemma} for Riemannian manifolds: {\em there exists a constant $\varepsilon_0$, only depending on the packing constants $P_0,r_0$,
such that for every discrete group of isometries $\Gamma$ of   $X$ 
the 	$\varepsilon_0$-almost stabilizer $\Gamma_{\varepsilon_0} (x)$ of any point $x$ is virtually nilpotent}  (cp. \cite{BGT11}, Corollary 11.17); that is, the elements of $\Gamma$ which displace $x$ less than $\varepsilon_0$ generate a virtually nilpotent group. \\
Finally,  this class is   {\em compact} for the Gromov-Hausdorff topology, for fixed $\delta$, $P_0$  and $r_0$, 
a fact  which we will use to investigate  limits of group actions on these spaces in Section \ref{sec-compactness}.

\vspace{3mm}
The first basic tool we needed to develop for our analysis, is a quantitative form of the Tits alternative in the framework of Gromov-hyperbolic, GCB-spaces. Recall that the classical  Tits alternative, proved by J. Tits \cite{tits1972}, says that any finitely generated linear group $\Gamma$ over a commutative field $\mathbb{K}$ either is virtually solvable or contains a non-abelian free subgroup. \linebreak
This result has been extended to other classes of groups during the years, for instance (restricting ourselves to non-positively curved spaces, and without intending to be exhaustive) to discrete, non-elementary groups of isometries of  Gromov-hyperbolic spaces,  of proper CAT(0)-spaces  of rank one, of finite dimensional CAT$(0)$-cube complexes etc.,   \cite{Gro87}, \cite{Ham09}, \cite{SW05}, \cite{Osi15}. \linebreak
%
%
A weaker form of the alternative, generally easier to establish, 
asks for the existence of free {\em semigroups}  in $\Gamma$   instead of free subgroups,   provided that  $\Gamma$ is not virtually solvable\footnote{Notice that,  in this weaker form, the  Tits Alternative is no longer   a {\em dichotomy}  for linear groups, since it is well known that there exist solvable groups of $GL(n,\mathbb{R})$ which  also contain free semigroups (and actually, {\em any} finitely generated solvable group which is not virtually nilpotent contains a free semigroup on two generators  \cite{Ros74}).
	It remains a dichotomy for those classes of groups for which ``virtually solvable'' implies sub-exponential growth, e.g. hyperbolic groups, groups acting geometrically on CAT(0)-spaces etc.}. 
The original result has also been improved by quantifying the depth  in $\Gamma$ of the free subgroup with respect to some fixed generating set $S$ of $\Gamma$; that is the minimal $N$ such that $S^N$ contains a free subgroup or sub-semigroup.
The first   forms of quantification of the Tits Alternative  for Gromov hyperbolic groups were proved  by T. Delzant  \cite{Del96}, by M. Koubi  \cite{Kou98} (for Gromov hyperbolic groups with torsion) and by G. Arzantsheva and I. Lysenok \cite{AL06} (for subgroups of a given hyperbolic group), \linebreak for a constant $N$ depending however   on the group $\Gamma$ under consideration.   \linebreak
Very general quantifications 
for free semigroups in discrete groups of isometries of Gromov-hyperbolic spaces were recently proved  independently by  Breuillard-Fujiwara \cite{BF18}, and by   Besson-Courtois-Gallot and the second author \cite{BCGS17}.
Namely,  there exists a universal function $N(C)$ such that for  any finite, symmetric subset $S$ of isometries of a $\delta$-hyperbolic space $X$,   \linebreak either the group $G$ generated by $S$ is elementary or $S^N$ contains two elements $a,b$ which generate a free semigroup, {\em provided that}   $X$ satisfies a  packing condition at scale $\delta$ with constant $C$ (or also simply if the counting measure of $G$ satisfies a {\em doubling} condition at some scale $r$ with constant $C$).\linebreak
%
\noindent For discrete isometry groups of  pinched, negatively curved manifolds, 
S. Dey, \linebreak M. Kapovich and B. Liu \cite{DKL18} recently improved \cite{BCGS17}, proving a  quantitative, true Tits Alternative:  there exists  $N=N(k,d)$ such that for any couple of  isometries   $a,b$ of  a complete, simply connected,  $d$-dimensional Riemannian  manifold $X$ with  negative sectional curvature  
$-k^2 \leq K_X \leq -1$,    generating a  discrete
non-elementary  group,  with $a$ not elliptic,  one can find  an isometry $w$, which can be written as a word $w=w(a,b)$ on $\{a,b\}$ of length less than $N$, such that $\{a^N, w\}$ generate a non-abelian free subgroup.  Noticeably, the authors not only find a true free subgroup with quantification,   but they also specify that  one of the  generators of the  free group can be  prescribed   a-priori, provided it is chosen non-elliptic: a property that we will call  {\em specification property}.
\vspace{1mm}

The first result of the paper  is a  generalization of  Dey-Kapovich-Liu's result in the framework of  GCB Gromov-hyperbolic  spaces. While Gromov-hyperbolicity is the most natural metric replacement  for the classical hypothesis of negative curvature, the packing condition at some scale is the correct (weak) metric replacement of the lower curvature bound. In this setting  we prove the following:
\begin{theo}
	[Quantitative free subgroup theorem with specification]
	\label{maintheorem} ${}$ \\
	Let $P_0,r_0$ and $\delta $ be fixed positive constants. Then, there exists an integer $N(P_0,r_0,\delta)$, only depending  on $P_0,r_0, \delta$, satisfying the following properties. \linebreak
	Let $(X,\sigma)$ be any $\delta$-hyperbolic, \textup{GCB}-space which  is $P_0$-packed at scale $r_0$: 
	\begin{itemize}		
		\item[(i)]  for any couple of   $\sigma$-isometries $S=\{a,b\}$ of $X$,  where  $a$  is non elliptic, such that the group $\langle a,b \rangle$   is  discrete and  non-elementary, there exists a word $w(a,b)$ in $a,b$ of length $\leq N$  such that one of the semigroups $\langle a^N, w(a,b) \rangle^+$, $\langle a^{-N}, w(a,b) \rangle^+$    is free; 
		\item[(ii)] for any couple of   $\sigma$-isometries  $S=\{a,b\}$ of $X$ such that the group $\langle a,b \rangle$ is discrete, non-elementary and torsion-free,  there exists a word $w(a,b)$ in $a,b$ of length $\leq N$ such that the group $\langle a^N, w (a,b) \rangle$ is   free.
	\end{itemize}

{\em	
\noindent Here by {\em $\sigma$-isometries} we mean the natural isometries of $(X, \sigma)$,  that is those  preserving the geodesic bicombing. In case of Busemann convex or CAT$(0)$-spaces, this condition is trivially satisfied by all the isometries of the space.	
}	
\end{theo}

The first part of the theorem precises Proposition 5.18 in \cite{BCGS17} and Theorem 5.11 of \cite{BF18},  showing the  specification property, under the additional hypothesis of   a convex bicombing. The difficulty here is that there is no  a-priori bounded power $a^N$ such that $\ell(a^N)$ is greater than a specified constant (for instance, when $a$ is parabolic this is false for every $N$). \linebreak To avoid replacing $a$ with a bounded word in $\{ a, b\}$ (as in \cite{BCGS17}, \cite{BF18}) we follow the strategy of \cite{DKL18}, which however  requires a bit of convexity.\linebreak
For the second part of Theorem \ref{maintheorem}, it should   be  remarked  that the torsionless assumption is actually necessary\footnote{The authors are not able to understand the proof in \cite{DKL18} without  the torsionless assumption,  which   seems to be used at page 14, below Lemma 4.5.},  as shown by some examples of groups acting on simplicial trees (with bounded valency, hence packed) produced  in Proposition 12.2 of \cite{BF18}.   
Finally, notice that  in our setting  {\em elementary} is the same as   {\em virtually nilpotent}, because of the bounded packing assumption (see Section \ref{sub-elementary} for a proof of this fact; notice however that, without the packing assumption,  there exist elementary isometry groups of negatively curved manifolds which are even free non-abelian,  cp.\cite{Bow93}).
 \vspace{2mm}

The motivation for us  behind the quantification and the specification property in the Tits Alternative is  geometric.  
For instance, the specification property will be used to bound from below the {\em systole} of the action
in terms of the {\em upper nilradius} (see below Corollary \ref{cor-entropy} for the definition)
thus yielding another version of the classical Margulis' Lemma in our context. \linebreak
This will allow us to extend some classical  convergence results of Riemannian manifolds with curvature bounded below to our setting, see further below.
Some applications we will describe do not need the specification property or the quantitative Tits Alternative, but the bounded packing condition remains almost everywhere essential, in particular for Breuillard-Green-Tao's   generalized Margulis' Lemma.

We want   to stress  that the proof we give  of the above theorem heavily draws from techniques developed in \cite{DKL18} and \cite{BCGS17}. However some ingredients are new (for instance  the estimate  of the distance between different levels of the Margulis  domains, cp. Propositions \ref{lemma-lowerdistance}, \ref{Margulis-estimate} and  \ref{Margulis-nonexplicit},    the use of convexity for controlling the packing at arbitrarily small scales,   etc),
 as well as  all the applications to non-cocompact  groups, which we are now going  to discuss.
\vspace{2mm}

A first, direct  consequence of  a quantitative free  group or semigroup  theorem 
is  a uniform estimate from below of the  {\em algebraic entropy} of the groups under consideration and of the entropy of the spaces they act on.

\noindent Namely given any GCB-space $X$ that is $P_0$-packed at scale $r_0$, the {\em (covering) entropy of $X$} is defined as  
\vspace{-2mm}

$$ \textup{Ent}(X)  = \limsup_{R \rightarrow +\infty}  \frac{\log \text{Pack} \left( B(x,R), r_0  \right) }{R}  $$
where Pack$(B(x,R),r)$ denotes the packing number of the ball of radius $R$ centered at $x\in  X$, i.e. the maximal cardinality of a $2r_0$-net inside $B(x,R)$ (the definition   does not depend on the point $x$, by the triangular inequality).\linebreak 
For a   discussion on this notion of entropy and its properties see \cite{Cav21}. \linebreak For instance when $X$ is CAT$(0)$ then Ent$(X)$ equals the volume entropy of the natural measure of $X$ (as introduced in \cite{LN19}, see also \cite{CavS20}). 
\linebreak Recall that in case of non-positively curved Riemannian manifolds the natural measure is exactly the Riemannian volume. \\
On the other hand, given any group $\Gamma$ acting  discretely  on  $X$, one also defines the {\em entropy  of the action} as the exponential growth rate
\vspace{-2mm}

$$ \textup{Ent}(\Gamma, X)  = \limsup_{R \rightarrow +\infty}  \frac{\log \textup{card} \left( B(x,R) \cap \Gamma x \right) }{R}.  $$
When $X$ is the Cayley graph of $\Gamma$ with respect to some finite generating set $S$, \linebreak we will also write $  \textup{Ent}(\Gamma,  S)$. 
The algebraic entropy  of a finitely generated group, denoted  $\textup{EntAlg}(\Gamma)$,  is accordingly defined as the infimum, over all possible finite generating sets $S$, of the exponential growth rates Ent$(\Gamma,S)$.  
We record here two  estimates  for the entropy of the groups and the spaces under consideration,  which stem  directly from the quantification of existence of a free semigroup  (already proved in \cite{BF18}, Theorem 13.9, or   in \cite{BCGS17}, Proposition 5.18(i)), combined with the packing assumption:

\begin{cor}\label{cor-entropy}
	Let $(X,\sigma)$ be a $\delta$-hyperbolic 
	\textup{GCB}-space,   $P_0$-packed at scale $r_0>0$, admitting    a non-elementary, discrete group $\Gamma$ of   $\sigma$-isometries. Then:   
 

	\begin{itemize}
		\item[(i)] $\textup{EntAlg}(\Gamma)\geq C_0 $,
		\item[(ii)] $\textup{Ent}(X)\geq \textup{Ent}(\Gamma,X)\geq C_0 \cdot  \textup{nilrad}(\Gamma, X)^{-1}  $,
	\end{itemize}
	
	\noindent where  $C_0=C_0(P_0,r_0,\delta)>0$ is a constant  depending only on $P_0,r_0$ and $\delta$. 
	\vspace{1mm}	
	
	{\em 
		\noindent The invariant  $\textup{nilrad}(\Gamma, X)$   appearing here is the {\em nilradius} of the action of $\Gamma$ on $X$, defined as the infimum  over all $x\in X$  of the largest  radius $r$ such that the $r$-almost stabilizer $\Gamma_r(x)$ of $x$ is virtually nilpotent.
		By definition,   $\textup{nilrad}(\Gamma, X)$    is always bounded from below by Breuillard-Green-Tao's generalized Margulis constant $\varepsilon_0$, which can be expressed in terms of  $P_0,r_0$.\linebreak 
		On the other hand, the nilradius can be arbitrarily large if the orbits of $\Gamma$ are very sparse in $X$.  
		However, 
		if $\Gamma$ is non-elementary (which in our case means non-virtually nilpotent) it is always a finite number.}
\end{cor}

A second geometric consequence  of Theorem \ref{maintheorem}   is a lower bound of the {\em systole} of the action of $\Gamma$ on $X$, that is the smallest non-trivial displacement of points of $X$ under the action of the group (for a non positively curved manifold $X$ and a torsionless group of isometries  $\Gamma$, this is exactly twice the injectivity radius of the quotient manifold $\bar X = \Gamma \backslash X$).\linebreak
Namely,  let $X_{\varepsilon_0} \subseteq X$ be the subset of points which are displaced less than the generalized Margulis constant $\varepsilon_0$ by some nontrivial element of  $\Gamma$: \linebreak 
this is classically called the {\em $\varepsilon_0$-thin subset} of $X$.
We define the  {\em upper nilradius} of $\Gamma$, as opposite  to the nilradius,   as  the {\em supremum} over $x \in X_{\varepsilon_0}$ of  the largest  $r$ such that  $\Gamma_r(x)$ is virtually nilpotent. In other words the upper nilradius, denoted nilrad$^+(\Gamma, X)$,
measures how far we need to travel from any   $x$ of  $X_{\varepsilon_0}$ to  find two points $g_1x, g_2x$ of the orbit    such that the subgroup $\langle g_1,g_2 \rangle$ is  non-elementary.
A natural bound of the upper nilradius is given, for  cocompact actions, by the diameter of the quotient $\Gamma {\backslash} X$. \linebreak However the upper nilradius can well be finite even for non-cocompact actions, for instance when $\Gamma$ is a quasiconvex-cocompact group, or a subgroup of infinite index of a cocompact group of $X$ (see Examples \ref{ex-coco}, \ref{ex-sub}).\linebreak
The specification property in the quantitative Tits Alternative   yields a  lower bound of the systole of the action in terms of the geometric parameters $P_0,r_0,\delta$ and of an  upper bound of the upper nilradius:

\begin{cor}
	\label{cor-systole-intro}
	Let $(X,\sigma)$ be a $\delta$-hyperbolic,
	\textup{GCB}-space that is $P_0$-packed at scale $r_0>0$.
	Then, for any  torsionless, non-elementary, discrete group of $\sigma$-isometries  $\Gamma$ of $X$ it holds:
\vspace{-3mm}
	
	$$\textup{sys}(\Gamma, X) \geq \min \left\{ \varepsilon_0,  \, \frac{1}{H_0}e^{-H_0\cdot   \textup{nilrad}^+(\Gamma, X)} \right\} $$
	where $H_0=H_0(P_0,r_0,\delta)$ is a constant depending only on $P_0,r_0,\delta$ and where $\varepsilon_0=\varepsilon_0(P_0,r_0)$ is the   generalized Margulis constant introduced before. 
\end{cor}
	
\noindent	  If $\Gamma$ has torsion, a similar estimate holds for the {\em free systole}, cp. Section \ref{sub-notation} and formula (\ref{sys-uppernil}).
\vspace{1mm}
 
 In the restricted framework of   nonpositively curved manifolds,  this gives a new group-theoretic estimate of the length of the shortest closed geodesic in  $\bar X = \Gamma \backslash X$, without any lower curvature bound assumption.

\noindent  Remark that without any bound of the upper nilradius of  $\Gamma$ there is no hope of estimating sys$(\Gamma, X)$ from below in terms of $\delta$ and the packing constants.
This is clear for  groups acting with parabolics (cp. Example \ref{ex-parabolics}), but it also fails for groups  without parabolics. It is enough to consider compact hyperbolic manifolds $\bar X = \Gamma \backslash X$ possessing very small periodic geodesics $\gamma$ of length $\varepsilon$, much smaller than the Margulis  constant: by the classical theory of Kleinian groups, $\gamma$ has a very long tubular neighbourhood and,  consequently, $\Gamma$ has arbitrarily large upper nilradius.
\vspace{1mm}

Even without any a-priori bound of the upper nilradius of the action of $\Gamma$, \linebreak one can always find a point $x$ where the minimal displacement is bounded below by a universal  function of the geometric parameters $P_0,r_0,\delta$ of $X$. \linebreak
We call    {\em diastole} of $\Gamma$ acting on $X$, denoted  $\textup{dias} (\Gamma, X)$,   the supremum over all $x\in X$ of the minimal displacement of $x$ under all  non trivial elements of the group. 
The next   result   generalizes one of the classical versions of the Margulis Lemma on Riemannian manifolds with pinched, negatively curvature:

\begin{cor}\label{cor-diastole} Let $(X,\sigma)$ be a $\delta$-hyperbolic,
	\textup{GCB}-space that is $P_0$-packed at scale $r_0>0$.
	Then, for any torsionless, discrete, non-elementary group of $\sigma$-isometries $\Gamma$  of $X$ we have:
	
 \vspace{-3mm}	
	$$\textup{dias}(\Gamma, X) =  \sup_{x\in X} \inf_{g \in \Gamma^\ast} d(x,gx) \geq  \varepsilon_0$$
	(where $\varepsilon_0=\varepsilon_0(P_0,r_0)$ is the   generalized Margulis constant).
\end{cor}

\noindent A version of this estimate for groups $\Gamma$ with torsion is proved in Section \ref{sub-diastole}.

\noindent Notice that the estimate, which holds also for cocompact groups,  does not depend on the diameter (in contrast with Proposition 5.25 of  \cite{BCGS17}; also notice that our groups do not belong  to any of the classes considered in \cite{BCGS17}, as they do not have a-priori a cocompact action on a convex, Gromov-hyperbolic space or an action on a Gromov-hyperbolic space with asymptotic displacement uniformly bounded below).
 \vspace{2mm}

\vspace{1mm}
In Section \ref{sec-compactness} we use these estimates to investigate  limits of group actions on our class of spaces.  Given a sequence of  faithful, isometric group actions $\Gamma_n  \curvearrowright (X_n,x_n)$   on general, pointed metric spaces and   a non-principal ultrafilter $\omega$, one can define a limit group $\Gamma_\omega$ by taking limits of any possible admissible sequence $(g_n)$, for $g_n\in \Gamma_n$ for every $n$; then, the limit group  (which heavily depends on the choice of base points $x_n$ in $X_n$)  naturally acts  by isometries on the ultralimit space $(X_\omega, x_\omega)$.  
When restricted to the class of  closed isometry groups of proper metric spaces,  this convergence,  is easily seen to be equivalent to the {\em pointed, equivariant Gromov-Hausdorff  convergence} \linebreak of some subsequence  $ \Gamma_{n_k}  \curvearrowright  (X_{n_k},x_{n_k})$, as defined by Fukaya    (see \cite{Fuk86}, \cite{FY92}  and \cite{Cav21survey} for comparison with other notions of convergence of groups). \linebreak
Since the class of $\delta$-hyperbolic, GCB-spaces is closed under ultralimits,  it is natural to ask for the properties of the limit actions on these spaces. \linebreak
According to the  behaviour of the distance of the base points $x_n$ from the thin subsets of the $X_n$'s, the  type of the limit $g_\omega$ of a sequence of hyperbolic or parabolic isometries $(g_n)$ can change  (see Proposition \ref{ultralimit-isometry} and Example \ref{exip-para}), \linebreak elliptic elements can appear and the limit action may be  non-discrete. \\ 
The following result resumes the possibilities for the limit action:
\begin{theo}[Extract from Theorems \ref{Dicotomia}\& \ref{elementarita} and Corollary \ref{Dicotomia-geometrica}] ${}$ \label{intro-ultra} \\
	Let $(X_n,x_n,\sigma_n)$ be  $\delta$-hyperbolic,  \textup{GCB}-spaces,  $P_0$-packed at scale $r_0>0$, and let $\Gamma_n$ be  torsion-free, discrete $\sigma_n$-isometry groups of the spaces $X_n$. \linebreak Let $\omega$ be a non-principal ultrafilter and let  $\Gamma_\omega$ be the ultralimit group of \linebreak $\sigma_\omega$-isometries of  the $\delta$-hyperbolic,  \textup{GCB}-space $(X_\omega,\sigma_\omega)$. Then:
	\begin{itemize}
		\item[(i)] $\Gamma_\omega$ is either discrete and torsion-free, or elementary;
		\item[(ii)] $\Gamma_\omega$ is not elementary if and only if there exist admissible sequences $(g_n), (h_n)$ such that the group $\langle g_n, h_n \rangle$ is not elementary for $\omega$-a.e.$(n)$.
	\end{itemize}
\noindent 	Moreover if the groups $\Gamma_n$ are non-elementary one can always choose the base points $x_n$ in $X_n$  so that the limit group is discrete and torsion-free.  \linebreak Finally when $\Gamma_\omega$ is discrete and torsion-free then the ultralimit of the  quotients $ \Gamma_n \backslash (X_n, x_n)$ is isometric to the quotient  space $\Gamma_\omega \backslash X_\omega$. 
\end{theo}

\noindent Notice that the two conditions in (i) are not mutually exclusive: the group $\Gamma_\omega$ can be discrete, torsion-free \emph{and} elementary. But, if it is not discrete, then it is necessarily elementary. We will see that the existence of points $x_n$ which make the limit action discrete is a direct consequence of the diastolic estimate given in Corollary \ref{cor-diastole}. The condition of non-elementarity of the $\Gamma_n$ is necessary, essentially  to rule out actions by $\mathbb{Z}$ on the real line with  smaller and smaller systole.
\vspace{2mm}

  We conclude with some compactness results which are consequence of  Theorem \ref{intro-ultra} and of the closure of our class under ultralimits.
For any given parameters  $P_0,r_0,\delta$ and $D> 0$ consider the class  
\vspace{-3mm}

$$\textup{GCB}_c(P_0,r_0,\delta; D)$$  
of {\em compact quotients}  of   $\delta$-hyperbolic, GCB-spaces $(X,\sigma)$ which are $P_0$-packed at scale $r_0$,  by some torsionless, non-elementary, discrete group $\Gamma$  of $\sigma$-isometries,  with codiameter diam$(\Gamma \backslash X) \leq D$. This class is compact with respect to the Gromov-Hausdorff topology, and contains only finitely many homotopy types, as a consequence of results proved\footnote{In \cite{BG20} and \cite{BCGS} the authors only consider compact quotients  of {\em convex}, \linebreak $\delta$-hyperbolic spaces, under  the weaker assumption of bounded entropy; however,  the compactness of the class GCB$(P_0,r_0,\delta; D)$  follows from the same arguments and the contractibility of convex, geodesically bicombed spaces.} in \cite{BG20} and \cite{BCGS}.
Dropping the assumption on the diameter, thus allowing non-compact spaces in our class, we cannot hope for any finiteness result. However 
 in section \ref{sec-compactness}   we will prove:
 \vspace{-2mm}
 
 \begin{theo}
	\label{compactness-abetc}
The class of abelian (resp.   $m$-step nilpotent, $m$-step solvable)  \linebreak coverings of spaces in the class \textup{GCB}$_c(P_0, r, \delta; D)$ is closed under ultralimits, \linebreak hence  compact w.r. to pointed   equivariant Gromov-Hausdorff convergence. \\
\noindent Namely, for any sequence $(\hat X_n, \hat x_n) \rightarrow (\bar X_n, \bar x_n) $  of normal coverings of pointed spaces   belonging to  \textup{GCB}$_c(P_0, r, \delta; D)$ whose   group   of deck transformations $\Gamma_n$ \linebreak is abelian   (resp. $m$-step nilpotent, $m$-step solvable), 
 there exists  a pointed space  $(\bar X, \bar x) $ in \textup{GCB}$_c(P_0, r, \delta; D)$  and a normal  covering 
  $(\hat X, \hat x) \rightarrow (\bar X, \bar x)$   with   abelian   (resp. $m$-step nilpotent, $m$-step solvable) group of deck transformations $\Gamma$,  such that the actions $\Gamma_n \curvearrowright   (\hat X_n, \hat x_n)$ tend  to $\Gamma  \curvearrowright   (\hat X, \hat x)$ in  pointed, equivariant Gromov-Hausdorff convergence.

\vspace{1mm} \noindent {\em Recall that a group is $m$-step nilpotent (resp. $m$-step solvable) if its lower central series (resp. its derived series)   has length at most $m$.}
\end{theo}

\noindent The next theorem generalizes the above result and extends the compactness  of the class \textup{GCB}$_c(P_0, r, \delta; D)$, by replacing the diameter  with   the upper nilradius.
  Let  $\textup{GCB}(P_0,r_0,\delta; \Delta)$ 
be the class  of quotients of  $\delta$-hyperbolic, GCB-spaces $(X,\sigma)$ which are $P_0$-packed at scale $r_0$,  by a  torsionless, non-elementary, discrete group of $\sigma$-isometries   with upper nilradius $\leq \Delta$.
Then:
\begin{theo}
	\label{compactness-nilpotence}
	${}$ \hspace{-2mm}The class \textup{GCB}$ (P_0,r_0,\delta, \Delta)$ is closed under ultralimits   and  compact with respect to pointed,  Gromov-Hausdorff convergence. 
\end{theo}
 
 \vspace{-2mm}
\noindent Remark that, as the condition of being CAT$(0)$ is stable under ultralimits, then the compactness Theorems \ref{compactness-abetc} \& \ref{compactness-nilpotence} also hold in this restricted  class.
 \vspace{2mm}
  
  \small
\noindent {\sc Acknowledgments.} {\em The  authors are grateful  to   G. Besson, G. Courtois and S. Gallot for many helpful discussions.}
\normalsize

%
%
%
%
%
%
%
%
%

\subsection{Notation} \label{sub-notation}
Given a  subset $A$ of a metric space $X$ and a real number $r \geq 0$, we denote by $B(A,r)$ and $\overline{B}(A,r)$ the open and closed {\em $r$-neighbourhood of $A$}, respectively. In particular if $A = \lbrace x \rbrace$, where $x\in X$ is a point, then $B(x,r)$ and  $\overline{B}(x,r)$ denote  the open and closed ball of radius $r$ and center $x$.

\noindent A {\em geodesic segment} 
is a curve $\gamma\colon [a,b] \to X$ such that $d(\gamma(t),\gamma(t')) = \vert t - t'\vert$ for all $t,t'\in [a,b]$;
by abuse of notation, a geodesic joining $x$ to $y$ is denoted by $[x,y]$ even if it is not unique.
A {\em geodesic ray} is an isometric embedding from $[0,+\infty)$ to $X$, while a {\em geodesic line} (or, simply, a {\em geodesic}) is an isometric embedding from $\mathbb{R}$ to $X$.
The space $X$ is called {\em geodesic} if for all $x,y \in X$ there exists a geodesic segment joining $x$ to $y$. 

\noindent If $X$ is any proper metric space we denote by Isom$(X)$ its group of isometries, endowed with the uniform convergence on compact subsets of $X$. A subgroup $\Gamma$ of Isom$(X)$ is called {\em discrete} if the following (equivalent) conditions (see \cite{BCGS17}) hold:
\begin{itemize}
	\item[(a)] $\Gamma$ is discrete as a subspace of Isom$(X)$;
	\item[(b)] $\forall x\in X$ and $R\geq 0$ the set $\Sigma_R(x) = \lbrace g \in \Gamma  \hspace{1mm} | \hspace{1mm}  g  x\in \overline{B}(x,R)\rbrace$ is finite.
\end{itemize}  
We will denote by $\Gamma^\ast$  the subset of nontrivial elements of $\Gamma$, while $\Gamma^\diamond \subset \Gamma$   will  denote the subset of   elements  with finite order. Moreover, if $S$ is a symmetric,  finite generating set for $\Gamma$ we will denote by  $S^n$ the ball of radius $n$ of $\Gamma$ with respect to the word metric associated to $S$. 
\vspace{1mm}

We register here also, for  easy reference, a number of  invariants associated to the action of $\Gamma$ on $X$ we will be interested in throughout the paper.\linebreak
For every $r > 0$ and every point $x\in X$  the {\em $r$-almost stabilizer} of $x$  in $\Gamma$  is the subgroup
$$\Gamma_r(x) = \langle \Sigma_r(x) \rangle.$$
The {\em $r$-thin subset} of $X$   (with respect to the action of $\Gamma$) is the subset 
$$ X_r = \{ x \in X   \hspace{1mm} | \hspace{1mm} \exists  g \in \Gamma^\ast \textit{ s.t. } d(x,gx) < r \}$$
and the  {\em free $r$-thin subset} $X_r^\diamond$ is obtained by replacing  $\Gamma^\ast$ in the definition above  with $\Gamma \setminus\Gamma^\diamond$.
Some numerical invariants associated to the action of $\Gamma$ we  will be interested in are:
\begin{itemize}
	\item the {\em minimal displacement}  of  $g \in \Gamma$, defined as  $\ell(g) = \inf_{x\in X}d(x, gx)$;
	\item  the {\em asymptotic displacement} of $g \in \Gamma$, defined as $\Vert g \Vert = \lim_{n\to +\infty}\frac{d(x,g^n x)}{n}$;
	
	\item the {\em minimal displacement} of $\Gamma$ at $x$, defined as  sys$(\Gamma,x) \!=\! \inf_{g\in \Gamma^\ast}d(x,gx)$;
	\item the {\em minimal free displacement} of $\Gamma$ at x, defined as   \vspace{-2mm}
	$$\text{sys}^\diamond  (\Gamma,x) = \inf_{g \in \Gamma\setminus  \Gamma^\diamond} d(x,gx);$$ 
	
	\item the {\em nilpotence radius} of  $\Gamma$ at $x$, defined as \vspace{-2mm}
	$$\textup{nilrad} (\Gamma,x)= \sup\lbrace r\geq 0 \hspace{1mm} \text{ s.t. } \hspace{1mm}  \Gamma_r(x) \text{ is virtually nilpotent}\rbrace ;$$ 
\end{itemize}
and their corresponding global versions:
\begin{itemize}	
	\item  the {\em systole} and the  {\em free systole}, defined respectively as : 
	$$\textup{sys}( \Gamma, X) = \inf_{x\in X}\text{sys}(\Gamma,x),  \hspace{8mm} 
	\textup{sys}^\diamond ( \Gamma, X) = \inf_{x\in X}\text{sys}^\diamond (\Gamma,x)$$
	\item  the {\em diastole} and the  {\em free diastole}, defined  as :  
	$$\textup{dias}( \Gamma, X)= \sup_{x\in X}\text{sys}(\Gamma,x),  \hspace{8mm}  
	\textup{dias}^\diamond ( \Gamma, X) = \sup_{x\in X}\text{sys}^\diamond (\Gamma,x)$$
	\item the {\em nilradius}: nilrad$( \Gamma, X)= \inf_{x\in X}\text{nilrad}(\Gamma,x)$.
\end{itemize}
Moreover, for convex, packed metric spaces  we can define an upper analogue of the nilradius (like the diastole for the systole):
\begin{itemize}	
	\item the {\em upper nilradius}: nilrad$^+( \Gamma, X)= \sup_{x\in X_{\varepsilon_0}}\text{nilrad}(\Gamma,x)$
\end{itemize}
\noindent where the constant $ \varepsilon_0$ appearing in the definition is the {\em generalized Margulis constant}, which will be introduced  in see Section \ref{sub-margulis}.

\section{Packed, GCB-spaces}
 
In the first part of this section we introduce  the notion of geodesically complete, convex geodesic bicombings. 
In the second part we define the packing condition on a metric space and we see how it interacts with the convexity of a geodesic bicombing.  
As a consequence we can apply a result of \cite{BGT11} to obtain a uniform estimate of the nilradius of groups acting on packed metric spaces, introducing  Breuillard-Green-Tao's {\em generalized Margulis constant}, which plays a key-role in the proof of the Free Subgroup Theorem and its applications.

\enlargethispage{1\baselineskip}
\subsection{Convex geodesic bicombings}\label{sub-bicombing}
A geodesic bicombing on a metric space $X$ is a map $\sigma \colon X\times X\times [0,1] \to X$ with the property that for all $(x,y) \in X\times X$ the restriction $\sigma_{xy}\colon t\mapsto \sigma(x,y,t)$ is a geodesic from $x$ to $y$, i.e. $d(\sigma_{xy}(t), \sigma_{xy}(t')) = \vert t - t' \vert d(x,y)$ for all $t,t'\in [0,1]$ and $\sigma_{xy}(0)=x, \sigma_{xy}(1)=y$.
A geodesic bicombing is:
\begin{itemize}
	\item \emph{convex} if the map $t\mapsto d(\sigma_{xy}(t), \sigma_{x'y'}(t))$ is convex on $[0,1]$ for all $x,y,x',y' \in X$;
	\item \emph{consistent} if for all $x,y \in X$, for all $0\leq s\leq t \leq 1$ and for all $\lambda\in [0,1]$ it holds $\sigma_{pq}(\lambda) = \sigma_{xy}((1-\lambda)s + \lambda t)$, where $p:= \sigma_{xy}(s)$ and $q:=\sigma_{xy}(t)$;
	\item \emph{reversible} if $\sigma_{xy}(t) = \sigma_{yx}(1-t)$ for all $t\in [0,1]$.
\end{itemize}

\noindent For instance, all convex   spaces in the sense of Busemann (hence, all CAT$(0)$-space) have a unique convex, consistent, reversible geodesic bicombing.

Given a geodesic bicombing $\sigma$ we say that a geodesic (segment, ray, line) $\gamma$ \linebreak 
is a $\sigma$-geodesic (segment, ray, line) if for all $x,y\in \gamma$ we have that $\sigma_{xy}$ is a parametrization of the subsegment of $\gamma$ between $x$ and $y$. 
We say that a geodesic bicombing is \emph{locally geodesically complete} if every $\sigma$-geodesic segment is contained in a $\sigma$-geodesic segment whose endpoints are both different from $x$ and $y$. We say that it is \emph{geodesically complete} if every $\sigma$-geodesic segment is contained in a $\sigma$-geodesic line.
A convex metric space in the sense of Busemann is (locally) geodesically complete if and only if its unique convex, consistent, reversible geodesic bicombing is (locally) geodesically complete. We report here some trivial facts:
\begin{lemma}
	\label{bicombing-basics}
	Let $\sigma$ be a convex, consistent, reversible geodesic bicombing of a metric space $X$. Then:	
	\begin{itemize}

		\item[(i)] the map $\sigma$ is continuous;
		\item[(ii)] $X$ is geodesic and contractible;
		\item[(iii)] if $X$ is complete and $\sigma$ is locally geodesically complete then $\sigma$ is    geode\-sically complete.
	\end{itemize}
\end{lemma}
\begin{proof}
	Assertion (ii) is a consequence of (i),  and (i) follows directly from the convexity condition. Assertion (iii) is standard.
\end{proof}
A geodesically complete, convex, consistent, reversible geodesic bicombing will be called  GCB for short.
A couple $(X,\sigma)$, where $\sigma$ is a GCB on a complete metric space $X$   will be called a {\em GCB-space.}
The interest in this weak notion of non-positive curvature is given by its stability under limits (while the Busemann convexity condition is not stable in general):
\begin{lemma}
	\label{bicombing-limit}
	The class of pointed \textup{GCB}-spaces is closed under ultralimits.
\end{lemma}
\begin{proof}
	Let $(X_n,x_n,\sigma_n)$ be pointed GCB-spaces and let $\omega$ be a non-principal ultrafilter. We define the map $\sigma_\omega \colon X_\omega \times X_\omega \times [0,1] \to X_\omega$ by $$\sigma_\omega(\omega\text{-}\lim y_n, \omega\text{-}\lim z_n, t) = \omega\text{-}\lim \sigma_n(y_n,z_n,t).$$
	It is easy to see it is well defined and it satisfies all the required properties.
\end{proof}

A subset $C$ of a GCB-space $(X,\sigma)$ is {\em $\sigma$-convex} if for all $x,y\in C$ the $\sigma$-geodesic segment $[x,y]$ is contained in $C$. We say that a map $f\colon X \to \mathbb{R}$ is $\sigma$-convex (resp. strictly $\sigma$-convex) if it is convex (resp. strictly convex) when restricted to every $\sigma$-geodesic.

 \enlargethispage{2\baselineskip}
 
\begin{lemma}
	\label{bicombing-projection}
	Let $(X,\sigma)$ be a proper \textup{GCB}-space and let $C \subset X$ be a  $\sigma$-convex subset. Then: 
	\begin{itemize}
		\item[(i)] for all $x_0\in X$ the distance map $d_{x_0}$  from $x_0$   is $\sigma$-convex;
		\item[(ii)] the distance map $d_C$  from $C$    is $\sigma$-convex;
		\item[(iii)] if $C$ is closed then   $\forall x_0\in X$ there exists a unique projection of $x_0$ on $C$.
	\end{itemize}
\end{lemma}

\begin{proof}
Claim (i) follows directly by the definition of $\sigma$-convexity.  
\noindent To show (ii), 
	let $x_0,x_1\in X$ and for all   $  \varepsilon>0$ pick points $y_0,y_1\in C$   such that $d(x_0,y_0)\leq d_C(x_0) + \varepsilon$ and  $d(x_1,y_1)\leq d_C(x_1) + \varepsilon$. For all $t\in [0,1]$ we denote by $x_t$ the point $\sigma_{x_0x_1}(t)$ and by $y_t$ the point $\sigma_{y_0y_1}(t)$. Observe that $y_t\in C$ by $\sigma$-convexity of $C$, so
$$d_C(x_t)\leq d(x_t,y_t)\leq td(x_0,y_0) +(1-t)d(x_1,y_1) \leq td_C(x_0)+ (1-t)d_C(x_1) + 2\varepsilon.$$
	By the arbitrariness of $\varepsilon$ we get (ii). 
	Finally,  the existence of the projection follows from the fact that $C$ is closed and $X$ is proper; moreover the map $x\mapsto d(x,x_0)^2$ is strictly $\sigma$-convex,  which  implies  uniqueness.
\end{proof}

  The {\em radius} of a bounded subset $Y\subseteq X$, denoted by $r_Y$,  is the infimum of the positive numbers $r$ such that $Y\subseteq B(x,r)$ for some $x\in X$. The following fact  is well-known for CAT$(0)$-spaces and will be used later to characterize elliptic isometries:

\begin{lemma}\label{lemma-center}
	Let $(X,\sigma)$ be a proper \textup{GCB}-space. Then for any bounded subset $Y$ of $X$ there exists a unique point $x\in X$ such that $Y\subseteq \overline{B}(x,r_Y)$. Such a point is called the {\em center} of $Y$.
\end{lemma}

\begin{proof}
	The existence of such a point is easy: just   take a sequence of points $x_n$ almost realizing the infimum in the definition of $r_Y$, i.e. $B(x_n, r_Y + \frac{1}{n}) \supseteq Y$. Since $Y$ is bounded, then the sequence $x_n$ is bounded. We may suppose, up to taking a subsequence, that the sequence $x_n$ converges to a point $x_\infty$. \linebreak 
	We then have 
	$d(x_\infty, y)=\lim_{n\to + \infty}d(x_n,y) \leq 
	r_Y$ for any $y \in Y$.\\
	Assume now that there exist two   points $x \neq x'$ satisfying the thesis, that is $Y\subseteq \overline{B}(x,r_Y)\cap \overline{B}(x',r_Y)$. We take the midpoint $m$ between $x$ and $x'$ along the $\sigma$-geodesic connecting them and we claim that there exists $\varepsilon > 0$ such that  $\overline{B}(m,r_Y- \varepsilon) \supseteq \overline{B}(x,r_Y)\cap \overline{B}(x',r_Y).$
	Otherwise there exist a sequence of points $x_n \in \overline{B}(x,r_Y)   \cap  \overline{B}(x',r_Y)$ with $x_n \notin \overline{B}(m,r_Y  -   \frac{1}{n})$,  that is  $d(x_n,x) \leq r_Y$, $ d(x_n,x') \leq r_Y$ and $d(x_n,m)> r_Y - \frac{1}{n}$. 
	Then again we may assume that the $x_n$'s converge to a point $x_\infty$ satisfying
\vspace{-3mm}	
	
	$$d(x_\infty,x) \leq r_Y, \quad d(x_\infty,x') \leq r_Y, \quad d(x_\infty,m)\geq r_Y.$$
	By Lemma \ref{bicombing-projection}.(i)   the distance from $x_\infty$ to any point of the $\sigma$-geodesic $[x,x']$ is constant; this is impossible as  the projection of $x_\infty$ on $[x,x']$ is unique by Lemma \ref{bicombing-projection}.(iii).
	We have therefore proved that there exists $\varepsilon > 0$ such that
\vspace{-3mm}
	
	$$\overline{B}(m,r_Y- \varepsilon) \supseteq \overline{B}(x,r_Y)\cap \overline{B}(x',r_Y).$$
	which contradicts with the definition of $r_Y$. Then $x= x'$.
\end{proof}

Let $(X,\sigma)$ be a GCB-space, $x \in X$ and $0<r\leq R$.  The {\em contraction map} 
	\begin{equation}
		\label{contraction-map}
		\varphi^R_r\colon \overline{B}(x,R) \to \overline{B}(x,r)
	\end{equation}
	is the map   sending any $y\in \overline{B}(x,R)$   to $\sigma_{xy}(r/R)$. This map is   $\frac{r}{R}$-Lipschitz, by the convexity of $\sigma$.
	Since $\sigma$ is geodesically complete,  $\varphi^R_r$ is also surjective:  given $y\in B(x,r)$ we can extend the geodesic segment $\sigma_{xy}$ to a $\sigma$-geodesic segment $\sigma_{xy'}$ with $d(x,y')=\frac{R}{r}d(x,y)$, so $y'\in B(x,R)$ and $\varphi^R_r(y')=y$.

\subsection{The packing condition and some examples}\label{sub-packing}
Recall that a subset $S$ of a metric space $X$ is called  {\em $2r_0$-separated} if for any two different $s,s'\in S$ it holds $d(s,s')>2r_0$, while $S$ is {\em $r_0$-dense} if for any $x\in X$ there exists $s\in S$ such that $d(x,s)\leq r_0$.
Accordingly, given a metric space $X$ we define the {\em packing} and the {\em covering function  of  a subset $Y \subseteq X$ at   scale $r_0 > 0$}, respectively   as:
\vspace{-5mm}	

$$\text{Pack}(Y,r_0) = \text{cardinality of a maximal }2r_0\text{-separated subset of }Y,$$
\vspace{-6mm}	
$$\text{Cov}(Y,r_0) = \text{cardinality of a minimal }r_0\text{-dense subset of }Y.$$
These numbers are related by the following, classical inequalities
\begin{equation}
	\label{pack-cov}
	\text{Pack}(Y,2r_0) \leq \text{Cov}(Y,2r_0) \leq \text{Pack}(Y,r_0).
\end{equation}

\noindent Then for every $0<r\leq R$ we define the {\em packing function of $X$} as
$$\text{Pack}(R,r) = \sup_{x\in X} \text{Pack}(\overline{B}(x,R), r).$$
We  say {\em $X$ is $P_0$-packed at scale $r_0$} if Pack$(\overline{B}(x,3r_0),r_0)\leq P_0$   $ \forall x\in X$.
\vspace{4mm}

The next result states that in a GCB-space the packing function $\textup{Pack}(R,r)$ is well controlled by the packing condition at any fixed scale $r_0$:

\begin{prop}[Theorem 4.2 and Lemma 4.5 of \cite{CavS20}]
	\label{packingsmallscales} ${}$ \\
	Let $(X,\sigma)$ be a \textup{GCB}-space that is $P_0$-packed at scale $r_0$. Then:
	\begin{itemize}
		\item[(i)] $X$ is proper and for all $r\leq r_0$ it is $P_0$-packed at scale $r$;	
		\item[(ii)]  for every $0<r\leq R$ it holds:					
		$$\textup{Pack}(R,r)\leq P_0(1+P_0)^{\frac{R}{r} - 1} \text{, if } r\leq r_0;$$	
		$$\textup{Pack}(R,r)\leq P_0(1+P_0)^{\frac{R}{r_0} - 1} \text{, if } r > r_0.$$
	\end{itemize}  	
\end{prop}

\vspace{-1mm}
\noindent (Notice that, in \cite{CavS20}, the above estimates are proved with an additional multiplicative factor $2$. This is due to the fact that, in \cite{CavS20}, the contraction maps in the spaces under consideration are $\frac{2r}{R}$-Lipschitz, while here these maps are precisely $\frac{r}{R}$-Lipschitz.)
\vspace{3mm}

\noindent Some basic examples of 
GCB-spaces that are $P_0$-packed at scale $r_0$ are:

\begin{itemize}

	\item[i)] complete,  simply connected Riemannian manifolds $X$ with pinched, nonpositive sectional curvature  $\kappa' \leq  K(X) \leq \kappa \leq 0$;
	\item[ii)] complete,  geodesically complete,  CAT$(\kappa)$-spaces $X$ with $\kappa \leq 0$ of dimension at most $n$ and volume of balls of radius $1$ at most $V_0$;  
	\item[iii)] simply connected $M^\kappa$-complexes of {\em bounded geometry}  
	without free faces, for $\kappa \leq 0$.    	 
	
\end{itemize} 

\noindent By {\em bounded geometry}  we mean  that the $M^\kappa$-complex $X$ has positive injectivity radius, valency  $\leq V_0$  and size   $\leq S_0$  (the {\em size}  of $X$ is by definition the largest $R$ such that every simplex of $X$ is contained in a ball of radius $R$ and contains a ball of radius $\frac{1}{R}$).

\noindent The second  and the third  class  above are   discussed in detail in \cite{CavS20}.  
Very special cases of (ii) and (iii) are geodesically complete  metric trees and CAT$(0)$-cube complexes 
with bounded valency (i.e. such that the number of edges and cubes, respectively, having   a common vertex   is uniformly bounded by some constant  $V_0$), see Proposition \ref{prop-cubes} below; for instance, all those admitting a cocompact group of isometries have bounded valency. \\
The  examples above  are  all  clearly Gromov-hyperbolic when $\kappa<0$. 
But they can manifest  a Gromov-hyperbolicity behaviour even without such a sharp bound; for instance it is classical that any proper cocompact CAT$(0)$-space without $2$-flats  is Gromov-hyperbolic (see \cite{Gro87}, \cite{BH09}).
\vspace{1mm}

For cube complexes with bounded valency we have the following hyperbolicity criterion. Recall that
one says that a cube complex $X$  {\em has $L$-thin rectangles} if  all Euclidean rectangles $[0,a]\times[0,b]$ isometrically embedded in $X$ satisfy $\min\lbrace a,b \rbrace\leq L$. Then:

\begin{prop}\label{prop-cubes}
	Let $X$ be a simply connected cube complex with no free faces, dimension   $\leq n$  and  valency  $\leq V_0$,   endowed with its canonical  $\ell^2$-metric (the length metric which makes any $d$-cube of $X$   isometric to $[-1, 1]^d$):  
	\begin{itemize}		
		\item[(i)]  if $X$ has 	 positive injectivity radius, then it  is a complete, geodesically complete, \textup{CAT}$(0)$-space which is  $P_0$-packed at scale $r_0=\frac{1}{3}$, for a packing constant $P_0 = P_0(V_0)$; 
		\item[(ii)] if moreover $X$ has  $L$-thin rectangles, then it is Gromov-hyperbolic with hyperbolicty constant $\delta = 4\cdot \textup{Ram} \lceil L  +1  \rceil$ (where $\textup{Ram}(m)$ denotes the Ramsey number of $m$).
	\end{itemize}	
\end{prop}

\begin{proof}
	The barycenter subdivision of the cube-complex gives a $M^0$-complex structure to $X$ with valency bounded uniformly in terms of $V_0$ and $n$ and, clearly, with uniformly bounded size. Moreover since $X$ has positive injectivity radius and it has no free fraces then the same is true for the metric induced by the complex structure (which is isometric), therefore we can apply Proposition 7.13 of \cite{CavS20} to conclude (i) (the fact that $X$ is globally CAT$(0)$ follows from the simply connectedness assumption). \\
	The proof of (ii) is presented in Theorem 3.3 of \cite{Gen16}.
\end{proof}

\vspace{1mm}
Theorem  \ref{maintheorem} and its consequences will apply to all these cases. In particular notice that the   quantitative Tits Alternative  with specification is new even for hyperbolic, CAT$(0)$ cube complexes (cp. \cite{SW05}, \cite{GJN20}).

\subsection{Margulis constant}\label{sub-margulis}
Consider a proper metric space $X$ and a discrete group of isometries $\Gamma$ of $X$. 
In general the nilradius of the action  can well be zero. 
However under the packing assumption it is always bounded  away from zero:
\begin{theo}[Corollary 11.17 of \cite{BGT11}] \label{theo-bgt}${}$ \\
	Let $X$ be a proper metric space such that $\textup{Cov}(\overline{B}(x,4),1)\leq C_0$ for all $x \in X$. \linebreak 
	Then there exists a  constant $ \varepsilon_M= \varepsilon_M(C_0)>0$, only depending  on $C_0$, such that for every discrete group of isometries $\Gamma$ of $X$ one has 
	\textup{nilrad}$(\Gamma, X) \geq   \varepsilon_M$.
\end{theo}

\noindent In analogy with  the case of Riemannian manifolds, the constant  $ \varepsilon_M$ is called the {\em (generalized) Margulis constant} of $X$. Combining with Proposition \ref{packingsmallscales}, we immediately get a  Margulis contant  for the class of complete GCB-spaces which are $P_0$-packed at some scale $r_0$:

\begin{cor}
	\label{margulis-constant}
	Given $P_0, r_0>0$ there exists $\varepsilon_0 = \varepsilon_0(P_0,r_0) > 0$ such that  
	for any \textup{GCB}-space   $X$ which is $P_0$-packed at scale $r_0$ and for any discrete group of isometries $\Gamma$ of $X$ one has \textup{nilrad}$(\Gamma, X) \geq   \varepsilon_0$.
\end{cor}
 
\begin{proof}
	By Proposition \ref{packingsmallscales} $X$ is proper.
	We rescale the metric by a factor $\frac{1}{2r_0}$. As the packing property is invariant under rescaling, we have Pack$(\frac{3}{2},\frac{1}{2})\leq P_0$ \linebreak by assumption. Hence Cov$(4,1)\leq$ Pack$(4,\frac{1}{2})\leq P_0(1+P_0)^7$, as follows from \eqref{pack-cov} and from Proposition \ref{packingsmallscales}. Applying  Theorem \ref{theo-bgt} we find a Margulis constant $\varepsilon_M$ for the space  $\frac{1}{2r_0}X$,  only depending  on $P_0$;  then, the constant $\varepsilon_0 (P_0, r_0)= 2r_0 \cdot \varepsilon_M$ satisfies the thesis.
\end{proof}

\section{Basics of Gromov hyperbolic spaces}

Let $X$ be a geodesic space. Given  three points  $x,y,z \in X$,  the {\em Gromov product} of $y$ and $z$ with respect to $x$  is defined as
\vspace{-3mm}

$$(y,z)_x = \frac{1}{2}\big( d(x,y) + d(x,z) - d(y,z) \big).$$

\noindent The space $X$ is said {\em $\delta$-hyperbolic} if   for every four points $x,y,z,w \in X$   the following {\em 4-points condition} hold:
\begin{equation}\label{hyperbolicity}
	(x,z)_w \geq \min\lbrace (x,y)_w, (y,z)_w \rbrace -  \delta 
\end{equation}

\vspace{-2mm}
\noindent  or, equivalently,
\vspace{-5mm}

\begin{equation}
	\label{four-points-condition}
	d(x,y) + d(z,w) \leq \max \lbrace d(x,z) + d(y,w), d(x,w) + d(y,z) \rbrace + 2\delta. 
\end{equation}

\noindent The space $X$ is   {\em Gromov hyperbolic} if it is $\delta$-hyperbolic for some $\delta \geq 0$. \\
The above formulations of $\delta$-hyperbolicity are convenient when interested in taking limits (since they are preserved under ultralimits). However, we will also  make use of other classical characterizations of  $\delta$-hyperbolicity, depending on which one is more useful in the context.  \\
Recall that a {\em geodesic triangle} in $X$ is the union of three geodesic segments $[x,y], [y,z], [z,x]$ and  is denoted by $\Delta(x,y,z)$. For every geodesic triangle there exists a unique {\em tripod} $\overline \Delta$ with vertices $\bar{x},\bar{y},\bar{z}$ such that the lengths of $[\bar{x}, \bar{y}], [\bar{y}, \bar{z}], [\bar{z}, \bar{x}]$ equal the lengths of $[x,y], [y,z], [z,x]$ respectively. There exists a unique map $f_{\bar \Delta}$ from $\Delta(x,y,z)$ to the tripod $\overline \Delta$ that identifies isometrically the corresponding edges, 
and there are exactly three points $c_x \in [y,z], c_y \in [x,z], c_z\in [x,y]$ such that $f_{\bar \Delta} (c_x) = f_{\bar \Delta} (c_y) = f_{\bar \Delta} (c_z) = c$, where $c$ is the center of the tripod $\overline \Delta$. By definition of $f_{\bar \Delta}$ it holds: 
$$d(x,c_z) = d(x,c_y), \qquad d(y,c_x) = d(y,c_z), \qquad d(z,c_x)=d(z,c_y).$$
The triangle $\Delta(x,y,z)$ is called {\em $\delta$-thin}
if for every $u,v \in \Delta(x,y,z)$ such that $f_{\bar \Delta}(u)=f_{\bar \Delta}(v)$ it holds $d(u,v)\leq \delta$;  in particular the mutual distances between $c_x,c_y$ and $c_z$ are at most $\delta$. 
It is well-known that  every geodesic triangle in a geodesic {\em $\delta$-hyperbolic} metric space (as defined above)  is $4\delta$-thin, and moreover satisfies the {\em Rips' condition}:
\vspace{-1mm}
\begin{equation}\label{rips}
	[y,z] \subset \overline{B}([x,y]\cup [x,z], 4\delta).
\end{equation}

\noindent Moreover these last conditions are equivalent to the above definition of hyperbolicity, up to slightly increasing the hyperbolicity constant  $\delta$ in  (\ref{hyperbolicity}). \linebreak
As a consequence of the $\delta$-thinness of triangles we have the following: let $x,y,z\in X$, $f_{\bar \Delta}\colon \Delta(x,y,z) \to \bar \Delta$ the tripod approximation and $c_x,c_y,c_z$ as before. Then
\begin{equation}
	\label{approximation-tripod}
	(y,z)_x = d(x,c_z)=d(x,c_y) \qquad\text{ and } \qquad d(x,c_x)\leq d(x,c_y) + 4\delta
\end{equation}
\emph{All Gromov-hyperbolic  spaces, in this paper, we will be supposed  proper; we will however stress this assumption in the statements where it is needed.}

\subsection{Gromov boundary}
\label{subsec-gromov-boundary}
We fix a $\delta$-hyperbolic metric space $X$ and a base point $x_0$ of $X$. \\
The {\em Gromov boundary} of $X$ is defined as the quotient 
$$\partial_G X = \lbrace (y_n)_{n \in \mathbb{N}} \subseteq X \hspace{1mm} | \hspace{1mm}   \lim_{n,m \to +\infty} (y_n,y_m)_{x_0} = + \infty \rbrace \hspace{1mm} /_\sim,$$
where $(y_n)_{n \in \mathbb{N}}$ is any sequence of points in $X$ and $\sim$ is the equivalence relation defined by $(y_n)_{n \in \mathbb{N}} \sim (z_n)_{n \in \mathbb{N}}$ if and only if $\lim_{n,m \to +\infty} (y_n,z_m)_{x_0} = + \infty$.  \linebreak
We will write $ y = [(y_n)] \in \partial_G X$ for short, and we say that $(y_n)$ {\em converges} to $y$. \linebreak Clearly, this definition  does not depend on the basepoint $x_0$. \\
The Gromov product can be extended to points $y,z  \in \partial_GX$ by 
$$(y,z)_{x_0} = \sup_{(y_n) , (z_n) } \liminf_{n,m \to + \infty} (y_n, z_m)_{x_0}$$
where the supremum is over all sequences such that $(y_n) \sim y$ and $(z_n)\sim z$.
For any $x,y,z \in \partial_G X$ it continues to hold
\begin{equation}
	(x,y)_{x_0} \geq \min\lbrace (x,z)_{x_0}, (y,z)_{x_0} \rbrace - \delta.
\end{equation}
Moreover, for all sequences $(y_n),(z_n)$ converging to  $y,z$ respectively it holds
\begin{equation}
	\label{product-boundary-property}
	(y,z)_{x_0} -\delta \leq \liminf_{n,m \to + \infty} (y_n,z_m)_{x_0} \leq (y,z)_{x_0}.
\end{equation}
In a similar way is defined the Gromov product between a point $y\in X$ and a point $z\in \partial_GX$. This product  satisfies a condition analogue of  \eqref{product-boundary-property}.
\vspace{1mm}

Any ray $\gamma$ defines a point  $\gamma^+=[(\gamma(n))_{n \in \mathbb{N}}]$   of the Gromov boundary $ \partial_GX$: we  say that $\gamma$ {\em joins} $\gamma(0) = y$ {\em to} $\gamma^+ = z$,  and   we denote it by  $[y, z]$. 
Notice that any point $y \in X$ can be joined to any point  $z = [(z_n)]  \in \partial_GX$: \linebreak
in fact, the sequence $(z_n)$ must be unbounded (as $(z_n,z_n)_{x_0}$ is unbounded), so   the geodesic segments   $[y,z_n]$ converge  uniformly on  compact sets, by properness of $X$,  to a geodesic ray $\gamma =  [y,z]$. 
Analogously, given different points $z = [(z_n)], z' = [(z'_n)] \in \partial_GX$ there always exists  a geodesic line $\gamma$ joining $z$ to $z'$, i.e. such that  $\gamma|_{[0, +\infty)}$ and $\gamma|_{(-\infty,0]}$ join   $\gamma(0)$ to $z,z'$ respectively (just  consider the limit $\gamma$ of the segments $[z_n,z'_n]$; notice that  all these segments intersect a ball   of fixed radius centered at $x_0$, since $(z_n,z'_m)_{x_0}$ is uniformly bounded above). We call $z$ and $z'$ the  {\em positive} and {\em negative endpoints} of $\gamma$, respectively,  denoted  $\gamma^\pm$. 
We will also write, for short,  $\partial \gamma := \{ \gamma^+, \gamma^- \}$.
\vspace{3mm}

%

  As $X$ is not uniquely geodesic,  it may happen that there are several geodesic rays joining a point of $X$ to some point $z \in \partial_GX$, or several geodesic lines joining two points of the boundary. 
However, the following standard uniform estimates hold:


\begin{lemma}[Prop. 8.10 of \cite{BCGS17}]
	\label{parallel-geodesics} 
	Let $X$ be a  $\delta$-hyperbolic space.
	\begin{itemize}
		\item[(i)] let $\gamma,\xi$ be two geodesic rays with $\gamma^+=\xi^+$: then   there exist $\hspace{1mm} t_1,t_2\geq 0$ with $t_1+t_2 = d(\gamma(0),\xi(0))$  
		 such that  $d(\gamma(t + t_1),\xi(t + t_2))\leq 8 \delta$, $\forall t\geq 0$;
		\item[(ii)] let $\gamma, \xi$ be two geodesic lines with $\gamma^+=\xi^+$ and $\gamma^-=\xi^-$: then for all $t\in \mathbb{R}$ there exists $\hspace{1mm} s\in \mathbb{R}$ such that  $d(\gamma(t),\xi(s)) \leq 8 \delta$. 
	\end{itemize}
 \end{lemma}

\subsection{Projections}

\noindent Recall that a  subset $C\subseteq X \cup \partial_G X$ is said {\em convex}  if for every $x,y \in C$ there exists at least one geodesic (segment, ray, line) joining $x$ to $y$ that is included in $C$.  
Given any  closed, convex subset $C$  of $X$  and a point $x\in X$,  a {\em projection} of $x$ to $C$ is a point $c\in C$ such that $d(x,C)=d(x,c)$. Since $C$ is closed and $X$ is proper, it is clear that there exists at least a projection. 

\noindent A fundamental tool in the study of projections in $\delta$-hyperbolic   spaces is the following: 
\begin{lemma}[Projection Lemma, cp. Lemma 3.2.7 of \cite{CDP90}] ${}$
	\label{projection}
	
	\noindent 	Let $X$ be a $\delta$-hyperbolic   space, and let $x,y,z \in X$. For any geodesic segment $[x,y]$ we have:
	$$(y,z)_x \geq d(x, [y,z]) - 4 \delta.$$
\end{lemma}

\noindent Therefore if $C$ is a convex subset and $x_0$ is a projection of $x$ on $C$  then   $(x_0,c)_x \geq  d(x, x_0) - 4 \delta$ for all $c\in C$.
This easily implies that the projection $x_0$ satisfies, for all $c \in C$:
\begin{equation}
	\label{eqprojection}
	(x,c)_{x_0} \leq 4 \delta
\end{equation}
One can  then extend the definition of projection to boundary points, using this relation, as follows:
we say that {\em $x_0$ is a projection of $x  \in \partial_GX$  on $C$} if 
$$(x,c)_{x_0} \leq  5\delta \hspace{2mm} \textup{ for all }c\in C.$$

\noindent In the next lemma we summarize the properties of projections we  need.\linebreak Recall that, since $C$ is convex and closed, then it is naturally a geodesic,  $\delta$-hyperbolic, proper metric space; furthermore the Gromov boundary  $\partial_GC$ of $C$  canonically   embeds into $\partial_GX$.
\begin{lemma}
	\label{projection-properties}
	Let $X$ be a proper, $\delta$-hyperbolic metric space and $C$ be a closed, convex subset of $X$. Let $x,x'\in X\cup \partial_G X \setminus \partial_G C$. The following facts hold:
	\begin{itemize}	
		\item[(a)] there exists at least one projection of $x$ on $C$;
		\item[(b)] if $x_1,x_2$ are two projections of $x$ on $C$ then $d(x_1,x_2)\leq 10 \delta$;
		\item[(c)] if $x_0$ and  $x'_0$ are respectively projections of $x$ and $x'$ on $C$, then \\ $d(x_0,x'_0)\leq d(x,x') + 12\delta$.
	\end{itemize}
\end{lemma} 
\begin{proof}
	We first  show  the existence of a projection for points $x  \in \partial_GX \setminus \partial_GC$.\linebreak
	Let   $(x_n)$ be a sequence  converging to $x$ and let   $c_n$ be a projection of $x_n$ on $C$. \linebreak
	First of all we claim that the sequence $(c_n)$ is bounded.
	As the sequence $(c_n)$ is  in $C$ and $x \notin \partial_G C$, then $(c_n)$ is not  equivalent to $(x_n)$. In particular $(x_n,c_n)_{c_0} \leq D$ for some $0\leq D < +\infty$ and some $c_0\in C$. This means
	$$d(c_0,x_n) + d(c_0,c_n) - d(c_n,x_n) \leq 2D.$$
	As $x_0 \in C$ and $c_n$ is a projection of $x_n$ on $C$, we have $d(c_0,x_n) \geq  d(c_n,x_n) $ and therefore $d(c_0, c_n) \leq 2D$ for all $n$.
	Therefore the sequence $c_n$ converges, up to a subsequence, to a point $c\in C$. Notice that for any $n$ and   any $c'\in C$ we have $(x_n, c')_{c_n} \leq 4 \delta$. 
	Applying \eqref{product-boundary-property} we get for all $c' \in C$
	$$(x,c')_c \leq \limsup_{n \to +\infty} (x_n,c')_c +\delta \leq \limsup_{n \to +\infty} (x_n,c')_{c_n} + d(c_n , c) +\delta \leq  5 \delta.$$
	This proves (a). Assertion (b) is an easy consequence of the definition, as
	$$(x,x_1)_{x_2}\leq 5 \delta, \qquad (x,x_2)_{x_1}\leq 5 \delta,$$
	so $d(x_1,x_2) = (x,x_1)_{x_2} + (x,x_2)_{x_1} \leq  10\delta$. \\
	Finally the proof of (c) can be found in \cite{CDP90}, Corollary 10.2.2.
\end{proof}
\begin{obs}
	\label{projection-remark} {\em 
		We record here a consequence of  the proof above:   if $(x_n)$ is a sequence of points converging to a point  $x\in \partial_G X \setminus \partial_G C$ and $c_n$ is a projection of $x_n$ on $C$ for all $n$, then, up to a subsequence, the limit point of the sequence $c_n$ is a projection of $x$ on $C$. }
\end{obs}

We now recall the Morse property of geodesic segments in a Gromov-hyperbolic space.  A map $\alpha \colon [0,l] \to X$ is a {\em $(1,\nu)$-quasigeodesic segment} if for any $t,t'\in [0,l]$ it holds:
 $\vert t - t' \vert - \nu \leq d(\alpha(t),\alpha(t')) \leq \vert t - t' \vert + \nu.$ \linebreak
The points $\alpha(0)$ and $\alpha(l)$ are called the endpoints of $\alpha$.
\begin{prop}[Morse Property]
	\label{Morse}
	Let $X$ be a $\delta$-hyperbolic space and let $\alpha$ be a $(1,\nu)$-quasigeodesic segment. 
	The following facts hold:

	\begin{itemize}	
		\item[(a)] for any geodesic segment  $\beta$ joining the endpoints of $\alpha$ we have $d_H(\alpha, \beta) \leq \nu +  12 \delta$, where $d_H$ is the Hausdorff distance;
		\item[(b)] for any $(1,\nu)$-quasigeodesic segment $\beta$ with the same endpoints of $\alpha$ and for any time $t$ where both $\alpha$ and $\beta$ are defined it holds
		$d(\alpha(t), \beta(t))\leq  6\nu + 48\delta $.
	\end{itemize}
\end{prop}
\noindent The proof of the first part can be found in \cite{Bow05}, while the second part is classical and follows from a straightforward computation. 

\noindent We can now state the contracting property of projections we will use later:
\begin{prop}[Contracting Projections]
	\label{projections-contracting}
	Let $X$ be a proper, $\delta$-hyperbolic space, $C\subseteq X$ any closed  convex subset and $Y\subseteq X$  another convex subset. The following facts hold:

	\begin{itemize}	
		\item[(a)] if the projections $c, c'$ on $C$ of  $y$ and $y' \in Y$  satisfy  $d(c,c')>  9 \delta$,  then   $[y,c]\cup [c,c'] \cup [c',y']$ is a $(1, 18 \delta)$-quasigeodesic segment;
		\item[(b)] if $d(Y,C)> 30 \delta$, then any two projections on $C$ of points of $Y$ are at distance at most $ 9 \delta$;
		\item[(c)] if  $\alpha, \beta$ are geodesic  with $\partial \alpha \cap \partial \beta =\emptyset$, then   any projection $b_+$ of $\beta^+$ on $\alpha$ satisfies   $d(b_+,\beta)\leq \max\lbrace  49 \delta, \hspace{1mm} d(\alpha,\beta) +  19 \delta\rbrace$.\\
	\end{itemize}
\end{prop}

\vspace{-5mm}	
\begin{proof}
	By assumption we have $(y,c')_c \leq  4 \delta$ and $(y',c)_{c'}\leq 4 \delta$, i.e.
	\begin{equation}
		\label{richiamo}
		d(y,c')\geq d(y,c) + d(c,c') -  8 \delta,\qquad d(y',c)\geq d(y',c') + d(c,c') -  8 \delta.
	\end{equation}
	We apply the four-points condition \eqref{four-points-condition} to $(y,c',y',c)$ obtaining
	$$d(y,c') + d(y',c) \leq \max \lbrace d(y,y') + d(c,c'), d(y,c) + d(y',c')\rbrace + 2\delta.$$
	Assuming $d(c,c')> 9 \delta$ we get by \eqref{richiamo}
	$$d(y,c') + d(y',c) \!\geq\! d(y,c) + d(c,c') + d(y',c') + d(c,c') -  16 \delta \!> \! d(y,c) + d(y',c') + 2\delta$$
	and the 4-points condition becomes: $d(y,c') \!+\! d(y',c) \leq d(y,y') \! +\! d(c,c') + 2\delta.$\linebreak
	Using again \eqref{richiamo} we get
	$$d(y,c) + d(c,c') + d(y',c') + d(c,c') -  16 \delta \leq d(y,y') + d(c,c') + 2\delta$$
	which proves (a).
	We suppose now $d(Y,C)>30 \delta$ and that there are two points $y,y'\in Y$ with projections $c,c'$ on $C$ such that $d(c,c')> 9 \delta$. Then, the path $[y,c]\cup [c,c'] \cup [c',y']$ is a $(1, 18 \delta)$-quasigeodesic segment by (a), and  it is at Hausdorff distance at most $30 \delta$ from any geodesic segment $[y,y']$, by Lemma \ref{Morse}. As $Y$ is convex, one of these geodesic segments is included in $Y$, so  $c$ is at distance at most $30 \delta$ from  $Y$. This contradiction proves (b).\\
	In order to prove (c) we observe that $\alpha,  \beta$ are two closed, convex subsets of $X$.\linebreak
	We divide the proof in two cases.\\
	{\em Case 1:  $d(\alpha,\beta)>30 \delta$}. Then let $x_0 \in \alpha$ and $y_0\in \beta$ be points minimizing the distance between $\alpha$ and $\beta$; in particular, $x_0$ is a projection of $y_0$ on $\alpha$. \linebreak
	By Remark \ref{projection-remark} and by (b) there exists a projection $b_+$ of $\beta^+$  on $\alpha$
	that falls at distance at most $ 9 \delta$ from $x_0$. Therefore we have
	$d(b_+, \beta) \leq d(b_+, x_0) + d(x_0, \beta)\leq  9 \delta  + d(\alpha, \beta)$. The thesis for all possible projections of $\beta^+$ on $\alpha$ follows from Lemma \ref{projection-properties}.(b).\\
	{\em Case 2: $d(\alpha,\beta)\leq 30 \delta$.} In this case we parameterize $\beta$ in such a way that $\beta(0)$ is at distance at most $30 \delta$ from $\alpha$. Then let
	\vspace{-4mm}
	
	$$t_0 = \max \lbrace t\in [0,+\infty) \hspace{1mm} \text{ s.t. } \hspace{1mm}  d(\beta(t), \alpha)\leq 30 \delta \rbrace,$$
	let $y_0=\beta(t_0)$ and let $x_0$ be any projection of $y_0$ on $\alpha$.  The convex subset $[\beta(t_0), \beta^+]$ of $\beta$ is at distance $> 30 \delta$  from $\alpha$, so arguing as before we have that any projection $b_+$ of $\beta^+$ on $\alpha$ is at distance at most $ 19 \delta$ from $x_0$. Then, again $d(b_+,\beta)\leq d(b_+,x_0) + d(x_0,y_0) \leq  49 \delta$.
\end{proof}

\subsection{Helly's Theorem}
A subset $C\subseteq X \cup \partial_G X$ is said {\em $\lambda$-quasiconvex}, where $\lambda \geq 0$, if for every $x,y \in C$ there exists at least one geodesic (segment, ray or line) joining $x$ to $y$ that is included in $\overline{B}(C,\lambda)$.
The subset $C$ is called {\em starlike with respect to a point $x_0\in C$} if for all $x \in C$ there exists at least one geodesic (segment, ray or line) $[x_0, x]$ entirely included in $C$. For instance a convex set is starlike with respect to all of its points. The proof of the following lemma can be found in \cite{DKL18} and \cite{CDP90}:
\begin{lemma}[Lemma 3.3 of \cite{DKL18} and Proposition 10.1.2 of \cite{CDP90}]
	\label{starlike-quasiconvex} ${}$\\
	Let $X$ be  $\delta$-hyperbolic and let $C\subseteq X \cup \partial_G X$ be starlike with respect to $x_0$. Then $C$ is $ 12 \delta$-quasiconvex and  $\overline{B}(C, \lambda)$ is $ 20 \delta$-quasiconvex  for all $\lambda \geq 0$.
\end{lemma}

We state now the version of Helly's Theorem which we will need. 

\begin{prop}[Helly's Theorem]
	\label{Helly}
	Let $X$ be a $\delta$-hyperbolic  space and let $(C_i)_{i \in I }$ be a family of $\lambda$-quasiconvex subsets of $X$ such that $C_i\cap C_j \neq \emptyset$  $\forall i,j$. \linebreak Then: 
	$$\bigcap_{i \in I} \overline{B}(C_i,  119\delta +15 \lambda) \neq \emptyset.$$
\end{prop}

This result has been proved for the first time in \cite{CDV17}, Theorem 5.1 in case of hyperbolic graphs. As pointed out in Remark 6 \emph{loc.cit.} the same holds for general geodesic hyperbolic spaces. For completeness we present here a different proof which follows the one given by \cite{BF18}, with the minor modifications needed to deal with {\em quasiconvex} subsets instead of convex ones. The proof is a direct consequence of the following lemma.
\begin{lemma}
	Let $X$ be a $\delta$-hyperbolic  space,  let  $C_1, C_2 \subseteq X$  be two \linebreak  $\lambda$-quasiconvex subsets with non-empty intersection and let  $x_0 \in X$ be fixed.  \linebreak Assume that we have points $x_1 \in C_1$  and $x_2 \in C_2$  which satisfy:
\vspace{-3mm}	
	
	$$d(x_0,x_i)\leq d(x_0,C_i) + \delta \textup{ for } i=1,2$$
	$$d(x_0,x_1) \geq d(x_0,x_2) - \delta$$
	Then, $d(x_1,C_2)\leq  119\delta +15 \lambda$.
\end{lemma}
\begin{proof}
	Let $u\in C_1 \cap C_2$. By the Projection Lemma \ref{projection} we have 
\vspace{-3mm}	
	
	$$(u,x_1)_{x_0} \geq d(x_0,[u,x_1]) -  4 \delta.$$
	Moreover,  
	 $d(x_0,[u,x_1]) \geq d(x_0, \overline{B}(C_1, \lambda)) \geq d(x_0, C_1) - \lambda \geq d(x_0,x_1) - \lambda - \delta.$ 
	So $(u,x_1)_{x_0} \geq d(x_0,x_1) - \lambda -5 \delta$.
	Computing the Gromov product we get
	$$d(u,x_0) + 10 \delta + 2\lambda \geq d(x_1,x_0) + d(x_1,u).$$
	The same conclusion holds for $x_2$. Hence the two paths $\alpha = [u,x_1] \cup [x_1,x_0]$ and $\beta=[u,x_2] \cup [x_2,x_0]$ are $(1, 10 \delta + 2\lambda)$-quasigeodesic segments with same endpoints. Applying Proposition \ref{Morse} we conclude that for any $t$ where the two paths are defined it holds: 
\vspace{-3mm}		
	
	$$d(\alpha(t),\beta(t))\leq 108\delta +12 \lambda.$$ 
	We estimate now the distance between $x_1$ and the geodesic segment $[u,x_2]$. Let $t_1, t_2$  be   such that $\alpha(t_1)=x_1$ and $\beta(t_2)=x_2$. If $t_1 \leq t_2$ then $\beta(t_1)$ belongs to $ [u,x_2]$ and   $t_1$ is a common time for both geodesic segments. \linebreak
	Therefore we can conclude that $d(x_1, [u,x_2])\leq 108\delta +12 \lambda$. We consider now the case $t_1\geq t_2$.
	Since $d(x_0,x_1)\geq d(x_0,x_2) - \delta$ we know that
	\begin{equation*}
		\begin{aligned}
			t_1=d(u,x_1) &\leq d(x_0,u)-d(x_0,x_1) + 10 \delta + 2\lambda \\
			&\leq d(x_0,u) - d(x_0,x_2) + 11\delta + 2\lambda \\
			&\leq d(x_0,x_2) + d(x_2, u) - d(x_0,x_2) + 11\delta + 2\lambda \\
			&= t_2 + 11\delta + 2\lambda.
		\end{aligned}
	\end{equation*}

\vspace{-3mm}	
\noindent	Therefore we get
\vspace{-3mm}	
	
	\begin{equation*}
		\begin{aligned}
			d(x_1,x_2) = d(\alpha(t_1),\beta(t_2)) &\leq d(\alpha(t_1), \alpha(t_2)) + d(\alpha(t_2), \beta(t_2)) \\
			&\leq 11\delta + 2\lambda + 108\delta +12 \lambda =  119\delta +14 \lambda.
		\end{aligned}
	\end{equation*}
	In any case we have $d(x_1, [u,x_2])\leq  119\delta +14 \lambda$. In conclusion
	$$d(x_1,C_2)\leq d(x_1, B(C_2,\lambda)) + \lambda \leq d(x_1, [u,x_2]) + \lambda \leq 119\delta +15 \lambda. \vspace{-8mm}	$$
 
\end{proof}

\vspace{2mm}	
\begin{proof}[Proof of Proposition \ref{Helly}.]
	We choose a point $x_0\in X$. Let $x_i \in C_i$ be points of $C_i$ that $\delta$-almost realize the distance $d(x_0,C_i)$. The points $x_i$ satisfy the first assumption of the previous lemma. Moreover without loss of generality we can assume that $d(x_0,x_1)\geq d(x_0,x_i) - \delta$ for all $i\in I$.
	Applying the lemma to any couple $C_1, C_i$  we find that the point $x_1$ belongs to the intersection of all the desired neighbourhoods of $C_i$.
\end{proof}

\vspace{1mm}	
\subsection{Elementary subgroups of packed,  $\delta$-hyperbolic   spaces}\label{sub-elementary}

We record here some basic properties about  discrete groups of isometries of a Gromov-hyperbolic metric space.

\noindent First recall that the isometries of $X$ are classified into three types
according to the behaviour of their orbits (cp. for instance \cite{CDP90}):
\begin{itemize}
	\item an isometry $g$ is {\em elliptic} if it has bounded orbits; when $g$ acts discretely, this is the same as asking that it is a torsion element, cp. \cite{BCGS17};
	\item an isometry $g$ is {\em parabolic} if there exists a point $g^\infty \in \partial_GX$ such that for all $x\in X$ the sequences $(g^k x)_{k \geq 0}$ and $(g^k x)_{k \leq 0}$ converge to $g^\infty$;
	\item an isometry $g$ is {\em hyperbolic} if   the map $k\mapsto g^kx$ is a quasi-isometry  $\forall x\in X$, i.e. there exist $L,C > 0$ such that for any $k,k'\in \mathbb{Z}$ it holds
	$$\frac{1}{L}\vert k - k'\vert - C \leq d(g^kx, g^{k'} x) \leq L\vert k - k' \vert + C.$$
	In this case  there exist two points $g^- \neq g^+ $ in $ \partial_GX$  such that for any $x\in X$ the sequence $(g^k x)_{k\geq 0}$ converges to $g^+$ and the sequence $(g^k x)_{k\leq 0}$ converges to $g^-$. 
\end{itemize}

\noindent Also, recall that the {\em asymptotic displacement} of an isometry $g$ is defined as the limit (which exists and  does not depend on the choice of $x \in X$):
$$\Vert g \Vert = \lim_{n\to +\infty}\frac{d(x,g^n x)}{n}.$$
It is well known that for any isometry $g$ of $X$ and for any $k\in \mathbb{Z}^\ast$ it holds $\Vert g^k \Vert = k\Vert g \Vert$ and that $g$ is hyperbolic if and only if $\Vert g \Vert > 0$ (see \cite{CDP90}).\\
The following lemma is well known.

\begin{lemma}[Lemma 8.20 of \cite{BCGS17}, \cite{CDP90}]
	\label{lemma-powers}
	For every isometry $g$  and every $x \in X$ we have:
	$d(x,g^2 x) \leq d(x,gx) + \Vert g \Vert  + 2\delta$.
\end{lemma}

 Any isometry of $X$ acts naturally on $\partial_GX$. Namely there exists a natural topology on $X \cup \partial_GX$ extending the metric topology of $X$. 
 If $X$ is proper the sets $\partial_GX$ and $X\cup \partial_GX$ are compact with respect to this topology and any isometry of $X$ acts as a homeomorphism on $X\cup \partial_GX$. \linebreak
 If $g$ is parabolic then $g^\infty$ is the unique fixed point of the action of $g$ on $X\cup \partial_GX$, while if $g$ is hyperbolic then $g^-,g^+$ are the only fixed points of the action of $g$ on $X\cup \partial_GX$. The set of fixed points of an isometry $g$ on the Gromov boundary   is denoted by Fix$_\partial(g)$.
 For any $k \in \mathbb{Z}^\ast = \mathbb{Z}\setminus \lbrace 0 \rbrace$ we have that an isometry $g$ is elliptic (resp.parabolic, hyperbolic) if and only if $g^k$ is elliptic (resp.parabolic, hyperbolic); moreover, if $g$ is parabolic or hyperbolic it holds $\text{Fix}_\partial (g^k) = \text{Fix}_\partial (g)$. 
 \vspace{2mm}

 The {\em limit set} $\Lambda(\Gamma)$ of a discrete group of isometries $\Gamma$ of a proper, $\delta$-hyperbolic  space $X$ is the set of accumulation points of the orbit $\Gamma x$ on $\partial_GX$, where $x$ is any point of $X$; it is the smallest $\Gamma$-invariant closed set of the Gromov boundary, cp \cite{Coo93}.  The group $\Gamma$ is called {\em elementary} if $\# \Lambda(\Gamma) \leq 2$. 
 For an elementary discrete group   $\Gamma$ 
 there are three possibilities (cp. \cite{Gro87}, \cite{CDP90}, \cite{DSU17}, \cite{BCGS17}):
 \begin{itemize}
 	\item $\Gamma$ is {\em elliptic}, i.e.  $\#\Lambda(\Gamma)=0$;  then   $d(x, \Gamma x) <\infty$ for all $x \in X$, so the orbit  of $\Gamma$ is finite by discreteness;
 	\item $\Gamma$ is  {\em parabolic}, i.e.  $\#\Lambda(\Gamma)=1$; then in this case $\Gamma$ contains only parabolic or elliptic elements and all the parabolic elements have the same fixed point at infinity;
 	\item $\Gamma$ is  {\em lineal}, i.e.  $\#\Lambda(\Gamma)=2$; in this case $\Gamma$ contains only hyperbolic or elliptic elements and all the hyperbolic elements have the same fixed points at infinity.
 \end{itemize}
 
 
 \noindent So if two non-elliptic isometries $a, b$ generate a discrete elementary group then they are either both parabolic or both hyperbolic and they have the same set of fixed points in $\partial_GX$;  conversely if $a,b$ are two non-elliptic isometries of $X$ generating a discrete group $\langle a,b \rangle$ such that Fix$_\partial(a) =$ Fix$_\partial(b)$, then $\langle a, b\rangle$ is elementary.
 We also recall the following property of elementary subgroups of general Gromov-hyperbolic spaces:
 \begin{lemma}
 	\label{M-elementary}
 	Let $X$ be a Gromov-hyperbolic space and let $\Gamma$ be a discrete   group of isometries of $X$. Then for any  non-elliptic  $g\in \Gamma$ there exists a unique maximal, elementary  subgroup of $\Gamma$ containing $g$.
 \end{lemma}
 \begin{proof}
 	The maximal, elementary subgroup of $\Gamma$ containing $g$ is:  $$\{ g' \in \Gamma \hspace{1mm} | \hspace{1mm} g' \cdot \textup{Fix}_\partial(g) = \textup{Fix}_\partial(g) \}.$$
 	
 	\vspace{-8mm}
 \end{proof}

 \vspace{2mm}
 
 It is well known that any  virtually nilpotent  group of isometries  $\Gamma$ of $X$ is elementary (since any non-elementary group $\Gamma$ contains a free subgroup).
 Conversely if $\Gamma$ is elliptic it is virtually nilpotent since it is finite. Also any lineal group is virtually cyclic (cp.  Proposition 3.29 of \cite{Cou16}), hence virtually nilpotent. On the other hand there are examples of non-virtually nilpotent,  even free non abelian, parabolic groups acting on simply connected Riemannian manifolds with curvature $\leq -1$ (see \cite{Bow93}, Sec. 6). However, under a mild packing assumption it is possible to conclude that any parabolic group is virtually nilpotent. Some version of this fact is probably known to the experts and we present here the proof for completeness;  we thank S. Gallot for  explaining it to us. 
 \begin{prop}
 	\label{parabolic-nilpotent}
 	Let $X$ be a proper, geodesic, Gromov hyperbolic   space that is $P_0$-packed at some scale $r_0$. Then any finitely generated, discrete, parabolic group of isometries $\Gamma$ of $X$ is virtually nilpotent.
 \end{prop}
 
 \noindent Hence, we deduce:
 
 \begin{cor}[Elementary groups are virtually nilpotent]
 	\label{elementary-nilpotent} ${}$\\
 	Let $X$ be a proper, geodesic, Gromov hyperbolic  space,  $P_0$-packed at scale $r_0$. Then a discrete, finitely generated group of isometries of $X$ is elementary if and only if it is virtually nilpotent.
 \end{cor}
 
 \noindent We remark that the scale of the packing is not important: it plays the role of an asymptotic bound on the complexity of the space. The proof is based on the following fundamental result proved in \cite{BGT11}: 
 \begin{theo}[Corollary 11.2 of \cite{BGT11}]
 	\label{nilpotency}
 	For every $p\in \mathbb{N}$ there exists $N(p) \in \mathbb{N}$ such that the following holds for every group $\Gamma$ and every finite, symmetric  generating set $S$ of $\Gamma$: if there exists some $A\subset \Gamma$ such that $S^{N(p)} \subset A$ and $\textup{card} \left( A\cdot A \right)  \leq p \cdot \textup{card} (A) $ then $\Gamma$ is virtually nilpotent.
 \end{theo}
 
 \begin{proof}[Proof of Proposition \ref{parabolic-nilpotent}]
 	Let $S$ be a finite, symmetric generating set of $\Gamma$.\linebreak
 	Moreover let $\Lambda (\Gamma) =\{z \}$ and let $\gamma$ be any geodesic ray  such that $\gamma^+ = z$. 
 	Finally set $\Sigma_{R}(x) := \lbrace g\in \Gamma \text{ s.t. } d(gx, x) \leq R\rbrace$. 
 	\vspace{1mm}
 	
 	\noindent {\em Step 1. Setting $R_0 = \max \lbrace 2r_0, 30\delta \rbrace$ and  $p = P_0(1+P_0)^{\frac{9R_0}{r_0} - 1}$, we have for any   $x\in X$:  
 		$$\textup{card} \big( \Sigma_{R_0}(x) \cdot \Sigma_{R_0}(x) \big) \leq p \cdot \textup{card} ( \Sigma_{R_0}(x) )$$} 
 	Actually by Lemma 4.7 of \cite{CavS20}  we have 
 	$$\text{Pack}\left(9R_0, \frac{R_0}{2}\right) \leq \text{Pack}(9R_0, r_0) \leq P_0(1+P_0)^{\frac{9R_0}{r_0} - 1}.$$
 	Then it easily follows (cp. Lemma 3.12 of \cite{BCGS17}) that  
 	$$\frac{\textup{card} \big(\Sigma_{R_0}(x) \cdot \Sigma_{R_0}(x) \big)}{\textup{card} ( \Sigma_{R_0}(x) )} \leq \frac{\textup{card} ( \Sigma_{2R_0}(x) ) }{\textup{card} (\Sigma_{R_0}(x) )} \leq P_0(1+P_0)^{\frac{9R_0}{r_0} - 1}$$
 	which is our claim. Remark that $p$ does not depend on the point $x$.
 	\vspace{1mm}
 	
 	\noindent {\em Step 2.  There exists $T = T(S,\gamma, p)$ such that  $d(\gamma(T), g\gamma(T))\leq 30\delta$ for all  $g \in S^{N(p)}$  (where $N(p)$ is the value associated to $p$ given by Theorem \ref{nilpotency}).}
 	
 	\vspace{1mm} 
 	\noindent Let $\rho_0 = \max_{s\in S} d(\gamma(0), s\gamma(0))$. So we have $d(\gamma(0), g\gamma(0))\leq N(p)\rho_0$ for all $g \in S^{N(p)}$. 
 	Let $g \in S^{N(p)}$. By definition we have $gz = z$, so  $(g\gamma)^+=z$.\linebreak    
 	Then  by Lemma \ref{parallel-geodesics}   there exist $t_1,t_2 \geq 0$ such that $t_1+t_2 = d(\gamma(0), g\gamma(0))$ and $d(\gamma(t+ t_1), g\gamma(t+t_2)) \leq  8 \delta$ for all $t\geq 0$. Therefore,  
 	\begin{equation}\label{s+}
 		d(\gamma(s+ t_1 - t_2), g\gamma(s)) \leq  8 \delta
 	\end{equation} 
 	for all $s\geq T:=\max\lbrace t_1,t_2\rbrace \leq N(p)\rho_0$. 
 	In the following we may assume $t_1 \geq t_2$ and call $\Delta=t_1-t_2$.  
 	If $\Delta \leq  9 \delta$, we apply the previous estimate to $s=T$ and we get $d(\gamma(T), g\gamma(T))\leq  17 \delta$, so the claim is true. \linebreak Otherwise, we have $\Delta >  9 \delta$, and
 	we consider the triangle with vertices \linebreak $A=\gamma(T +  \Delta ), B=\gamma(T + 2\Delta)$ and $C=gA$.
 	Let $ (\bar A, \bar B, \bar C)$ be the corresponding tripod  with center $\bar o$ with edge lengths  
 	$\rho= \ell([\bar A, \bar o])$,  $\sigma=\ell([\bar C, \bar o])$ and $\tau=\ell([\bar B, \bar o])$. We therefore have:
 	$$\rho + \tau = \Delta, \hspace{4mm} \sigma + \tau \leq  8 \delta  \hspace{2mm}\textup{and }  \hspace{2mm} \sigma+ \rho = d(A,gA).$$
 	This implies that $\rho - \sigma = (\rho + \tau  ) - (\sigma   + \tau) \geq \Delta - 8 \delta \geq 0$. In particular,  if $m$ is the midpoint of $[A,gA]$ and $\bar m$ the corresponding point on the tripod,  we have $d(\bar m, \bar A) \leq d(\bar o, \bar A)$, so  there exists a point $m'\in [A,B]$ such that $d(m,m')\leq 4 \delta$. 
 	Applying Lemma 8.21 of \cite{BCGS17}  \footnote{Notice that in \cite{BCGS17} they use the notation $\ell(g)$ to denote $\Vert g \Vert$, while we use $\ell(g)$ to denote the minimal displacement.} we deduce
 	$$d(m',gm')\leq d(m,gm) +  8 \delta \leq 14 \delta$$ (as  $g$ is elliptic or parabolic, so $\Vert g \Vert = 0$).
 	Moreover, as $m' = \gamma(s + \Delta)$ for some $s\geq T$ (since $\Delta\geq0$) we have 
 	by (\ref{s+}) that  $d(m',g\gamma(s))\leq  8 \delta$. By the triangle inequality we deduce
 	$$\Delta = d(g\gamma(s),gm') 
 	\leq 22 \delta.$$
 	Therefore also in this case we get $d(\gamma(T),g\gamma(T)) \leq  30 \delta$.
 	\vspace{1mm}
 	
 	\noindent 	{\em Conclusion.} We have $S^{N(p)} \subset \Sigma_{R_0}(\gamma(T))$, where $T$ is the constant of step 2. \linebreak  So we  apply Theorem \ref{nilpotency} and conclude that $\Gamma$ is virtually nilpotent.
 \end{proof}

 \section{Isometries of packed,  $\delta$-hyperbolic  \textup{GCB}-spaces}

 {\em Throughout   this section   $(X,\sigma)$ will be a proper, $\delta$-hyperbolic \textup{GCB}-space.} \\\emph{We will consider only $\sigma$-invariant isometries (called  $\sigma$-isometries) of $X$: \linebreak they are isometries $g$ such that for all $x,y\in X$ it holds $\sigma_{g(x)g(y)} = g(\sigma_{xy})$.}

 \subsection{Special properties in the convex case}
The boundary at infinity of $X$, denoted by $\partial X$, is defined as the set of geodesic rays modulo the equivalence relation $\approx$, where 
$$\gamma \approx \xi \hspace{1mm}\textup{ if and only if } 
d_\infty (\gamma, \xi) = \sup_{t\in [0,+\infty)} d(\gamma(t), \xi(t)) < +\infty$$
endowed with the topology of  uniform convergence on compact subsets. \linebreak
The $\sigma$-boundary at infinity of $X$, denoted $\partial_\sigma X$, is the set of $\sigma$-geodesic rays modulo $\approx$. In general the natural inclusion $\partial_\sigma X \subseteq \partial X$ can be strict, but these boundaries coincide when $X$ is proper and $\delta$-hyperbolic:
\begin{lemma}
	\label{bicombing-boundaries}
	Let $\sigma$ be a convex, consistent, reversible geodesic bicombing of a proper, $\delta$-hyperbolic metric space $X$. Then:
\begin{itemize}	
 \item[(i)] geodesic bigons are uniformly thin, i.e. for every  geodesic segments   $\gamma,\xi$   between  $x,y\in X$ it holds $d(\gamma(s),\xi(s))\leq 2\delta$,  for all $0\leq s\leq d(x,y)$;
		\item[(ii)] for every geodesic ray $\gamma$ and every $x\in X$ there exists a unique \linebreak $\sigma$-geodesic ray starting at $x$ which is equivalent to $\gamma$, so $\partial_\sigma X = \partial X$;
		\item[(iii)] there exists a natural homeomorphism between $\partial_\sigma X$ and $\partial_G X$;
		\item[(iv)] for every geodesic line $\gamma$ there exists a $\sigma$-geodesic line $\gamma'$ staying at bounded distance from $\gamma$.
	\end{itemize}
\end{lemma}
\begin{proof}
Claim (i) is well known in $\delta$-hyperbolic spaces. \\
The uniqueness in (ii) follows by the convexity of $\sigma$. 
Actually, if $\xi, \xi'$ are $\sigma$-geodesic rays from $x$ such that $d_\infty (\xi, \xi') \leq D$,   then for any $s>0$  we have $d(\xi(st),\xi'(st)) \leq sd(\xi(t),\xi'(t)) \leq sD$ for all $t$; taking $t \gg0$ it follows that $d(\xi(s),\xi'(s)) \leq \frac{sD}{t}$ is arbitrarily small, hence zero.
To show the existence, we can suppose $\gamma$ is a $\sigma$-geodesic ray. Indeed for a sequence $t_n \to +\infty$ we consider the $\sigma$-geodesic segments $[\gamma(0),\gamma(t_n)]$. By the properness of $X$ they converge to a geodesic ray $\gamma'$ with origin $\gamma(0)$ staying at bounded distance from $\gamma$ by (i),  which is in fact a $\sigma$-geodesic ray   since the bicombing $\sigma$ is continuous (Lemma \ref{bicombing-basics}).
Now for a generic $x\in X$ we consider a sequence of $\sigma$-geodesic segments $\xi_n=[x,\gamma(t_n)]$ for $t_n \rightarrow +\infty$. As before they converge to a $\sigma$-geodesic ray $\xi$ with origin $x$. By convexity we get $d(\xi(T), \gamma)\leq d(x,\gamma(0))$ for all $T\geq 0$, so $\xi$ is asymptotic to $\gamma$.
From (ii) and the classical homeomorphism between $\partial X$ and $\partial_GX$ (see Lemma III.3.13 of \cite{BH09}) we conclude (iii). \\
 Finally (iv) follows as explained in Section \ref{subsec-gromov-boundary}, by taking the limit of  $\sigma$-geodesic segments $\xi_n=[\gamma(-n), \gamma(n)]$. By the continuity of $\sigma$,  this gives a $\sigma$-geodesic line, which is  equivalent to $\gamma$ by (i).
\end{proof}

Recall that the {\em displacement function} of an isometry $g$  is defined as $d_g(x)=  d(x,gx)$, 
and the {\em minimal displacement}  of $g$ is   $$\ell(g) = \inf_{x\in X}d(x, gx).$$
In  general the   asymptotic displacement always  satisfies  $\Vert g \Vert \leq \ell(g)$; however on a GCB-space $(X,\sigma)$ we always have   $  \Vert g \Vert= \ell(g)$  for $\sigma$-invariant isometries, see \cite{Des15}, Proposition 7.1.   In particular, if $(X,\sigma)$ is a $\delta$-hyperbolic GCB-space all parabolic $\sigma$-isometries have zero minimal displacement  (notice that this is false for arbitrary convex spaces, cp. \cite{Wu18}) and   $\ell(g) > 0$ if and only if $g$ is of hyperbolic type.   
Moreover by Lemma \ref{lemma-center}  every elliptic $\sigma$-isometry $g$ of $X$ has a fixed point (the center of any bounded orbit $g^n x$ of $g$, which is clearly invariant by $g$); reciprocally, every isometry with a fixed point is clearly elliptic.
  We can therefore restate the classification of isometries of a proper, Gromov-hyperbolic GCB-space $(X,\sigma)$ as follows:

\begin{itemize}
	\item a $\sigma$-isometry $g$ is elliptic if and only if $\ell(g) = 0$ and the value of minimal displacement  is attained for some $x \in X$; 
	\item a $\sigma$-isometry $g$ is parabolic if and only if $\ell(g) = 0$ and the minimal displacement  is not attained; 
	\item a $\sigma$-isometry $g$ is hyperbolic if and only if $\ell(g) > 0$ and the minimal displacement is attained.  
	\end{itemize}

  It is well known that on a Busemann  space any hyperbolic isometry has an axis, i.e. a geodesic  joining $g^-$ to $g^+$ on which $g$ acts as a translation by $\ell(g)$. \linebreak
  This is also true for GCB-spaces, as proved in \cite{Des15}, Proposition 7.1. \linebreak
  Less trivially (compare with Example 7.7 of \cite{Des15}),  a hyperbolic isometry $g$ of a Gromov-hyperbolic GCB-space  also has a {\em $\sigma$-axis}, that is an axis which is a   $\sigma$-geodesic:
  
\begin{lemma}
	Let $(X,\sigma)$ be a proper, $\delta$-hyperbolic \textup{GCB}-space and let $g$ be a $\sigma$-invariant isometry of hyperbolic type. Then there exists a $\sigma$-axis of $g$. 
\end{lemma}
\begin{proof}
	We consider the minimal set Min$(g)$, that is the subset of points where $d_g$ attains its minimum. This is a $\sigma$-convex, closed and $g$-invariant subset. \linebreak
	Define $\varphi_g\colon X \to X$ as the midpoint of   the $\sigma$-geodesic segment $ [g^{-1}x,gx]$ from $g^{-1}x$ to $gx$.  By Lemma 7.5 of \cite{Des15} we know that there exists a $\sigma$-axis of $g$ if and only if there is some $x\in \text{Min}(g)$ such that $d(x,\varphi_g(x)) = 0$. \linebreak
	Moreover, by Proposition 7.6 of \cite{Des15},  there always exists a sequence of points $x_k\in \text{Min}(g)$ such that $\lim_{k\to +\infty}d(x_k,\varphi_g(x_k)) = 0$. \\
	Fix now an arbitrary axis $\gamma$ of $g$. 
	By  Proposition 7.1 in \cite{Des15}, for every $k$ there exists an axis of $g$ passing through $x_k$, so by Lemma \ref{parallel-geodesics} we deduce that there exists $t_k\in \mathbb{R}$ such that $d(x_k, \gamma(t_k))\leq 8\delta$. 
	Since $\gamma$ is a geodesic we have $\vert t_k - d(\gamma(0), x_k)\vert \leq 8\delta$. Taking $n_k = \lceil \frac{t_k}{\ell(g)} \rceil$ we obtain
	$$d(g^{-n_k}x_k, \gamma(0)) \leq d(g^{-n_k}x_0, g^{-n_k}\gamma(t_k)) + d(g^{-n_k}\gamma(t_k), \gamma(0)) \leq 8\delta + \ell(g).$$
	Then, the sequence $(g^{-n_k}x_k)$ belongs  to the closed ball $\overline{B}(\gamma(0),8\delta + \ell(g))$,  so there exists a subsequence converging to some point $x_\infty$. 
	Finally, since $g$ is $\sigma$-invariant, we deduce that $\varphi_g(g^{-n_k}x_k)=g^{-n_k}\varphi_g(x_k)$ and so 
	$$d(g^{-n_k}x_k,\varphi_g(g^{-n_k}x_k)) = d(x_k,\varphi_g(x_k)) \underset{k\to+\infty}{\longrightarrow} 0,$$
	which implies that $d(x_\infty, \varphi_g(x_\infty)) = 0$, hence the existence of a $\sigma$-axis.
\end{proof}

 \subsection{The Margulis domain}  
We are interested in the sublevel sets of $d_g$.
Given $\varepsilon > 0$ the subset 
$$M_\varepsilon(g) = \lbrace x \in X \text{ s.t. } d(x,gx)\leq \varepsilon \rbrace$$
is called the {\em Margulis domain} of $g$ with displacement $\varepsilon$.
As $d_g$ is $\sigma$-convex,  the Margulis domain is a  closed and $\sigma$-convex subset of $X$. Finally, let 
$$\textup{Min}(g)=M_{\ell(g)} (g)$$
be  the subset of points of $X$ where $d_g$ attains its minimum (which is empty for a parabolic isometry $g$).
\begin{lemma}
	\label{starlike-margulis}
	The Margulis domain   $M_\varepsilon(g)$ of	 any $\sigma$-invariant isometry $g$ of $X$, if non-empty, is starlike with respect to any point $z\in \textup{Fix}_\partial(g) \cup \textup{Min}(g)$. 
\end{lemma}

\begin{proof}
	Fix a point $x\in M_\varepsilon(g)$
	and   $z\in \text{Fix}_\partial(g)$. Assume that the $\sigma$-geodesic ray $[x,z]$   is not contained in $M_\varepsilon(g)$. Then there exists a point $y \in [x,z]$ such that $d(y,gy) \geq \varepsilon + \eta$ for some $\eta > 0$. Let  $L=d(x,y)$ and consider the points $y_n$  along $[x,z]$ at distance $nL$ from $x$. By $\sigma$-convexity of the displacement function,  we have $d(y_n, g y_n) \geq \varepsilon + n\eta.$ We observe that the points $gy_n$ belong to the $\sigma$-geodesic ray $[gx,gz]$, defining the point $gz$ of the boundary. The two rays $[x,z]$ and $[gx,gz]$ are not parallel, hence $gz \neq z$ which is a contradiction since $z\in \text{Fix}_\partial (g)$. The case where $z\in \text{Min}(g)$ follows directly from the $\sigma$-convexity of the displacement function and the minimality of $z$. 
\end{proof}
The {\em generalized Margulis domain} of $g$ at level $\varepsilon$ is the set

 $$\mathcal{M}_{\varepsilon}(g) = \bigcup_{i\in\mathbb{Z}^\ast} M_{\varepsilon}(g^i).$$
	It  clearly is a  $g$-invariant subset of $X$. We remark that for all $\varepsilon > 0$ the union is finite when $g$ is elliptic or hyperbolic, while it is infinite when $g$ is of parabolic type.

\begin{lemma}
\label{margulis-properties}
 The generalized Margulis domain   ${\mathcal M}_\varepsilon(g)$   
 is $12\delta$-quasiconvex and connected. 
 \end{lemma}
 
 \begin{proof}
	As a consequence of Lemma \ref{starlike-margulis}, the domain $\mathcal{M}_{\varepsilon}(g)$   is starlike with respect to any  $x \in  \textup{Min}(g) \cup \textup{Fix}_\partial (g)$.  
So, by Lemma \ref{starlike-quasiconvex}, it is $12\delta$-quasiconvex.  
 The last assertion is trivial if $g$ is elliptic or hyperbolic: in that case $\mathcal{M}_\varepsilon(g)$ is a finite union of connected sets with a common point.
 If $g$ is parabolic we fix a point $x\in M_\varepsilon(g)$ and any $y \in \mathcal{M}_\varepsilon(g)$, so $y\in M_\varepsilon(g^i)$ for some $i\neq 0$. Since $\ell(g) = 0$ we can take a point $x' \in M_{\varepsilon/\vert i \vert}(g)$. By $\sigma$-convexity we have that the $\sigma$-geodesic segment $[x,x']$ is entirely contained in $M_\varepsilon(g)$. Moreover $d(x', g^ix')\leq \varepsilon$, so the $\sigma$-geodesic segment $[y,x']$ is contained in $M_\varepsilon(g^i)$. As a consequence the curve $[x,x'] \cup [x',y]$ is contained in $\mathcal{M}_\varepsilon(g)$. We conclude that $\mathcal{M}_\varepsilon(g)$ is connected since every of its points can be connected to the fixed point $x$. 
 \end{proof}


One of the key ingredients in the proof of   Theorem \ref{maintheorem} and in the applications are the following  lower and upper uniform estimates of the distance between the boundaries of two different generalized Margulis domains. 
Hyperbolicity   in used only in the upper estimate, while the packing condition is essential to both.
%
%
%

\begin{prop}
 \label{lemma-lowerdistance}
  Let $(X,\sigma)$ be a $\delta$-hyperbolic \textup{GCB}-space that is $P_0$-packed at scale $r_0$, and let    $0 < \varepsilon_1 \leq \varepsilon_2$. 
  Let $g$ be any $\sigma$-invariant non-elliptic isometry,   $x \in  X \setminus \mathcal{M}_{\varepsilon_2}(g)$ and assume   $ \mathcal{M}_{\varepsilon_1}(g) \neq \emptyset$. Then:
\begin{itemize}
\item[(i)]	 $\displaystyle d(x,  \mathcal{M}_{\varepsilon_1}(g)) \geq \frac12 (\varepsilon_2-\varepsilon_1)$; 
\vspace{-4mm}

\item[(ii)] if $\varepsilon_2 \leq r_0$ then  $\displaystyle d(x, \mathcal{M}_{\varepsilon_1}(g)) > L_{\varepsilon_2} (\varepsilon_1)  =    \frac{ \log    \left(   \frac{2}{\varepsilon_1} -1  \right)  }{  2 \log(1+P_0) }   \cdot  \varepsilon_2- \frac12 $.\\
 \end{itemize} 
\end{prop}

 Notice that  the estimate (ii) is significative only for $\varepsilon_2$ small enough, and has a different geometrical meaning from (i): it says that the  $\mathcal{M}_{\varepsilon_2}(g)$ contains a large ball around any point $x \in  \mathcal{M}_{\varepsilon_1}(g)$,  of radius which is larger and larger as $\varepsilon_1 $ tends to zero.

\begin{proof}
The first estimate is simple and does not need any additional condition on the metric space $X$.  
Let $\bar x$ be a projection of $x$ on the closure $\overline{ \mathcal{M}_{\varepsilon_1}(g)} $ of the generalized Margulis domain. By definition for all $\eta >0$ there exists some nontrivial power  $g_\eta$   of $g$ such that $d(\bar x, g_\eta \bar x) \leq \varepsilon_1 + \eta$. So:
$$  \varepsilon_2 \leq d(x,  g_\eta x) \leq d(x, \bar x) + d(\bar x,g_\eta \bar x) + d(g_\eta \bar x , g_\eta x) 
\leq 2 d(x,\mathcal{M}_{\varepsilon_1}(g)) +\varepsilon_1 + \eta. $$
The estimate follows from the arbitrariness of $\eta$.\\
Let us now  prove (ii). Let again $\bar x    \in    \overline{\mathcal{M}_{\varepsilon_1}(g)} $ with $d(x,\bar x ) \!  =d(x,  \mathcal{M}_{\varepsilon_1}(g)) \!  =R$. For any $\eta >0$ let $g_\eta$ be some nontrivial power of $g$ satisfying  $d(\bar x,g_\eta \bar x)\leq \varepsilon_1 + \eta$.
Then again $$d(x, g_\eta^k x ) \leq 2R + d(\bar x, g_\eta^k \bar x ) \leq 2R+ |k| (\varepsilon_1 + \eta) $$
 so   $d(x, g_\eta^k x ) \leq 2R +1$ for all $k$ such that $| k | \leq 1/(\varepsilon_1 + \eta)$.
 Therefore we have at least $n(\varepsilon_1, \eta) =1 + 2 \lfloor 1/(\varepsilon_1 + \eta)  \rfloor$ points in the orbit $\Gamma x$ inside the ball $\overline{B} (x, 2R+1)$.
 We deduce that if  $n(\varepsilon_1, \eta)>  \text{Pack}(2R +1, \varepsilon_2)$  two of these points stay at distance less than $\varepsilon_2$ one from the other, which implies that  $x \in X_{\varepsilon_2}$, a contradiction.
Therefore, 
$$n(\varepsilon_1, \eta)  = 1 + 2 \lfloor 1/ (\varepsilon_1 + \eta) \rfloor  \leq    \text{Pack}(2R +1, \varepsilon_2) \leq P_0(1+P_0)^{\frac{2R+1}{ \varepsilon_2 } - 1}$$
by Proposition \ref{packingsmallscales} (since  $\varepsilon_2 \leq r_0$), 
 which implies that $R=d(x,y)$ is greater than   the function $L_{\varepsilon_2} (\varepsilon_1)$ in (ii) by the arbitrariness of $\eta$. 
\vspace{-4mm}

\end{proof}

\vspace{3mm}
The upper bound is more tricky:   
 
\begin{prop}
	\label{Margulis-estimate}
	Let $(X,\sigma)$ be a $\delta$-hyperbolic \textup{GCB}-space that is $P_0$-packed at scale $r_0$ and let $0 < \varepsilon_1 \leq \varepsilon_2$. Then there exists $K_0$, only depending on $P_0,r_0,\delta,\varepsilon_1$ and $\varepsilon_2$, such that for every non-elliptic, $\sigma$-invariant isometry $g$ of $X$ 
such that	  $ \mathcal{M}_{\varepsilon_1}(g) \neq \emptyset$
	it holds:
\vspace{-3mm}	
	
	$$\sup_{x\in \mathcal{M}_{\varepsilon_2}(g)}d(x, \mathcal{M}_{\varepsilon_1}(g))\leq  K_0.$$
\end{prop}
\noindent 

\begin{proof}
Let $x\in \mathcal{M}_{\varepsilon_2}(g)$, so by definition there exists $i_0$ with $d(x,g^{i_0}x)\leq\varepsilon_2$. \linebreak If $x\in \overline{\mathcal{M}_{\varepsilon_1}(g)}$ there is nothing to prove. Otherwise we can find a point $\bar{x}$ of $\partial\mathcal{M}_{\varepsilon_1}(g)$ such that
$d(x,\bar{x}) = d(x, \mathcal{M}_{\varepsilon_1}(g)).$ 
We set $\tau = \max\lbrace \varepsilon_1 ,\delta \rbrace$ and $N_0 =\text{Pack}\left(
42\tau, \frac{\varepsilon_1}{2}\right)$, which is a number depending only on $P_0,r_0,\delta$ and $\varepsilon_1$, by Proposition \ref{packingsmallscales}. 

\vspace{2mm}	
\noindent {\em Step 1: we prove that there exists an integer $k \leq 2^{N_0 +1}$ such that
	\begin{equation}\label{firststep}
		d(\bar{x},g^{k \cdot i_0}\bar{x}) > 
		42\tau \text{ and } d(x,g^{k \cdot i_0}x) \leq K:=\varepsilon_2 + 
		84\tau + 2(N_0 + 1)\delta
	\end{equation}
}

\noindent If this was not true, then for all  $k\leq 2^{N_0 +1}$ such that  
$d(x,g^{k \cdot i_0}x) \leq K$ we would have
$d(\bar{x},g^{k \cdot i_0}\bar{x})\leq 
42\tau$. 
Let $p_0$ be the largest integer such that $d(\bar{x},g^{2^p \cdot i_0}\bar{x}) \leq 
42\tau$ for all $0 \leq p\leq  p_0$.
We affirm that $p_0 \geq N_0+1$.\\
Actually,   $p_0 \geq 0$ because  $d(x,g^{i_0}x) \leq \varepsilon_2  \leq K$, hence $d(\bar{x}, g^{i_0}\bar{x}) \leq 
42 \tau$ by assumption. 
Also, by Lemma \ref{lemma-powers}, we get
\vspace{-6mm}

$$d(x,g^{2^i \cdot i_0}x) \leq d(x,g^{2^{i-1} \cdot i_0}x) + \ell(g^{2^{i-1}  \cdot i_0}) + 2\delta = d(x,g^{2^{i-1} \cdot i_0}x) + 2^{i-1}   i_0\ell(g) + 2\delta.$$

\vspace{-2mm}
\noindent	and, iterating,   
\vspace{-6mm}

$$d(x,g^{2^p \cdot i_0}x) \leq d(x,g^{i_0}x) + (2^p - 1)i_0\ell(g) + 2p\delta \leq (2^p - 1)i_0\ell(g) + 2p\delta + \varepsilon_2$$
for every $0\leq p \leq N_0 +1$. 
So, if $p_0 \leq N_0$ we would have:
\vspace{-2mm}

\begin{equation*}
	\begin{aligned}
		d(x,g^{2^{(p_0+1)} \cdot i_0}x)&  \leq  d(x,g^{2^{p_0} \cdot i_0}x) + 2^{p_0} \cdot i_0\ell(g) + 2\delta \\
		&  \leq  d(x,g^{2^{p_0} \cdot i_0}x)  +    d(\bar{x},g^{2^{p_0} \cdot i_0}\bar{x})+ 2\delta \\
		& \leq (2^{p_0} - 1)i_0\ell(g) + 2p_0\delta + \varepsilon_2 + 
		42 \tau  + 2\delta\\
		&\leq 
		84 \tau +  2(p_0 +1)\delta + \varepsilon_2 < K
	\end{aligned}
\end{equation*}
since $2^{p_0}i_0\ell(g)\leq d(\bar{x},g^{2^{p_0} \cdot i_0}\bar{x}) \leq 
42\tau$ by definition. Hence by assumption $d(\bar{x},g^{2^{(p_0  + 1)}\cdot i_0}\bar{x})\leq 
42\tau$ and $p_0$ would not be maximal.\\	
Moreover, since $\bar{x}$ is in the boundary  $\partial\mathcal{M}_{\varepsilon_1}(g)$, then
\vspace{-3mm}

$$\inf_{i\in\mathbb{Z}^\ast} d(\bar{x}, g^i\bar{x}) \geq \varepsilon_1.$$
Indeed if $d(\bar{x},g^i\bar{x})=\varepsilon_1 -\eta $ for some  $i\in\mathbb{Z}^\ast$ and some $\eta>0$,   
then it is easy to show that for any $y\in B(\bar{x},\eta/2)$ we would have  $d(y,g^iy)< \varepsilon_1$; hence, $B(\bar{x},\eta/2) \subset \mathcal{M}_{\varepsilon_1}(g)$ and   $\bar{x}$ would not belong to $\partial\mathcal{M}_{\varepsilon_1}(g)$.\\
Then the points $g^{2^p \cdot i_0}\bar{x}$, for $p=1,\ldots, N_0 + 1$, are $\varepsilon_1$-separated. 
But, as they belong all to the  ball $B(\bar{x}, 
42\tau)$,  they should be at most $N_0$ and this is a contradiction. This proves the first step.
\vspace{3mm}

\noindent {\em Step 2: for any $k_0 \leq 2^{N_0 +1}$ satisfying the conditions  (\ref{firststep}),  we have:}
	\begin{equation}\label{secondstep}
		d(x,g^{k_0 \cdot i_0}x)\geq d(x,\bar{x}) + d(g^{k_0 \cdot i_0}x, g^{k_0 \cdot i_0}\bar{x})
\end{equation}
\noindent Indeed let us write  $y=g^{k_0 \cdot i_0}x$ and $\bar{y} = g^{k_0 \cdot i_0}\bar{x}$. 
By definition the point   $\bar x$  satisfies 
$d(x,\bar{x})= d(x, \mathcal{M}_{\varepsilon_1}(g))$;
so, from the $12\delta$-quasiconvexity of $\mathcal{M}_{\varepsilon_1}(g)$ (Lemma \ref{margulis-properties}), we deduce that   
$$d(x,[\bar{x},\bar{y}]) \geq d(x, \overline{B}(\mathcal{M}_{\varepsilon_1}(g), 12\delta)) = d(x, \bar{x}) - 12\delta.$$
Moreover from the Projection Lemma \ref{projection} we have 
$d(x,[\bar{x},\bar{y}]) \leq (\bar{x},\bar{y})_x + 4\delta.$
Combining these estimates and expanding the Gromov product  we obtain
\begin{equation}\label{dxbary}d(x,\bar{y})\geq d(x,\bar{x}) + d(\bar{x}, \bar{y}) - 
	20\delta.
\end{equation}
Similarly,  using that  $d(y,  \bar{y}) = d(y, \mathcal{M}_{\varepsilon_1}(g))$ (as $\mathcal{M}_{\varepsilon_1}(g)$ is $g$-invariant),  we obtain 
\begin{equation}\label{dybarx} d(y,\bar{x}) \geq d(y,\bar{y}) + d(\bar{y},\bar{x}) - 
	20\delta.\end{equation}
Adding these last two inequalities and using that $d(\bar{x},\bar{y})> 
42\tau \geq 42 \delta$ we deduce
$$d(x,\bar{y}) + d(y,\bar{x}) > d(x,\bar{x}) + d(y,\bar{y}) + 2\delta. $$
\noindent 	Therefore applying the four-points condition \eqref{four-points-condition} to $x,\bar{y},\bar{x},y$ we find
\begin{equation*}
	\begin{aligned}
		d(x,\bar{y}) + d(\bar{x},y)&\leq \max \lbrace d(x,\bar{x}) + d(y,\bar{y}); d(x,y) + d(\bar{x},\bar{y}) \rbrace + 2\delta \\
		&= d(x,y) + d(\bar{x},\bar{y}) + 2\delta
	\end{aligned}
\end{equation*}

\noindent	It follows:
\begin{equation*}
	\begin{aligned}
		d(x,y) &\geq d(x,\bar{y}) + d(\bar{x},y) - d(\bar{x},\bar{y}) - 2\delta \\
		&\geq d(x,\bar{x}) + d(\bar{x},\bar{y}) - 
		20\delta + d(y,\bar{y}) + d(\bar{y},\bar{x}) - 
		20\delta - d(\bar{x},\bar{y}) - 2\delta \\
		&\geq d(x,\bar{x}) + d(y,\bar{y}),
	\end{aligned}
\end{equation*}
\noindent where we have used again (\ref{dxbary}), (\ref{dybarx})  and that  $d(\bar{x},\bar{y}) >  
42\tau \geq 
42\delta$ (the first condition in (\ref{firststep})).
Moreover, the second condition in (\ref{firststep}) now yields
$$d(x,\bar{x}) + d(y,\bar{y}) \leq K$$
The conclusion follows observing that $d(y,\bar{y}) = d(x,\bar{x})$, so that
$$  d(x, \mathcal{M}_{\varepsilon_1}(g)) \leq  d(x,\bar{x}) \leq K/2$$
which is the announced bound,  depending only on $P_0, r_0, \delta, \varepsilon_1$ and $\varepsilon_2$.
\end{proof}

 \enlargethispage{1\baselineskip}
 
The distance between two (non generalized)  Margulis domains  of a non-elliptic isometry can also be bounded uniformly in $\delta$-hyperbolic GCB-spaces, but the bound is not explicit. We will use this estimate to study the limit of sequences of isometries in Section \ref{sec-compactness}.
\begin{prop}
	\label{Margulis-nonexplicit}
For  any given $\varepsilon, \delta \!>\!0$ there exists   $c(\varepsilon, \delta)>\!0$  satisfying the following property.  Let $(X,\sigma)$ be a proper, $\delta$-hyperbolic \textup{GCB}-space $X$ and $g$ be a non-elliptic, $\sigma$-invariant isometry of $X$ with $M_{\varepsilon}(g) \neq \emptyset$: then, for all $x\in X$ it holds
$$d(x,gx)\geq c(\varepsilon,\delta) \cdot d(x,M_{\varepsilon}(g)).$$
In particular for all $0<\varepsilon_1 \leq \varepsilon_2$ we get
\vspace{-3mm}

$$\sup_{x\in M_{\varepsilon_2}(g)} d(x, M_{\varepsilon_1}(g)) \leq \frac{\varepsilon_2}{c(\varepsilon_1,\delta)}=: K_1(\varepsilon_1, \varepsilon_2, \delta).$$
\end{prop}
\begin{proof}
Suppose by contradiction that the thesis is not true. Then for every $n\in \mathbb{N}$ there exist a proper, $\delta$-hyperbolic GCB-space $(X_n,\sigma_n)$,  a non-elliptic, $\sigma_n$-invariant isometry $g_n$ of $X_n$ such that $M_{\varepsilon}(g_n)\neq \emptyset$ and a point $x_n \in X_n$ such that
$$0<d(x_n,g_nx_n)\leq \frac{1}{n}d(x_n,M_{\varepsilon}(g_n)),$$
(where the first inequality follows from the assumption that $g_n$ is not elliptic). Observe that this implies that $x_n\notin M_\varepsilon(g_n)$ for every $n$.\\
For every $n$,  let $y_n$ be the projection of $x_n$ on the $\sigma_n$-convex set $M_\varepsilon(g_n)$.\linebreak
We fix a non-principal ultrafilter $\omega$ and we consider the ultralimit $X_\omega$ of the sequence $(X_n,y_n)$.
The space  $X_\omega$ is $\delta$-hyperbolic  (the stability of the  hyperbolicity constant   follows from \eqref{hyperbolicity}) and admits a structure   $\sigma_\omega$  of GCB-space, by Lemma \ref{bicombing-limit}. The sequence of isometries $(g_n)$ is admissible, as for every $n$ it holds $d(g_n y_n, y_n) = \varepsilon$, so it defines a limit isometry $g_\omega = \omega$-$\lim g_n$ of $X_\omega$ (cp. Proposition A.5 of \cite{CavS20}); moreover, the isometry $g_\omega $ is  $\sigma_\omega$-invariant, as follows from  the definition of the limit structure $\sigma_\omega$  and  of $g_\omega$.\\
We consider the sequence of $\sigma_n$-geodesic segments $\gamma_n=[y_n,x_n]$. We claim  that this sequence converges to a $\sigma_\omega$-geodesic ray of $X_\omega$: by Proposition A.5 and Lemma A.6 of \cite{CavS20} it is enough to show that $\omega$-$\lim d(y_n,x_n)=+\infty$.\linebreak 
Since $y_n \in M_{\varepsilon}(g_n)$ and $x_n$ is not in this set we deduce
\vspace{-3mm}

$$d(y_n,x_n)\geq n\cdot d(x_n,g_n x_n) \geq n \cdot \varepsilon$$
for $\omega$-a.e.$(n)$, so $\gamma_\omega=\omega$-$\lim   \gamma_n$ is a $\sigma_\omega$-geodesic ray.\\
We claim now  that    $d(g_\omega\gamma_\omega(T),\gamma_\omega(T)) \leq \varepsilon$ for all $T\geq 0$.
For this,  fix $T\geq 0$. \linebreak We   showed  before that  the segment $\gamma_n$ is defined at time $T$,  for $\omega$-a.e. $n$.\linebreak  Clearly, the sequence of points $(\gamma_n(T))$ defines the point $\gamma_\omega(T)$ of $X_\omega$. \linebreak
Now, for every arbitrary $\eta>0$, we will find an upper bound for all the integers $n$ such that 
$d(\gamma_n(T), g_n \gamma_n(T)) > \varepsilon + \eta$.
By the arbitrariness of $\eta$, this   will imply the claim.
Let $m\in \mathbb{N}$ be the smallest integer such that $mT \geq d(x_n,y_n)$ and, for every integer $k\in \lbrace 0, \ldots, m-1 \rbrace$ consider the point $\gamma_n(kT)$. From the $\sigma_n$-convexity of the displacement function it holds 
$d(\gamma_n(kT), g_n \gamma_n(kT)) > \varepsilon + k\eta.$
In particular, again by $\sigma_n$-convexity of the displacement function, we get
\begin{equation*}
	\begin{aligned}
		mT \geq 
		n\cdot d(x_n,g_n x_n) &\geq n \cdot d(\gamma_n((m-1)T), g_n \gamma_n((m-1)T)) \\
		&\geq n(\varepsilon + (m-1)\eta).
	\end{aligned}
\end{equation*}
so $n\leq \frac{m}{m-1}\frac{T}{\eta} \leq \frac{2T}{\eta} =: n_\eta$, and the claim is proved. \\
This imples that $\gamma_\omega^{+}$ is a fixed point at infinity of $g_\omega$. Therefore if $g_\omega$ is not elliptic one of the two sequences $\lbrace g_\omega^k y_\omega \rbrace$, $\lbrace g_\omega^{-k}y_\omega\rbrace$, where $k\in \mathbb{N}$, must converge to $\gamma_\omega^+$. 
Let now $k\in \mathbb{Z}$ be fixed. Clearly $g_\omega^ky_\omega = \omega$-$\lim g_n^k y_n$.  Moreover   the point $g_n^ky_n$ belongs to $M_\varepsilon(g_n)$ for $\omega$-a.e.$(n)$, and since $y_n$ is the projection of $x_n$ on this $\sigma_n$-convex set we get $(x_n,g_n^ky_n)_{y_n} \leq 4\delta$ by \eqref{eqprojection}. Since this is true for $\omega$-a.e.$(n)$ we deduce that  $(\gamma_\omega^+, g_\omega^ky_\omega)_{y_\omega}\leq 5\delta$. Since this  holds for every $k\in \mathbb{Z}$,   the sequences $\lbrace g_\omega^k y_\omega \rbrace$ and $\lbrace g_\omega^{-k} y_\omega \rbrace$ defined for $k\in \mathbb{N}$ do not converge to $\gamma_\omega^+$, which implies $g_\omega$ is elliptic. This means there is a point $z_\omega = \omega$-$\lim z_n$ such that $d(z_\omega, g_\omega z_\omega) < \frac{\varepsilon}{2}$ and so $d(z_n, g_nz_n) < \frac{\varepsilon}{2}$ for $\omega$-a.e.$(n)$. Moreover there exists $L\geq 0$ such that $d(y_n,z_n)\leq L$ for $\omega$-a.e.$(n)$. Let us consider the $\sigma$-geodesic segment $[z_n,x_n]$ and let us denote by $w_n$ the point along this geodesic with $d(w_n,g_nw_n)=\varepsilon$. Clearly we have $d(x_n,w_n)\geq d(x_n,y_n)$ and $d(w_n,z_n)\leq L$. Moreover by convexity along $[z_n,x_n]$ we get
$$\varepsilon + \frac{d(x_n,w_n)}{L}\frac{\varepsilon}{2} \leq d(x_n,g_nx_n) \leq \frac{1}{n}d(x_n,y_n)\leq \frac{1}{n}d(x_n,w_n)$$
implying $\frac{\varepsilon}{2L}\leq \frac{1}{n}$ which is clearly impossible if $n$ is big enough, a contradiction.
\end{proof}

\section{Free subgroups and sub-semigroups}
\noindent In this section we will prove Theorem \ref{maintheorem}.

\vspace{1mm}
\noindent First remark that it is possible to assume   $\ell(a)=\ell(b)=:\ell$. \\
Indeed we can take $b'=bab^{-1}$ and $\langle a,b'\rangle$ is still a discrete and non-elementary group, which moreover is  torsion-free if $\langle a,b \rangle$ was torsion-free. Futhermore $\ell(b')=\ell(a)$ and the length of $b'$ as a word of $a,b$ is $3$.

The proof  will then be divided into two cases: 
$ \ell \leq \frac{\varepsilon_0}{3}$ and $\ell >  \frac{\varepsilon_0}{3}$,   where $\varepsilon_0=\varepsilon_0 (P_0, r_0)$  is the Margulis constant 
given by Corollary \ref{margulis-constant}.
The proof in the first case does not use the torsionless assumption and produces a true free subgroup; 
it heavily  draws, in this case, from techniques introduced in \cite{DKL18} and \cite{BCGS17}.   
On the other hand, in the case where $ \ell > \frac{\varepsilon_0}{3}$ the proofs of the  statements (i) and (ii) diverge. 
In this last case producing a free sub-semigroup  is quite standard, 
while producing a free subgroup  is much more complicated and for this we need to properly modify the argument of \cite{DKL18} to use it in our context.

\subsection{Proof of Theorem \ref{maintheorem}, case $\ell \leq \frac{\varepsilon_0}{3}$.}
{\em We  assume here   $a,b$ non-elliptic, $\sigma$-isometries  with 
	$\ell(a) = \ell(b) = \ell \leq   \frac{\varepsilon_0}{3}$ 
	of a $\delta$-hyperbolic \textup{GCB}-space $(X,\sigma)$ which is $P_0$-packed at scale $r_0$, and  that  the group  $\langle a,b \rangle$ is non-elementary and discrete.
	In particular $a$ and $b$ are both parabolic or both hyperbolic.}
\vspace{1mm}

In order to find a free subgroup in this case we will use a criterion  which  is the generalization of  Proposition 4.21 of \cite{BCGS17}    to non-elliptic isometries.
Recall that, given two isometries $a,b \in \text{Isom}(X)$,   the {\em Margulis constant of the couple} $(a,b)$  is the number
$$L(a,b) = \inf_{x\in X} \inf_{  (p,q)\in \mathbb{Z}^*\times \mathbb{Z}^*} \Big\lbrace \max\lbrace d(x, a^px), d(x,b^qx) \rbrace\Big\rbrace.$$

\begin{prop}
	\label{free-group-gallot}
	Let $a,b$ be two non-elliptic $\sigma$-isometries of $X$ of the same type such that $\langle a,b \rangle$ is discrete and non-elementary. If $a,b$ satisfy
	$$L(a,b) > \max\lbrace \ell(a),\ell(b) \rbrace + 56\delta$$ 
	\noindent	then $\langle a,b \rangle$ is a free group.
\end{prop}

\noindent The proof   closely follows  (mutatis mutandis) the proof of Proposition 4.21  \cite{BCGS17}; we report it here in the case of parabolic isometries for the sake of clarity and completeness; the case where one isometry is hyperbolic and the other one is parabolic can be proved in a similar way but we do not need it for our purposes. 

\noindent As a first step notice that we can control the Margulis constant  $L(a,b)$   by the distance of the corresponding generalized Margulis domains:

\begin{lemma}
	\label{margulis-domains-distance}
	Let $a,b$ two $\sigma$-isometries of $X$ and let $L > 0$: 	
	\begin{itemize}
			\item[(a)] if $d(\mathcal{M}_{L}(a), \mathcal{M}_{L}(b))>0$ then $L(a,b)\geq L$;
		\item[(b)] conversely,  if $L(a,b) > L$ then $\overline{\mathcal{M}_{L}(a)} \cap \overline{\mathcal{M}_{L}(b)} = \emptyset.$
	\end{itemize}
\end{lemma}
\begin{proof}
	If $d(\mathcal{M}_{L}(a), \mathcal{M}_{L}(b))>0$ then $\mathcal{M}_{L}(a) \cap \mathcal{M}_{L}(b) = \emptyset$.   In particular for every $x\in X$ and for all $p,q \in \mathbb{Z}^\ast$ we have $d(x,a^px)>L$ or $d(x,b^qx)>L$. 
	Taking the infimum over $x \in X$ we get $L(a,b) \geq L$, proving (a).\\
	Suppose now $L(a,b) > L$ and $\overline{\mathcal{M}_{L}(a)} \cap \overline{\mathcal{M}_{L}(b)} \neq \emptyset$. Take $x$ in the intersection. In particular  $\forall \eta >0$ there exist  $ x_\eta \in \mathcal{M}_{L}(a)$, $y_\eta \in \mathcal{M}_{L}(b)$ such that for some $(p_\eta,q_\eta)\in \mathbb{Z}_* \times \mathbb{Z}_*$
	$$d(x,x_\eta)<\eta, \qquad d(x,y_\eta) < \eta, \qquad d(x_\eta, a^{p_\eta} x_\eta) \leq L, \qquad d(y_\eta, b^{q_\eta} y_\eta) \leq L.$$
	By the triangle inequality we get $d(x,a^{p_\eta}x), d(x,b^{q_\eta}x) \leq L + 2 \eta$. As this is true  for every  $\eta >0$ we get $L(a,b)\leq L$. This contradiction proves (b).
\end{proof}

\begin{proof}[Proof of Proposition \ref{free-group-gallot}]
	We will assume that $a,b$ are parabolic isometries, the hyperbolic case being covered in \cite{BCGS17}. 
	The aim is to show that there exists $x \in X$ such that 
	\begin{equation}\label{condizionefree}
		d(a^px, b^qx) > \max \{ d(x,a^px), d(x,b^qx) \} + 2\delta  \hspace{5mm} \forall (p,q)\in \mathbb{Z}^*\times\mathbb{Z}^* ;
	\end{equation}
	this will  imply that  $a$ and $b$ are in Schottky position  by Proposition 4.6 of \cite{BCGS17},   so the group $\langle a, b\rangle$ is free.\\
	\noindent With this in mind,  choose $L_0$ and   $0<\varepsilon <\delta$ with
	$L(a,b) > L_0 > 56 \delta + 2 \varepsilon$, \linebreak and set  $\ell_0=\delta + \varepsilon$.
	Since $L(a,b) > L_0$ then $\overline{\mathcal{M}_{L_0}(a)} \cap \overline{\mathcal{M}_{L_0}(b)} = \emptyset$ by Lemma \ref{margulis-domains-distance}. Moreover the two Margulis domains are non-empty since $L_0 > 0$. \\
	Now  fix points $x_0 \in \mathcal{M}_{\ell_0}(a)$ and $y_0 \in \mathcal{M}_{\ell_0}(b)$ which $\frac{\varepsilon}{2}$-almost realize the distance between the  two   generalized  Margulis domains, that is:
	\vspace{-2mm}
	
	$$d(x_0, y) \geq d(x_0,y_0) - \frac{\varepsilon}{2} \hspace{5mm} \forall y \in \mathcal{M}_{\ell_0}(b)$$
	\vspace{-2mm}
	$$d(y_0, x) \geq d(y_0,x_0) - \frac{\varepsilon}{2} \hspace{5mm} \forall x \in \mathcal{M}_{\ell_0}(a).$$ 
	Then we can find a point $x\in [x_0,y_0]$ such that: 
	\vspace{-3mm}
	
	\begin{equation}\label{desired}
		d(x,a^px)>L_0 \textup{ and } d(x,b^qx)>L_0 \hspace{5mm}\forall (p,q)\in \mathbb{Z}^*\times\mathbb{Z}^*.
	\end{equation}
	
	\noindent Indeed the sets $\overline{\mathcal{M}_{L_0}(a)}$ and $\overline{\mathcal{M}_{L_0}(b)}$ are non-empty, closed and disjoint; moreover the former  contains $x_0$  and the latter contains $y_0$. Then their union cannot cover the whole geodesic segment $[x_0,y_0]$. Any $x$ on this segment which does not belong  to 
	$\overline{\mathcal{M}_{L_0}(a)} \cup \overline{\mathcal{M}_{L_0}(b)}$   satisfies our requests. \linebreak
	As $x_0$ and $a^px_0$ belong to $\mathcal{M}_{\ell_0}(a)$   and   $x \in X \setminus \mathcal{M}_{L_0}(a)$ we deduce  by Lemma \ref{lemma-lowerdistance} that
	\vspace{-6mm}
	
	\begin{equation}\label{a^p}
		d(x,x_0) 
		\geq d(x,\mathcal{M}_{\ell_0}(a)) - \frac{\varepsilon}{2}
		\geq \frac{L_0 - \ell_0}{2} - \frac{\varepsilon}{2}  >27\delta \hspace{5mm} \forall p\in \mathbb{Z}^* 
	\end{equation}
	
	\vspace{-1mm}
	\noindent   (notice that  $x \in [x_0,y_0]$  and $x_0$ is a $\frac{\varepsilon}{2}$-almost projection of $y_0$ on $\mathcal{M}_{\ell_0}(a)$) and the same is true  for $d(x,a^px_0)$.

	%
	
	\noindent 
	Now choose  points $u \in [x,a^px]$, $u' \in [x,a^p x_0]$ and $u'' \in [x, x_0]$ at distance $11\delta$ from $x$ (notice that this is possible as $d(x,a^px) >   L_0> 56\delta$ and by (\ref{a^p})).\linebreak
	Consider the approximating tripod $f_{\bar \Delta}: \Delta(x,x_0,a^px_0) \rightarrow \bar{\Delta}$ and the preimage $c \in f_{\bar \Delta}^{-1} ({\bar c}) \cap [x_0,a^px_0]$ of its center $\bar c$. By Lemma \ref{margulis-properties} we deduce  that \linebreak $d(c,  \mathcal{M}_{\ell_0}(a))\leq 12 \delta$ and then, by  \eqref{approximation-tripod} and Proposition \ref{lemma-lowerdistance}, 
	that 
	$$( a^px_0, x_0)_x \geq d(x, c) -4\delta \geq d(x,\mathcal{M}_{\ell _0}(a)) -16\delta
	> 11 \delta = d(x,u')=d(x,u'').$$ 
	So $f_{\bar \Delta}(u') \!=\! f_{\bar \Delta}(u'')$ and, by thinness of  $\Delta(x,x_0, a^px_0)$, we get $d(u',u'')\! \leq \!4 \delta$.\linebreak
	Observe that $x_0$ is also a $\frac{\varepsilon}{2}$-almost projection of $x$ on $\mathcal{M}_{\ell_0}(a)$, in the sense that $d(x,x_0) \leq d(x, z) + \frac{\varepsilon}{2}$ for all $z\in \mathcal{M}_{\ell_0}(a)$. Since $a^px_0 \in \mathcal{M}_{\ell_0}(a)$ we deduce that $d(x,a^px_0) \geq d(x,x_0) - \frac{\varepsilon}{2} = d(a^px,a^px_0) - \frac{\varepsilon}{2}$. Moreover $d(x,a^px) > L_0$, so
	$$(a^p x_0, a^p x)_x 
	>\frac 12 \left(L_0 - \frac{\varepsilon}{2}\right) > d(u,x)=d(u',x).$$
	So, again by  thinness of the triangle  $\Delta(x,a^px_0, a^px)$, we have $d(u,u') \leq 4\delta$.
	Therefore $d(u,u'') \leq 8 \delta$.
	One analogously proves that, choosing  $v \in [x,b^qx]$, $v' \in [x,b^q y_0]$ and  $v'' \in [x, y_0]$ at distance $11\delta$ from $x$, we have  $d(v,v'') \leq 8 \delta$.  
	
	\noindent Therefore (as $x$ belongs to the geodesic segment  $[x_0,y_0]$) we deduce that 
	$$d(u,v) \geq  d (u'',x ) +   d(x,v'') - d(u,u'') - d(v,v'') \geq 6\delta.$$ 
	
	\noindent Comparing with the tripod $\bar \Delta'$ which approximates the triangle $\Delta ( a^p x, x, b^q x)$,  we deduce by the $4\delta$-thinness that $f_{\bar \Delta'} (u) \neq f_{\bar \Delta'} (v)$.  
	It follows that $$(a^px, b^q x)_x < d(x,u) = d(x,v)=11\delta.$$  
	One then computes:  
	\begin{equation*}
		\begin{aligned}
			d(a^px,b^qx) & = d(a^px,x) + d(x,b^qx) - 2(a^px, b^q x)_x    \\
			& \geq   \max\{ d(a^px,x)  ,    d(x,b^qx)   \}     
			+ \min\{ d(a^px,x) ,    d(x,b^qx)  \}  -   22\delta \\
			&  \geq    \max\{ d(a^px,x)  ,    d(x,b^qx)   \}  + L_0 - 22\delta  
		\end{aligned}
	\end{equation*}
	which implies (\ref{condizionefree}), by definition of $L_0$. The last inequality is direct consequence of \eqref{desired}.
\end{proof}

We continue the proof of Theorem $\ref{maintheorem}$ in the case $\ell(a)=\ell(b) =\ell \leq \frac{\varepsilon_0}{3}$.\\
Set $b_i = b^i a b^{-i}$. Then $\ell(b_i)=\ell$ for all $i$, and $b_i$ is of the same type of $a$, in particular it is non-elliptic. Moreover for any $i\neq j$ the group $\langle b_i, b_j \rangle$ is discrete (as a subgroup of a discrete group) and non-elementary. \\
Indeed otherwise 
there would exist a subset $F \subset \partial X$ fixed by both $b_i, b_j$,
so 
$b^i\text{Fix}_\partial(a) = \text{Fix}_\partial(b^i a b^{-i})= F = \text{Fix}_\partial(b^j a b^{-j}) = b^{j} \text{Fix}_\partial(a) $.   
This  implies that  $b^{i-j}(F)=F$, hence $F \subseteq \text{Fix}_\partial(b)$
and, as these sets have the same cardinality, they coincide. 
Therefore we deduce that $\text{Fix}_\partial(a) = b^{-i}(F) = \text{Fix}_\partial(b)$, \linebreak
which means that the group $\langle a,b\rangle$ is elementary, a contradiction.\\
Now, since for any $i\neq j$ the group $\langle b_i,b_j \rangle$ is discrete and non-elementary, then by definition of the Margulis constant $\varepsilon_0$ we have
$$\mathcal{M}_{\varepsilon_0}(b_i)\cap \mathcal{M}_{\varepsilon_0}(b_j) = \emptyset.$$
(otherwise there would exist a point $x\in X$ and powers $k,h$ such that $d(x, b_i^{k}x)\leq \varepsilon_0$ and $d(x, b_j^hx)\leq \varepsilon_0$, and $\langle b_i^k, b_j^h \rangle$ would be  virtually nilpotent, hence elementary; but we have just seen that this implies  that $\langle a,b\rangle$ is elementary, a contradiction).
Moreover each Margulis domain $\mathcal{M}_{\varepsilon_0}(b_i)$ is non-empty, since we assumed $\ell = \ell(b_i)  \leq \frac{\varepsilon_0}{3}$.\\
We now need the following 

\begin{lemma}\label{conj}
	Let ${\mathcal B}$ the set of all the conjugates $b_i=b^iab^{-i}$, for $i \in \mathbb{Z}$.\linebreak
	For any fixed $L>0$, the cardinality of  every subset $S$ of ${\mathcal B}$  such that 
	$$d(\mathcal{M}_{\varepsilon_0}(b_{i_h}), \mathcal{M}_{\varepsilon_0}(b_{i_k})) \leq L \quad\quad \forall \,b_{i_h}, b_{i_k} \in S$$
	is bounded from above by a constant $M_0$ only depending on $P_0,r_0,\delta$ and $L$. 
\end{lemma}
\begin{proof}
	Let $S=\{ b_{i_1}, \cdots, b_{i_M}\} \subseteq{\mathcal B}$ satisfying 
	$d(\mathcal{M}_{\varepsilon_0}(b_{i_h}), \mathcal{M}_{\varepsilon_0}(b_{i_k})) \leq L$ 
	for \linebreak all $b_{i_h}, b_{i_k} \in S$. 
	Fix $\eta>0$ and consider the closed $(\frac{L}{2} + \eta)$-neighbourhoods of the   generalized Margulis domains $B_k = \overline{B}(\mathcal{M}_{\varepsilon_0}(b_{i_k}), \frac{L}{2} + \eta)$, for $b_{i_k} \in S$. Since the   domains are starlike, 
	then  $B_k$ is $20\delta$-quasiconvex for all $k$, 
	by Lemma \ref{starlike-quasiconvex}. Moreover $B_h \cap B_k \neq \emptyset$ for every $h,k$. Indeed, chosen points $x_{i_h} \in \mathcal{M}_{\varepsilon_0}(b_{i_h})$ and $x_{i_k} \in \mathcal{M}_{\varepsilon_0}(b_{i_k})$  which $2\eta$-almost realize the distance between these domains, then  the midpoint of any geodesic segment $[x_{i_h} , x_{i_k}]$  is in $B_h\cap B_k$.
	Therefore we can apply Helly's Theorem (Proposition \ref{Helly}) to find a point $x_0$ at distance at most $419\delta$ from each $B_k$. So
	\vspace{-3mm}
	
	$$d(x_0, \mathcal{M}_{\varepsilon_0}(b_{i_k}))\leq R_0 =  419 \delta + \frac{L}{2} + \eta, \quad \textup{ for }k=1,\ldots,M.$$
	Notice that $x_0$ belongs at most  to one of the domains $\mathcal{M}_{\varepsilon_0}(b_{i_k})$, since they are pairwise disjoint. So for each of the remaining $M-1$ domains we can find points  $x_k  \in \mathcal{M}_{\varepsilon_0}(b_{i_k})\cap \overline{B}(x_0,R_0 + \eta)$ and $ p_k \in \mathbb{Z}^\ast$
	such that  $d(x_k, b_{i_k}^{p_k}x_k) = \varepsilon_0$. For this consider any  $x'_k \in \mathcal{M}_{\varepsilon_0}(b_{i_k})\cap \overline{B}(x_0,R_0 + \eta)$: by definition there exists $p_k\in\mathbb{Z}^\ast$ such that $d(x'_k, b_{i_k}^{p_k}x'_k)\leq \varepsilon_0$. On the other hand $d(x_0, b_{i_k}^{p_{i_k}}x_0) > \varepsilon_0$  since $x_0\notin \mathcal{M}_{\varepsilon_0}(b_{i_k})$.  Then, by continuity of the displacement function of the isometry $b_{i_k}^{p_k}$, we can find a point $x_k$ along a geodesic segment $[x_0,x'_k]$  such that $d(x_k, b_{i_k}^{p_k} x_k) = \varepsilon_0$ precisely. Remark that $x_k \in \overline{B}(x_0,R_0 + \eta)$ as it belongs to the geodesic $[x_0,x'_k]$.\\
	Now, since $\ell(b_{i_k}) \leq \frac{\varepsilon_0}{3}$ for all $k$, we can apply Proposition \ref{Margulis-estimate} and get
\vspace{-2mm}	
	
	$$d\left(x_k, \mathcal{M}_{\frac{\varepsilon_0}{3}}(b_{i_k})\right)\leq K$$
	for some  $K$ depending only on $P_0,r_0, \delta$ and $\varepsilon_0$, so (by Corollary \ref{margulis-constant}) \linebreak ultimately only on $P_0,r_0$ and $\delta$. \\
	So for each $k$ we have some point $y_k \in X$ such that
\vspace{-3mm}	
	
	$$d(y_k, b_{i_k}^{q_k}y_k)\leq \frac{\varepsilon_0}{3} \hspace{3mm}\textup{ and } \hspace{3mm} d(x_k,y_k)\leq K + \eta$$
	for some $q_k\in\mathbb{Z}^\ast$. 
	Set $R_1 = R_0 + \eta + K + \eta$, so that $y_k \in \overline{B}(x_0, R_1)$.
	We remark now  that the ball $\overline{B}(y_k, \frac{\varepsilon_0}{3})$ is contained in $\mathcal{M}_{\varepsilon_0}(b_{i_k})$: indeed for every $z \in \overline{B}(y_k, \frac{\varepsilon_0}{3})$ it holds
	$$d(z, b_{i_k}^{q_k}z)\leq d(z,y_k) + d(y_k, b_{i_k}^{q_k}y_k) + d(b_{i_k}^{q_k}y_k, b_{i_k}^{q_k}z)\leq \varepsilon_0.$$
	Finally we set $R_2 = R_1 + \frac{\varepsilon_0}{3}$: then we have $\overline{B}(y_k, \frac{\varepsilon_0}{3}) \subset \overline{B}(x_0, R_2)$ for all $k$.\linebreak
	All the balls $\overline{B}(y_k, \frac{\varepsilon_0}{3})$  are pairwise disjoint, so the points $y_k$ are $\frac{\varepsilon_0}{3}$-separated. Hence the cardinality $M$ of the set $S$ satisfies $M \leq 1+\text{Pack}\left( R_2, \frac{\varepsilon_0}{6}\right)=: M_0$,
	which is a number depending only on $P_0$, $r_0$, $\delta$, $L$ and $\eta$ by Proposition \ref{packingsmallscales}. \linebreak
	Taking for instance the constant $M_0$ obtained for $\eta=1$, we get the announced bound.
\end{proof}

To conclude the proof of the theorem in this case we will apply the previous lemma for an appropriate value of $L$.
By Proposition \ref{Margulis-estimate} we get
$$\sup_{x\in \mathcal{M}_{\varepsilon_0 + 56\delta}(b_i)} d(x, \mathcal{M}_{\varepsilon_0}(b_i)) \leq  K_0(P_0,r_0,\delta, \varepsilon_0) = K'_0(P_0,r_0,\delta),$$
where again Corollary \ref{margulis-constant}  bounds the value of $\varepsilon_0$ in terms of $P_0$ and $r_0$.\\
We set $L=2K'_0$
and apply Lemma \ref{conj}: so there exist $i,j \leq M_0(P_0,r_0,\delta)$ such that
$$d(\mathcal{M}_{\varepsilon_0}(b_i), \mathcal{M}_{\varepsilon_0}(b_j)) > 2K'_0.$$
In particular
$$d(\mathcal{M}_{\varepsilon_0 + 56\delta}(b_i), \mathcal{M}_{\varepsilon_0 + 56\delta}(b_j)) > 0.$$
(otherwise 
we would find $x_i \in \mathcal{M}_{\varepsilon_0 + 56\delta}(b_i)$ and $x_j \in \mathcal{M}_{\varepsilon_0 + 56\delta}(b_j)$ at arbitrarily small distance;
but there exists also points $y_i \in \mathcal{M}_{\varepsilon_0}(b_i)$ and $y_j \in \mathcal{M}_{\varepsilon_0}(b_j)$ with $d(x_i,y_i) \leq K'_0$ and $d(x_j,y_j) \leq K'_0$,
which would yield 
$d(\mathcal{M}_{\varepsilon_0}(b_i), \mathcal{M}_{\varepsilon_0}(b_j)) \leq 2K'_0$, a contradiction).\\
Applying $b^{-i}$,  we deduce that
$d(\mathcal{M}_{\varepsilon_0 + 56\delta}(a), \mathcal{M}_{\varepsilon_0 + 56\delta}(b^{j -i}a b^{i -j})) > 0.$\linebreak
This implies, by Lemma \ref{margulis-domains-distance}, that 
$$L(a, b^{j -i}a b^{i -j}) \geq \varepsilon_0 + 56\delta > 
\max\lbrace \ell(a), \ell(b^{j -i}a b^{i -j})\rbrace + 56\delta.$$
By Proposition \ref{free-group-gallot} we then deduce that the subgroup generated by $a$ and $w = b^{j -i}a b^{i -j} $ is free.
Remark that the length of $w$ is bounded above by $3M_0$ that is a function depending only on $P_0$, $r_0$ and $\delta$.\qed

\vspace{3mm}
\subsection{Proof of Theorem \ref{maintheorem}, case $\ell > \frac{\varepsilon_0}{3}$.}
{\em We  assume here that  $a,b$ are hyperbolic   $\sigma$-isometries  of a $\delta$-hyperbolic \textup{GCB}-space $(X,\sigma)$ which is $P_0$-packed at scale $r_0$,  satisfying  
	$\ell(a) =  \ell(b)  =\ell >    \frac{\varepsilon_0}{3}$, and such that  the group $\langle a,b \rangle$ is non-elementary and discrete.}
\vspace{1mm}

\noindent The proof of assertion (i) in Theorem \ref{maintheorem} in this case  stems directly from Proposition 4.9 of \cite{BCGS17} (free sub-semigroup theorem for isometries with minimal displacement bounded below), since  Corollary \ref{margulis-constant} bounds $\varepsilon_0$ in terms of $P_0$ and $r_0$.
So, {\em we will  focus here on the proof of  assertion (ii), \linebreak therefore assuming moreover  $\langle a,b \rangle$ torsionless}.
\vspace{1mm}

We fix $\sigma$-axes $\alpha$ and $\beta$ of $a$ and $b$ respectively. 
The proof of assertion (ii)  \linebreak of Theorem \ref{maintheorem} will break down into three subcases, according to the value of the distance $d_0= d(\alpha,\beta)$ between the minimal sets:  
the case where $d_0 \leq \frac{\varepsilon_0}{ 74 }$,  
the case  $\frac{\varepsilon_0}{ 74 } < d_0 \leq  30 \delta$ and the case $d_0>  30 \delta$.
In all cases we will use a ping-pong argument 
which we will explain in the next subsection.

\subsubsection{Ping-pong}
\label{sub-pingpong}

\noindent The aim of this subsection is to prove the following:

\begin{prop}
	\label{freegroup}
	Let $(X,\sigma)$ be a proper, $\delta$-hyperbolic \textup{GCB}-space and let $a,b$ be  hyperbolic $\sigma$-isometries of $X$ with  minimal displacement  $\ell(a) = \ell(b)=\ell$.  \linebreak Let $\alpha, \beta$ be two $\sigma$-axes for $a$ and $b$ respectively, satisfying $\partial \alpha \cap \partial \beta = \emptyset$. 
	Finally let $x_-,x_+$ be respectively projections of $\beta^-$, $\beta^+$ on $\alpha$  and suppose that $x_+$ follows $x_-$ along the (oriented) geodesic $\alpha$. Assume $d(x_-, x_+) \!\leq M_0$: 
	then the group $\langle a^N,b^N \rangle$ is free for any
	$N \geq (M_0 +  77  \delta) / \ell $. 
\end{prop}

For any $x \in \alpha$ and any $T\geq 0$ we will denote for short  by $x\pm T$   the points along $\alpha$ at distance $T$ from $x$, according to the orientation of $\alpha$; we will use the analogous notation $y \pm T$ for the points on $\beta$ such that $d(y, y\pm T) =T$.
By assumption, we have $x_+=x_-+T_0$ for some $T_0 \geq0$. We denote by $y_{\pm}$ a projection of $x_{\pm}$ on $\beta$.
Finally, for  $T>0$ we define  $T$-neighbourhoods of $\alpha^+$ and $\alpha^-$ as:  
\begin{equation}
	\label{pingpongequation}
	\begin{aligned}
		A_+(T) &= \lbrace z \in X \text{ s.t. } d(z,x_+ + T) \leq d(z,x_+) \rbrace\\
		A_-(T) &= \lbrace z \in X \text{ s.t. } d(z,x_- - T) \leq d(z,x_-) \rbrace
	\end{aligned}
\end{equation}
and their  analogues 
$B_\pm(T) = \lbrace z \in X \text{ s.t. } d(z,y_\pm \pm T) \leq d(z,y_\pm) \rbrace$ 
for $\beta$.
\vspace{2mm}

\noindent The proof will  stem from a series of  technical lemmas.
\begin{lemma}
	\label{pingpong1}
	For any $T\geq 2t \geq 0$ one has:
	\vspace{-4mm}	
	
	$$A_+(T) \subset \lbrace z \in X \hspace{1mm} | \hspace{1mm} (\alpha^{+}, z)_{x_+} \geq t\rbrace$$
	
	\vspace{-6mm}	
	
	$$A_-(T) \subset \lbrace z \in X \hspace{1mm} | \hspace{1mm} (\alpha^{-}, z)_{x_-} \geq t \rbrace.$$
	
	\vspace{-1mm}	
	\noindent Analogously, $B_\pm(T) \subset \lbrace z \in X \hspace{1mm} | \hspace{1mm} (\beta^{\pm}, z)_{y_\pm} \geq t \rbrace$
	for $T\geq 2t\geq0$.
\end{lemma}

\begin{proof}
	Let $z\in A_+(T)$, i.e. $d(z, x_+ + T) \leq d(z,x_+)$.
	For any $S\geq T$ we have
	$$d(z, x_+ + S) \leq d(z,x_+ + T) + (S - T) \leq d(z,x_+) + (S -T).$$
	
	\vspace{-3mm}	
	\noindent Hence
	\vspace{-3mm}	
	\begin{equation*}
		(\alpha^+, z)_{x_+} 
		\geq \liminf_{S\to +\infty} \frac{1}{2} \big[ d(z,x_+) + S - (d(z,x_+) + S -T) \big] = \frac{T}{2} \geq t.
	\end{equation*}
	The   proof  for $ B_\pm(T) $ is analogous. 
\end{proof}


\begin{lemma}
	\label{pingpong2}
	For any $T\geq 2t +  8 \delta$ one has:
	\vspace{0mm}	
	$$B_\pm(T) \subset \lbrace z \in X \hspace{1mm} | \hspace{1mm} (\beta^{\pm}, z)_{x_{\pm}} \geq t \rbrace.$$
\end{lemma}

\begin{proof}
	Let $z\in B_+(T)$, i.e. $d(z,y_+) \geq d(z,y_+ + T)$. We have, again, $\forall S\geq T$,
	\vspace{-3mm}	
	$$d(z,y_+ + S) \leq d(z,y_+ + T) + (S - T) \leq d(z,y_+) + (S - T).$$
	\noindent	Since $y_+$ is also a projection of $x_+$ on the geodesic segment $[y_+, y_+ + S]$,  by Lemma \ref{projection} it holds $(y_+, y_+ + S)_{x_+} \geq d(x_+,y_+) -  4 \delta$. 
	Expanding the Gromov product we get 
	$$d(x_+,y_++S) \geq d(x_+,y_+) + S -  8 \delta.$$
	Therefore 
	\begin{equation*}
		\begin{aligned}
			2(z,\beta^+)_{x_+} &\geq \liminf_{S\to +\infty} \big[d(x_+,z) + d(x_+,y_+ + S) - d(z,y_++S)\big] \\
			&\geq \liminf_{S\to +\infty} \big[d(x_+,z) + d(x_+,y_+) + S - 8\delta - d(z,y_+) - (S - T) \big] \\
			&\geq T -8\delta \geq 2t.
		\end{aligned}
	\end{equation*}
	
	\vspace{-2mm}		
	\noindent The proof for $B_-(T)$ is the same.
\end{proof}

\begin{lemma}
	\label{pingpong3}
	For any $z\in X$ it holds: 
	$$(\alpha^+, z)_{x_-} \geq (\alpha^\pm, z)_{x_+} - \delta , \qquad (\alpha^+, \beta^-)_{x_-} \geq (\alpha^+, \beta^-)_{x_+} - \delta,$$		 
	$$ (\alpha^-, z)_{x_+} \geq (\alpha^\pm, z)_{x_-} - \delta, \qquad (\alpha^-, \beta^+)_{x_+} \geq (\alpha^-, \beta^+)_{x_-} - \delta.$$ 
\end{lemma}
\begin{proof}
	We have $(\alpha^+,z)_{x_-} \geq \liminf_{S\to +\infty}(x_+ + S, z)_{x_-}$. On the other hand 
	for any $S\geq 0$ we get (as $x_+$ follows $x_-$ along $\alpha$)
	\begin{equation*}
		\begin{aligned}
			2(x_+ + S, z)_{x_-} &= d(x_+ + S, x_-) + d(x_-,z) - d(x_+ + S, z) \\
			&= d(x_+, x_+ + S) + d(x_+,x_-) + d(x_-,z) - d(x_+ + S, z) \\
			&\geq d(x_+, x_+ + S) + d(z, x_+) - d(x_+ + S, z) \\
			&= 2(x_+ + S, z)_{x_+}.
		\end{aligned}
	\end{equation*}			
	\noindent 	So, by (\ref{product-boundary-property}), we get 
	$ (\alpha^+,z)_{x_-} \geq \liminf_{S\to +\infty}(x_+ + S, z)_{x_+} \geq (\alpha^+,z)_{x_+} - \delta$.
	Taking any sequence $z_n$ converging to $\beta^-$ then proves the second formula.
	The proof for  $(\alpha^-, z)_{x_+} $ and  $(\alpha^-, \beta^+)_{x_+}$ is analogous.
\end{proof}


\begin{lemma}
	\label{pingpong5}
	For any $u\in X$ it holds:
	\vspace{-3mm}	
	
	$$(\beta^+, u)_{x_-} \geq (\beta^+, u)_{x_+} - 13\delta, 
	\quad  (\beta^-, u)_{x_+} \geq (\beta^-, u)_{x_-} - 13\delta.$$
\end{lemma}
\begin{proof}
	Take a sequence $(y_i)$ defining $\beta^+$ and  let $x_i$ be a projection of $y_i$ on $\alpha$. By Remark \ref{projection-remark} we know that, up to a subsequence, the sequence $x_i$ converges to $x_\infty$ which is a projection on $\alpha$ of $\beta^+$. So $d(x_\infty, x_+)\leq 10 \delta$ by Lemma \ref{projection-properties}. For any $\varepsilon >0$ and for every $i$ large enough, by Lemma \ref{projection}, we have 
	$$d(y_i, x_-) \geq d(y_i,x_i) + d(x_i,x_-) - 8\delta \geq d(y_i, x_+) + d(x_+, x_-) - 24\delta - \varepsilon.$$
	Therefore we get
	\begin{equation*}
		\begin{aligned}
			2( \beta^+, u)_{x_-} &\geq \liminf_{i\to +\infty}\big[d(x_-, u) + d(x_-,y_i) - d(u,y_i)\big] \\
			&\geq \liminf_{i\to +\infty}\big[d(x_-, u) + d(x_+, y_i) + d(x_-,x_+) - 24\delta - d(u,y_i)\big] \\
			&\geq \liminf_{i\to +\infty} \big[d(x_+, y_i) + d(x_+,u) - d(u,y_i) - 24\delta\big] \\
			&\geq 2( \beta^+, u)_{x_+} - 26 \delta.
		\end{aligned}
	\end{equation*}
	
	\vspace{-7mm}	
\end{proof}

\begin{lemma}
	\label{pingpong6}
	We have:  
	$$(\alpha^+, \beta^+)_{x_+} \leq 13 \delta, \quad (\alpha^-, \beta^+)_{x_+} \leq  13 \delta, $$
	$$ (\alpha^-, \beta^-)_{x_-} \leq 13 \delta,   \quad (\alpha^+, \beta^-)_{x_-} \leq 13 \delta $$
\end{lemma}
\begin{proof}
	Take a sequence $(y_i)$ defining $\beta^+$ and  call $x_i$ a projection of $y_i$ on $\alpha$. \linebreak
	As before, 
	$x_i$ converges, up to a subsequence,  to a projection  $x_\infty$  of $\beta^+$  on $\alpha$,  \linebreak and $d(x_\infty, x_+)\leq  10 \delta$. For any $\varepsilon >0$, for any $S\geq 0$ and for every $i$ large enough, by Lemma \ref{projection}, we have 
	\begin{equation*}
		\begin{aligned}
			d(y_i, x_+ \pm S) &\geq d(y_i,x_i) + d(x_i,x_+ \pm S) - 8\delta \\
			&\geq d(y_i, x_+) + d(x_+, x_+ \pm S) -  24 \delta - \varepsilon.
		\end{aligned}
	\end{equation*}
	Therefore we get	
	\begin{equation*}
		\begin{aligned}
			2(\alpha^+, \beta^+)_{x_+} &\leq \liminf_{i,S\to +\infty}\big[d(x_+, x_+ \!+\! S) + d(x_+,y_i) - d(x_+ \!+\! S,y_i)\big] + 2\delta \leq 26 \delta.
		\end{aligned}
	\end{equation*}
	
	\vspace{-2mm}
	\noindent	The same computation with $-S$ instead of $S$ proves the second inequality. \linebreak
	The inequalities involving $\beta^-$ are proved in the same way.
\end{proof}
\begin{lemma}
	\label{pingpong}
	The subsets $A_+(T)$,  $A_-(T)$, $B_+(T)$ and $B_-(T)$ are pairwise  disjoint for  $T > 64 \delta$.
\end{lemma}

\begin{proof} Fix some $t> 28 \delta$. We claim that  the subsets $A_+(T), A_-(T)$, $B_+(T)$ and $B_-(T)$ are pairwise disjoint provided that  $T\geq 2t + 8 \delta > 64 \delta$.\linebreak
	We first prove that  $A_+(T) \cap A_-(T) = \emptyset$. If $z\in A_+(T) \cap A_-(T)$ then $(\alpha^+, z)_{x_+} \geq t > 28 \delta$ and $(\alpha^-, z)_{x_-} \geq t > 28 \delta$ by Lemma \ref{pingpong1}. Then by Lemma \ref{pingpong3} we have $(\alpha^-,z)_{x_+} >  27  \delta$.
	Thus we obtain a contradiction since 
	\vspace{-3mm}
	
	$$\delta \geq (\alpha^+, \alpha^-)_{x_+} \geq \min \lbrace (\alpha^+, z)_{x_+}, (\alpha^-, z)_{x_+}\rbrace - \delta > 26 \delta.$$
	
	\noindent Let us prove now that $A_+(T) \cap B_+(T) = \emptyset$. If $z\in A_+(T) \cap B_+(T)$ then we have $(\alpha^+, z)_{x_+} > 28 \delta$ and $(\beta^+, z)_{x_+} > 28 \delta $ by Lemma \ref{pingpong1} and Lemma \ref{pingpong2}. \linebreak We then obtain again a contradiction by Lemma \ref{pingpong6}, as
	\vspace{-3mm}
	
	$$13 \delta \geq (\alpha^+,\beta^+)_{x_+} \geq \min \lbrace (\alpha^+, z)_{x_+}, (\beta^+, z)_{x_+}\rbrace - \delta > 27 \delta. $$
	
	\noindent We now prove that $A_+(T) \cap B_-(T) = \emptyset$. Actually if $z\in A_+(T) \cap B_-(T)$ then $(\alpha^+, z)_{x_+} > 28 \delta$ and $(\beta^-, z)_{x_-} > 28 \delta$ by Lemma \ref{pingpong1} and Lemma \ref{pingpong2}. Moreover by Lemma \ref{pingpong5} we have $(\beta^-, z)_{x_+} > 15 \delta$ and combining Lemma \ref{pingpong6} with Lemma \ref{pingpong3} we deduce that  $(\beta^-,\alpha^+)_{x_+} \leq (\beta^-,\alpha^+)_{x_-} + \delta \leq 14 \delta.$ \linebreak
	So we again get a contradiction since
	\vspace{-3mm}
	
	$$14 \delta \geq (\alpha^+,\beta^-)_{x_+} \geq \min \lbrace (\alpha^+, z)_{x_+}, (\beta^-, z)_{x_+}\rbrace - \delta >14 \delta.$$

	\noindent The proof of  $B_+(T) \cap B_-(T) = \emptyset$ can be done as for $A_+(T) \cap A_-(T) = \emptyset$, using Lemma \ref{pingpong3}.  
	The remaining cases can be proved similarly.
\end{proof}
  
We are now in position to prove Proposition \ref{freegroup}.

\begin{proof}[Proof of Proposition \ref{freegroup}]
	We define the sets:
	\vspace{-3mm}	
	
	\begin{equation*}
		\begin{aligned}
			A_+ &= \lbrace z \in X \text{ s.t. } d(z,a^N x_-) \leq d(z,x_+)\rbrace, \\
			A_- &= \lbrace z \in X \text{ s.t. } d(z,a^{-N} x_+) \leq d(z,x_-)\rbrace
		\end{aligned}
	\end{equation*}
	
	\vspace{-1mm}
	\noindent and their analogues
	$B_\pm = \lbrace z \in X \text{ s.t. } d(z,b^N y_\mp) \leq d(z,y_\pm)\rbrace$ for  $b$.\\
	By assumption we have $d(x_-, a^N x_-) = N\ell(a) \geq d(x_-,x_+) +  65 \delta$. \\
	In particular $A_+ \subset A_+( 65 \delta)$ as defined in \eqref{pingpongequation}, as follows directly from the $\sigma$-convexity of the distance function from the $\sigma$-geodesic line $\alpha$ using the fact that $x_+ + 65\delta$ is between $x_+$ and $a^Nx_-$ (according to the chosen orientation of $\alpha$), and $A_- \subset A_-(65 \delta)$.\\
	Moreover we have $d(y_-,y_+) \leq d(x_-,x_+) +12 \delta$ by Lemma \ref{projection-properties},  and we can prove in the same way that $B_+\subset B_+(65 \delta)$ and $B_- \subset B_-(65 \delta)$.\\
	Then by Proposition \ref{pingpong} the sets $A_+,A_-,B_+,B_-$ are pairwise disjoint. \linebreak
	We will prove now the following relations:
	\vspace{-3mm}
	
	\begin{equation*}
		\begin{aligned}
			&a^N(X \setminus A_-) \subset A_+, \quad a^{-N}(X \setminus A_+) \subset A_-,\\
			&b^N(X \setminus B_-) \subset B_+, \quad b^{-N}(X \setminus B_+) \subset B_-
		\end{aligned}
	\end{equation*}
	Indeed if $z\in X \setminus A_-$ then $d(z,a^{-N}x_+)> d(z,x_-)$; applying  $a^N$ to both sides we get $d(a^N z, x_+)> d(a^N z, a^N x_-)$, i.e. $a^Nz \in A_+$. The other relations are proved in the same way.
	As a consequence we have, for all $k \in \mathbb{N}^\ast$,
	\vspace{-3mm}
	
	$$a^{kN}(X \setminus A_-) \subset  A_+, \quad a^{-kN}(X \setminus A_+) \subset A_-,$$
	$$b^{kN}(X \setminus B_-) \subset B_+, \quad b^{-kN}(X \setminus B_+) \subset B_-.$$	

\noindent It is then standard to deduce by a ping-pong argument that the group generated by $a^N$ and $b^N$ is free.
	Actually no nontrivial reduced word $w$ in $\{a^N,b^N \}$ can represent the identity, since it sends any point of $X \setminus (A_+ \cup A_- \cup B_+ \cup B_-)$ into the complement $A_+ \cup A_- \cup B_+ \cup B_- $ (notice that the former set is non-empty, as $X$ is connected).
\end{proof}

\subsubsection{Proof of Theorem \ref{maintheorem} (ii), 
	when  $ d(\alpha,\beta) \leq \frac{\varepsilon_0}{74}$.}
\label{sub1}

{\em Recall that we are assuming $a,b$   $\sigma$-isometries  with 
	$ \ell(a) =  \ell(b) =\ell >  \frac{\varepsilon_0}{3}$.} \\ 
Let  $x_0 \in \alpha, y_0\in \beta$ be points with $d(x_0,y_0)= d(\alpha,\beta)$.
Denote by $\pi_\alpha$ and $\pi_\beta$ the projection maps on $\alpha$ and $\beta$ respectively, and call  $x_\pm$  the projections of $\beta^\pm$ on $\alpha$.
As in subsection \ref{sub-pingpong}, up to replacing $b$ with $b^{-1}$ we can assume that $x_+$ follows $x_-$ along $\alpha$.

Let now $[z_-,z_+]$ be the set of points of $\alpha$  at distance $d = \frac{\varepsilon_0}{37} < \frac13 \ell $   from $\beta$. \linebreak It is a nonempty, finite geodesic segment since the distance function from $\beta$ is convex  when restricted to $\alpha$,  and $\partial \alpha \cap \partial \beta = \emptyset$ (the group  $\langle a, b \rangle$ being non-elementary).  Clearly, it holds:
\vspace{-3mm} 

$$d(z_-,\beta) = d(z_+,\beta) = d .$$

\noindent Call $2L$ the length of $[z_-,z_+]$. 
The following estimate of this length  is due to Dey-Kapovich-Liu: since the proof  is scattered in different papers (it appears in the discussion  after Lemma 4.5 in \cite{DKL18}, using an argument of \cite{Kap01} for trees), we consider worth to recall it for completeness:

\begin{prop}\label{kapo}
	With the notations above it holds: $2L < 5 \ell $.
\end{prop}

\noindent For $0\leq T \leq L$  let $\alpha^T$ denote the central segment of $[z_-,z_+]$ of length $2T$ (so, $\alpha^{L} = [z_-,z_+]$). Then:

\begin{lemma}
	\label{first-lemma}
	For any $x\in \alpha^{L-\ell}$  we have that either
	\vspace{-3mm}		
	
	\begin{equation}\label{eqacomeb}
		d(\pi_\beta(ax),\pi_\beta(bx)) \leq 2d   \quad \textit{and} \quad d(\pi_\beta(a^{-1}x),\pi_\beta(b^{-1}x)) \leq 2d  
	\end{equation}
	
	\vspace{-3mm}		
	\noindent 	or 
	\vspace{-3mm}	
	\begin{equation}\label{eqaoppostob}d(\pi_\beta(ax),\pi_\beta(b^{-1}x)) \leq 2d \quad \textit{and} \quad d(\pi_\beta(a^{-1}x),\pi_\beta(bx)) \leq 2d.
	\end{equation}
	Moreover the  conditions  (\ref{eqacomeb}) and  (\ref{eqaoppostob}) cannot hold together,  and one of the two  hold on the whole interval $\alpha^{L-\ell}$.
\end{lemma}

\begin{proof}
	By assumption, as $\ell(a)=\ell$, the points $a^{\pm 1}x$ belong to $\alpha^{L}$.
	Then, by definition of $\alpha^{L}$, we have
	$d(\pi_\beta (x), x) \leq d$ and $d(\pi_\beta (a^{\pm 1} x), a^{\pm 1}x) \leq d$.
	Therefore, we have    $| d(\pi_\beta( x) , \pi_\beta (a^{\pm 1}x)) - \ell ( a) | \leq 2d.$ 
	As $\pi_\beta( x) , \pi_\beta (a^{\pm 1}x)$ belong to $\beta$ and  $b$ translates $\pi_\beta (x)$ along $\beta$ precisely by $\ell=\ell(b)=\ell (a)$,  it follows that there exists $\tau, \tau' \in \{1,-1\}$ such that 
	\vspace{-3mm}	
		
	$$  d(   \pi_\beta (a  x),  \pi_\beta( b^\tau x)) = d(  \pi_\beta (a  x), b^\tau\pi_\beta(  x) )   \leq 2d,$$
		$$ d(   \pi_\beta (a^{-1}  x),  \pi_\beta( b^{\tau'} x))     \leq 2d.$$
	Moreover, since $d(\pi_\beta(a x) , \pi_\beta (a^{-1}x))  \geq  2\ell  -2d$, the above relations cannot hold together with $\tau=\tau'$ when $ 2\ell  -2d > 4d$.    
	Therefore, for our choice of $d <   \frac13  \ell$ we have $\tau'=-\tau$ and  (\ref{eqacomeb}) is proved.\\
	Since $d(\pi_\beta(b x) , \pi_\beta (b^{-1}x)) = 2\ell >4d$, the relations (\ref{eqacomeb}) and  (\ref{eqaoppostob}) cannot hold at the same time. Finally the last assertion follows from the connectedness of the interval $\alpha^{L- \ell}$.
\end{proof}	

\begin{lemma}
	\label{second-lemma}
	Let $\eta \geq 0$ and $x\in \overline{B}(\alpha^{L - \ell}, \eta)$. Then either  		
	\begin{equation}\label{eqacomeb'}
		d(b x,\pi_\alpha(a x))\leq   3\eta +   6d  \quad \textit{and} \quad d(b^{- 1}x,\pi_\alpha(a^{-1}x))\leq  3\eta +   6d \end{equation}		
	\noindent 	or 	
	\begin{equation}\label{eqaoppostob'}
		d(bx,\pi_\alpha(ax))\leq  3\eta + 6d  \quad \textit{and} \quad d(b^{- 1}x,\pi_\alpha(a^{-1}x))\leq 3\eta +   6d
	\end{equation}	
	Moreover the first (resp. second) condition occurs if and only if the first (resp. second) condition in Lemma \ref{first-lemma} holds.
\end{lemma}
\begin{proof}
	We fix $x\in X$ such that $d(x, \alpha^{L - \ell})\leq \eta$. 
	Since every point of $\alpha^{L-\ell}$ is at distance at most $d$ from $\beta$ then $d(x, \pi_\beta(x))\leq \eta +   d$. We assume that (\ref{eqacomeb}) holds and we  prove the first one in (\ref{eqacomeb'}),  the other cases  being similar. \linebreak We have:
	$$d(bx, \pi_\alpha(ax)) \leq d(bx, \pi_\beta(bx)) + d(\pi_\beta(bx), \pi_\beta(\pi_\alpha(ax))) + d(\pi_\beta(\pi_\alpha(ax)), \pi_\alpha(ax)).$$
	The first term equals $d(x, \pi_\beta(x)) \leq \eta +  d$. We observe that from the choice of ${L - \ell}$ we have $\pi_\alpha(ax) \in \alpha^{L}$, so the third term is smaller than or equal to $d$. For the second term we have
	\begin{equation*}
		\begin{aligned}
			d(\pi_\beta(bx), \pi_\beta(\pi_\alpha(ax))) &\leq d(\pi_\beta(bx), \pi_\beta(b\pi_\alpha(x))) + d(\pi_\beta(b\pi_\alpha(x)), \pi_\beta(a\pi_\alpha(x)))\\
			&\leq  d(\pi_\beta ( x) , x) + d(x, \pi_\beta(\pi_\alpha(x)))    +   2d
			\leq 2\eta +4d.
		\end{aligned}
	\end{equation*}
	where we used (\ref{eqacomeb}) to estimate the   term $d(\pi_\beta(b\pi_\alpha(x)), \pi_\beta(a\pi_\alpha(x)))$. \\
	In conclusion,  $d(bx, \pi_\alpha(ax)) \leq  3\eta+6d$. 
\end{proof}
\begin{lemma}\label{lemmacommutators}
	If  $2L \geq 5\ell$ then for any commutator $g$ of $\{a^{\pm 1}, b^{\pm 1}\}$ and any $x\in \alpha^{ L/5}$ we have $d(x,gx) \leq 36d$.
\end{lemma}
\begin{proof}
	We give the proof for $[a,b] = aba^{-1}b^{-1}$, the other cases are similar. \linebreak
	Assume that (\ref{eqacomeb'}) holds. We have $d(b^{-1}x, a^{-1}x) \leq 6d$ by Lemma \ref{second-lemma}. 
	Calling $x'= a^{-1}b^{-1}x$, we have 
	\begin{equation}\label{eqconduetermini}
		d([a,b] x,x) \!=\!	d(b x',  a^{-1}x)  \leq  d(b x', \pi_\alpha( b^{-1}x)) \!+\! d(\pi_\alpha( b^{-1}x),  a^{-1}x).
	\end{equation}
	By the second inequality in (\ref{eqacomeb'}) we have $d(x', \alpha ) = d (b^{-1}x, \alpha) \leq  6d$; hence applying  the first one in (\ref{eqacomeb'})  to $x'$ yields $d(b x', \pi_\alpha( b^{-1}x)) \leq 24d$. 
	The second term in (\ref{eqconduetermini}) is less than or equal to
	$$d(\pi_\alpha(b^{-1}x), b^{-1}x) + d(b^{-1}x, a^{-1}x) \leq 2d(b^{-1}x, a^{-1}x) \leq 12d.$$
	So $d([a,b] x,x)\leq 36d$.
	The proof in case (\ref{eqaoppostob'}) holds is analogous.
\end{proof}

\vspace{3mm}   
\begin{proof}[Proof of Proposition \ref{kapo}.]
	If $2L \geq 5\ell$, then by Lemma \ref{lemmacommutators} there exists a point $x \in \alpha^{L/5}$ which is displaced by all the commutators of $a^{\pm 1}, b^{\pm 1}$ by less than $36d < \varepsilon_0$. In particular $[a^{\pm  1}, b^{\pm 1}]$ and $[a^{\pm 1}, b^{\mp 1}]$ belong to the same elementary group. This in turns implies that $\langle a, b \rangle$ is elementary.  Indeed, if one commutator is the identity then $\langle a, b \rangle$ is abelian, hence elementary. If the commutators are all different from the identity, then they do not have finite order (since $\langle a, b \rangle$ is assumed to be torsion-free); so,  as they belong to the same elementary group, there exists a subset $F\subseteq \partial X$ made of one or two points that is fixed by all the commutators. But since
	$a^{-1}[a,b]a = [b, a^{-1}]$ and $b^{-1}[a,b]b = [b^{-1}, a]$
	we deduce that  
	$a^{-1}(F)=F$ and $b^{-1}(F)=F$.
	Therefore $F$ is invariant  for both $a$ and $b$, hence $\langle a, b \rangle$ should be elementary. This contradiction concludes the proof.
\end{proof}

Let now $x \in [z_-,z_+]$ be any point whose distance from $\beta$ is at most  $\frac{\varepsilon_0}{ 74 }$. \linebreak It exists, by assumption on the distance between the geodesics $\alpha$ and $\beta$. \linebreak 
Now the distance function $d_\beta (\cdot) = d(\cdot, \beta)$ restricted to $\alpha$ is convex
with  $d_\beta (x)\leq \frac{\varepsilon_0}{ 74 } $ and  $d_\beta (z_\pm)= d =  \frac{\varepsilon_0}{37}$.
Furthermore by Proposition \ref{projections-contracting} the value of $d_\beta$ at $x_+$ and at $x_-$ does not exceed $M=\max\lbrace  49 \delta,   \frac{\varepsilon_0}{ 74 }+19\delta\rbrace$. 
Therefore by convexity we deduce
$$d(x, x_+) \leq \frac{d_\beta (x_+) - d_\beta(x)  }{d_\beta(z_+) - d_\beta(x)} \cdot d(x,z_+) \leq \frac{74M}{ \varepsilon_0}\cdot d(x,z_+) $$
\noindent and the same estimate holds for $d(x,x_-)$. Thus:  
\vspace{-2mm}

$$d(x_-,x_+)  \leq  74 \max \left\lbrace 49 \frac{\delta}{\varepsilon_0},   19 \frac{\delta}{\varepsilon_0} + 1  \right\rbrace \cdot 2L =:M_0$$

\noindent As $2L\leq 5 \ell$, we   conclude the proof in this case  using Proposition \ref{freegroup}. \linebreak
Recall that $\ell(a) =\ell > \frac{\varepsilon_0}{3}$, so it is enough to choose
%
%
%
$$N = \left\lceil 36491 \frac{\delta}{\varepsilon_0} +1 \right\rceil $$
\noindent which is bounded  above by a function depending only on $P_0,r_0$ and $\delta$.\qed

\vspace{2mm}
\subsubsection{Proof of Theorem \ref{maintheorem}(ii), 
	when  $\frac{\varepsilon_0}{ 74 }  <  d( \alpha,  \beta)\leq  \! 30  \delta$}

{\em Recall that we are assuming $a,b$  $\sigma$-isometries  with 
	$\ell(a) =  \ell(b) =\ell >   \frac{\varepsilon_0}{3}$.}\\
We use the same notations as in  \ref{sub1}:      $\alpha(0)=x_0$,  $\beta(0)=y_0$  and   $d(\alpha,\beta)=d(x_0,y_0)$.
Finally the points $x_+, x_-$  are respectively projections of $\beta^+, \beta^-$ on $\alpha$,   and $x_+$ follows $x_-$ along $\alpha$.   
\vspace{1mm}

\noindent The strategy here is to use the packing assumption to reduce the proof to the previous subcase.
For this, we set 
\vspace{-3mm}

$$P := \text{Pack}\bigg( 60 \delta, \frac{\varepsilon_0}{148}\bigg) + 1.$$
By assumption $d(x_0,y_0)=d(x_0,\beta) \leq  30 \delta$. 
We define $[z_-,z_+]$ as the (non-empty) subsegment of $\alpha$ of points whose distance from $\beta$ is at most $ 60 \delta$. \linebreak
Moreover let $w_-$ and $w_+$ be some projections of $z_-$ and $z_+$ on $\beta$, respectively. There are two possibilities:  $d(w_-,w_+) \geq 2P\ell$ or the opposite.
\vspace{2mm}

Assume that we are in the first case.  
Let $w$ be the midpoint of the segment $[w_-,w_+]$. For every $i=1,\ldots, P$ we consider the isometry $b^i$. 
Then: 
$$d(b^i z_-, \beta) = d(z_-, \beta) =60 \delta, \quad d(b^i z_+, \beta) = d(z_+, \beta) =60 \delta.$$
\noindent Notice that the points $w^i_- = b^i w_-$ and  $w^i_+ = b^i w_+$  are respectively  projections of $b^i z_-$ and  $b^i z_+$ on $\beta$.
So $d(w_-, w^i_-)= i\cdot \ell \leq P\cdot\ell \leq d(w_-,w)$. \linebreak
In particular $w^i_-$ belongs to the segment $[w_-, w]$ for all $1 \leq i \leq P$.\linebreak
Hence $w$ belongs to the segment $[w^i_-,w^i_+]$. Since the distance from $ b^i \alpha$ of $w^i_-$ and $w^i_+$ is $\leq 60 \delta$ then,  by convexity, we get $d(w, b^i  \alpha)\leq 60 \delta$. Hence there exists a point $z_i \in  b^i  \alpha$ such that $d(z_i,w) \leq 60 \delta$.
If the distance between any two of these points $z_i$ was greater than $\frac{\varepsilon_0}{74}$ then  the subset $S=\{ z_1, \cdots, z_P\}$ would be a $\frac{\varepsilon_0}{74}$-separated subset of $\overline{B}(w, 60 \delta)$,  but this is in contrast with the definition of $P$.
So 
there must be two different indices $1\leq i,j\leq P$ such that $d(z_i,z_j)\leq \frac{\varepsilon_0}{74}$. 
Therefore: 
\vspace{-3mm}

$$d(\alpha, b^{j-i}\alpha)=d(b^i\alpha, b^j\alpha) \leq \frac{\varepsilon_0}{74} .$$
Clearly $b^{j-i}\alpha$ is a $\sigma$-axis of $b^{j-i} a b^{i-j}$ and the group $\langle a, b^{j-i} a b^{i-j} \rangle$ is again discrete, non-elementary and torsion-free.
From the proof of Theorem \ref{maintheorem} in the case where there are $\sigma$-axes at distance $\leq \frac{\varepsilon_0}{74}$ given in Subsection \ref{sub1}, we then deduce that there exists an integer $N(\delta, P_0, r_0)$ such that the group $\langle a^N, (b^{j-i}ab^{i-j})^N \rangle$ is free.\\ 
Remark that the length of $(b^{j-i}ab^{i-j})^N$, as a word in $\{a,b\}$, is at most $3PN$, and this number is bounded above in terms of $\delta, P_0$ and $r_0$.
\vspace{2mm}

Assume now that we are in the  case where $d(w_-,w_+) < 2P\ell$.  
Then:
\vspace{-3mm}

$$d(z_-, z_+) \leq 2P\ell  + 120 \delta.$$
Starting from this inequality we want to bound the distance between the projections $x_-, x_+$  of $\beta^-$ and $\beta^+$ on $\alpha$, in order to apply Proposition \ref{freegroup}.\linebreak We look again at the  distance  $d_\beta$  from $\beta$: we know that $d_\beta(x_0)\leq  30 \delta$ and $d_\beta(z_-)=d_\beta(z_+) =60 \delta$. Moreover $d(x_0, z_+)\leq d(z_-,z_+) \leq 2P\ell + 120 \delta$.\linebreak
By Proposition \ref{projections-contracting}.c, we know that $d_\beta(x_+) \leq \max\lbrace  49 \delta,  19 \delta + d(\alpha, \beta)\rbrace = 49 \delta$. \linebreak
Then we can conclude that $d(x_0, x_+)\leq d(x_0,z_+)\leq 2P\ell  +  120 \delta$, by convexity of the function $d_\beta$   restricted to $\alpha$. 
The same estimate holds for $x_-$, so 
\vspace{-3mm}

$$d(x_-,x_+)\leq 4P\ell +240\delta =: M_0.$$
We can therefore conclude that  group $\langle a^N, b^N \rangle$ is free by Proposition \ref{freegroup},   for any
$N\geq 4P +  317 \delta/\ell$.
\noindent Again we remark that $N$ can be bounded from above by a function depending only on $P_0,r_0$ and $\delta$.\qed

\subsubsection{Proof of Theorem \ref{maintheorem} (ii), 
	when  $ d( \alpha, \beta) > 30 \delta$.}

{\em Recall that we are assuming $a,b$  $\sigma$-isometries  with 
	$\ell(a) =  \ell(b)  =\ell  >    \frac{\varepsilon_0}{3}$.}
We use the same notations as in  \ref{sub1}:     $\alpha(0)=x_0$,  $\beta(0)=y_0$  and   $d(\alpha,\beta)=d(x_0,y_0)$.
Finally the points $x_+, x_-$  are respectively projections of $\beta^+, \beta^-$ on $\alpha$,   and $x_+$ follows $x_-$ along $\alpha$.   
By Remark \ref{projection-remark} the points $x_-, x_+$ can be chosen in this way: for any time $t \geq 0$ we denote by $x_t$ a projection on $\alpha$ of $\beta(t)$. The limit point of a convergent subsequence of $(x_t)$, for $t\to + \infty$, defines a projection $x_+$ of $\beta^+$ on $\alpha$.   The point $x_-$ can be similarly chosen as the limit point of a convergent subsequence of $(x_t)$ for $t\to -\infty$. By Proposition \ref{projections-contracting}.b we have $d(x_t, x_{t'})\leq 9\delta$ for any $t,t' \in \mathbb{R}$ and this implies that $d(x_-,x_+)\leq 9\delta$.
We can then apply again Proposition \ref{freegroup}  
to conclude that the group $\langle a^N, b^N \rangle$ is free for  $N\geq  86 \frac{\delta}{\ell} $,  and the least $N$ with this property can be bounded as before  by a function depending only on $P_0,r_0$ and $\delta$. \qed



\section{Applications}\label{sec-app}
In this last section we will always assume that $(X,\sigma)$ is a $\delta$-hyperbolic, GCB-space that is $P_0$-packed at scale $r_0$.

\subsection{Lower bound for the  entropy}
\label{sub-lower-entropy}
We prove here the universal lower bounds for the entropy of $X$ and for the algebraic entropy of any finitely generated,  non-elementary discrete group $\Gamma$ of $\sigma$-isometries of $X$.
\vspace{1mm}

\noindent Recall that we defined in Section \ref{sub-notation} the {\em nilpotence radius} of  $\Gamma$ at $x$ as 
$$\textup{nilrad} (\Gamma,x)= \sup\lbrace r\geq 0 \hspace{1mm} \text{ s.t. } \hspace{1mm}  
\Gamma_r(x) \text{ is virtually nilpotent}\rbrace$$ 
\noindent and the {\em nilradius} of the action as nilrad$( \Gamma, X)= \inf_{x\in X}\text{nilrad}(\Gamma,x)$.

\begin{proof}[Proof of Corollary \ref{cor-entropy}]
	We fix any symmetric, finite generating set $S$ of $\Gamma$. \linebreak 
	Clearly there exist $a,b \in S$ such that $\langle a,b \rangle$ is non-elementary, or $\Gamma$ would be elementary. By Theorem \ref{maintheorem}(i), there exists a free semigroup $\langle v, w \rangle^+$ where $v,w \in S^N$, with $N$ depending on $\delta, P_0, r_0$. In particular $\textup{card} ( S^{kN} ) \geq 2^k$  
	for all $k$, so  $\text{Ent}(\Gamma, S) \geq \frac{\log 2}{N}=C_0$. Since this holds for any $S$,  this proves (i). \\
The first inequality in (ii) is classic (see for instance \cite{Cav21}, Lemma 7.2).\\
	In order to show the second inequality notice that, if $\nu_0=\textup{nilrad} (\Gamma, X)$, there exist a point $x\in X$ and $g_1, g_2 \in \Gamma$  generating a non-virtually nilpotent subgroup such that $\max\{ \ell(g_1), \ell (g_2) \} \leq \nu_0$. Since $X$ is packed, by Corollary \ref{elementary-nilpotent} the subgroup $\langle g_1,g_2 \rangle$ is non-elementary, so we  apply Theorem \ref{maintheorem}(i) and deduce the existence of  a free semigroup $\langle v, w \rangle^+$ where $v,w \in S^N$, \linebreak 
	 for $N=N(P_0,r_0,\delta)$. 
	As $\textup{card} \big( \langle g_1, g_2 \rangle  x \cap  B(x,  kN \nu_0  )\big) \geq 2^{k}$, we have 
	$\textup{Ent} (\Gamma,X) \geq   \frac{\log 2}{N \nu_0} = C_0 \cdot \nu_0^{-1}$.
\end{proof}

\subsection{Systolic estimates}\label{sub-systole}
Recall that we defined in Section \ref{sub-notation} the  {\em minimal free displacement} $\text{sys}^\diamond    (\Gamma,x)$  at $x$  as the infimum of $d(x,gx)$ when $g$ runs over the subset $ \Gamma \setminus  \Gamma^\diamond$ of the torsionless elements of $\Gamma$, and the {\em free systole} of the action as
$$\textup{sys}^\diamond ( \Gamma, X)= \inf_{x\in X} \text{sys}^\diamond (\Gamma,x).$$

\begin{cor}
	\label{cor-systole}
	Let $(X,\sigma)$ be a $\delta$-hyperbolic  \textup{GCB}-space, $P_0$-packed at scale $r_0>0$.
	Then for any non-elementary discrete group of $\sigma$-isometries  $\Gamma$ of $X$ it holds:
	\begin{equation}\label{formula:sys-nil}
		\textup{sys}^\diamond (\Gamma, x) \geq \min \left\{ \varepsilon_0, \frac{1}{H_0}e^{-H_0\cdot   \textup{nilrad}(\Gamma, x)} \right\}
	\end{equation}
	where $H_0=H_0(P_0,r_0,\delta)$ is a constant depending only on $P_0,r_0$ and $\delta$.
\end{cor}

The proof  is a modification of the proof of Theorem 6.20 of \cite{BCGS17}. 
We will also need the following:
\begin{lemma}[see \cite{Cav21}, Lemma 3.3 \& Lemma 7.2]${}$
	\label{entropybound}\\
	Let $(X,\sigma)$ be a $\delta$-hyperbolic  \textup{GCB}-space, $P_0$-packed at scale $r_0>0$, and let $\Gamma$ be a discrete group of $\sigma$-isometries of $X$. Then
	$$\textup{Ent}(\Gamma, X) \leq \frac{\log(1+P_0)}{r_0} =:E_0.$$ 
\end{lemma}

\begin{proof}[Proof of Corollary \ref{cor-systole}]
	Suppose that $\text{sys}^\diamond(\Gamma,x) < \varepsilon_0$: then we can choose a non-elliptic element $a \in \Gamma$ such that $d(x,ax) =  \text{sys}^\diamond(\Gamma,x) <\text{nilrad}(\Gamma, x)$.  \linebreak
	Indeed the infimum in the definition of $\text{sys}^\diamond(\Gamma,x)$ is attained since the action is discrete  and $\text{nilrad}(\Gamma, x) \geq \varepsilon_0$ by definition of $\varepsilon_0$.	If nilrad$(\Gamma,x)=+\infty$ there is nothing to prove. 
	Otherwise we fix any arbitrary $\varepsilon >0$ and set $R=(1+\varepsilon)\cdot\text{nilrad}(\Gamma,x)$.
	By definition $\Gamma_R(x)$ is not virtually nilpotent and contains $a$. This implies that    there exists $b\in \Gamma_R(x)$ such that $d(x,bx)\leq R$ and $\langle a,b \rangle$ is not elementary. 
	Indeed  $\Gamma_R(x)$ is finitely generated by  some $b_1,\ldots, b_k$  with $d(x,b_ix)\leq R$ for all $i=1,\ldots,k$ (since $\Gamma$ is discrete);  if $\langle a, b_i \rangle$ was elementary for all $i$  then each $b_i$ would belong to  the maximal, elementary subgroup containing $g$ (Lemma \ref{M-elementary}), hence  $\Gamma_R(x)$ would be elementary and virtually nilpotent (by Corollary \ref{elementary-nilpotent}), a contradiction.\linebreak
	We can therefore apply Theorem \ref{maintheorem} and infer that the semigroup $\langle a^{\tau N}, w\rangle^+$
	is free, where $w$ is a word on $a$ and $b$ of length at most $N=N(P_0,r_0,\delta)$, and $\tau \in \{\pm 1\}$.
	We now use Lemma 6.22.(ii) of \cite{BCGS17} to deduce that
	\begin{equation*}
		\text{Ent}(\Gamma, X) \cdot \min\left\{  d(x,a^{\tau N}x), d(x,wx) \right\} \geq e^{-\text{Ent}(\Gamma, X) \cdot \max \lbrace d(x,a^{\tau N}x), d(x,wx) \rbrace}.
	\end{equation*}
	As $d(x,wx)\leq NR$ and 
	$d(x,a^{\pm N}x)\leq N\cdot\text{sys}^\diamond (\Gamma,x) < N \cdot \text{nilrad}(\Gamma, x)$,  
	this implies, by Lemma \ref{entropybound} and by the arbitrariness of $\varepsilon$, that
	\begin{equation*}
		E_0N \cdot\text{sys}^\diamond(\Gamma,x) \geq e^{-E_0N \cdot\text{nilrad}(\Gamma,x)}.
	\end{equation*}
	\vspace{-4mm}
	
	\noindent The conclusion follows by setting $H_0=E_0\cdot N$.
\end{proof}

We define the {\em upper nilradius} of $\Gamma$ acting on $X$ as  the supremum of $\text{nilrad}(\Gamma,x)$ over the $\varepsilon_0$-thin subset  
$$\textup{nilrad}^+( \Gamma, X)= \sup_{x\in X_{\varepsilon_0}}\text{nilrad}(\Gamma,x)$$
where  $\varepsilon_0= \varepsilon_0 (P_0,r_0)$ always denotes the generalized Margulis constant.\linebreak 
By convention we set $\textup{nilrad}^+( \Gamma, X)=-\infty$ if $X_{\varepsilon_0} = \emptyset$.
Then, by taking in (\ref{formula:sys-nil}) the supremum over all $x\in X_{\varepsilon_0}$ we deduce the formula:
\begin{equation}\label{sys-uppernil}
	\textup{sys}^\diamond (\Gamma, X)\geq \min\left\lbrace \varepsilon_0, \frac{1}{H_0}e^{-H_0 \cdot  \text{nilrad}^+( \Gamma , X)} \right\rbrace
\end{equation}
which proves in particular Corollary \ref{cor-systole-intro} when $\Gamma$ is torsionless.
\vspace{2mm}

 Notice  that the upper nilradius  can  well be infinite, even when $X_{\varepsilon_0} \! \neq \! \emptyset$. \linebreak
For instance, for every elementary group $\Gamma$, we have  $\textup{nilrad}^+( \Gamma, X) = +\infty$ 	by Corollary \ref{elementary-nilpotent}, since   $(X,\sigma)$ is packed.
 Here are some non-trivial examples:  

\begin{ex}\label{ex-parabolics}
Let $(X,\sigma)$ be a $\delta$-hyperbolic  \textup{GCB}-space that is $P_0$-packed at scale $r_0>0$.
	Let $\Gamma$ be a discrete group of $\sigma$-isometries of $X$ containing a parabolic element: then $\sup_{x\in X_{r}}\text{nilrad}(\Gamma,x)= +\infty$ for all $r > 0$.\\
	Actually let $g$ be a parabolic element of $\Gamma$, in particular $\ell(g) = 0$. We take a sequence of points $x_n \in X$ such that $d(x_n,gx_n)\leq \frac{1}{n}$. Thus $\text{sys}^\diamond(\Gamma, x_n) \leq \frac{1}{n}$. So for any fixed $r> 0$ we have points $x_n$ such that $\text{sys}^\diamond(\Gamma, x_n) \leq r$ and the nilpotence radius at $x$ is larger and larger by the corollary.
\end{ex}

Clearly if $\Gamma$ is a cocompact group one has $\text{nilrad}^+(\Gamma, X) \leq 2 \text{diam}(\Gamma \backslash X)$. However there are many non-compact examples where the upper nilradius is finite and the estimate $(\ref{sys-uppernil})$ is non-trivial:


\begin{ex}[Quasiconvex-cocompact groups]\label{ex-coco}${}$\\
	A discrete group $\Gamma$ of isometries of a  $\delta$-hyperbolic space $X$ is called {\em quasiconvex-cocompact} if it acts cocompactly on the quasiconvex-hull of its limit set $\Lambda(\Gamma)$. The latter is defined as the union of all the geodesic lines joining two points of $\Lambda(\Gamma)$ and is clearly a $\Gamma$-invariant subset of $X$ denoted by QC-Hull$(\Lambda(\Gamma))$. We define the {\em codiameter} of $\Gamma$ acting on $X$  as the infimum   $D>0$ such that the orbit $\Gamma x$ is $D$-dense in QC-Hull$(\Lambda(\Gamma))$ for all $x\in$QC-Hull$(\Lambda(\Gamma))$.\\
	Consider now a quasiconvex-cocompact group of $\sigma$-isometries $\Gamma$ of a $\delta$-hyper\-bolic GCB-space $(X,\sigma)$ that is $P_0$-packed at scale $r_0$.
	It is 	classical that, for all $x\in$QC-Hull$(\Lambda(\Gamma))$, the subset $\Sigma_{2D}(x)$ of elements $g$ of $\Gamma$ such that $d(x,gx) \leq 2D$ generates $\Gamma$. 
	\\ 
	We affirm that there exists $s_0 = s_0(P_0,r_0,\delta; D)$ such that sys$^\diamond(\Gamma,X)\geq s_0$. Indeed since the action is quasiconvex-cocompact then any element	of $\Gamma$ is either elliptic or hyperbolic. Therefore the infimum defining the free systole equals the infimum over points belonging to the axes of all hyperbolic elements of $\Gamma$. Any such axis is  in QC-Hull$(\Lambda(\Gamma))$ by definition, so
\vspace{-3mm}	
	
	$$\text{sys}^\diamond(\Gamma, X) = \inf_{x\in \text{QC-Hull}(\Lambda(\Gamma))}\text{sys}^\diamond(\Gamma,x).$$
	By \eqref{formula:sys-nil} we conclude that $\text{sys}^\diamond(\Gamma, X) \geq \min\left\lbrace \varepsilon_0, \frac{1}{H_0}e^{-2H_0D}\right\rbrace =:s_0.$\\
	Now, for any point $x\in X_{\varepsilon_0}$ by definition there exists $g\in \Gamma$ such that $d(x,gx)\leq \varepsilon_0$ and we denote one of its axis by Ax$(g)$. By applying Proposition \ref{Margulis-estimate} to all $\varepsilon_1\in [s_0, \varepsilon_0]$
we deduce
\vspace{-3mm}	

	$$d(x,\text{Ax}(g))\leq K_0(P_0,r_0,\delta,  \ell(g), \varepsilon_0) =K(P_0,r_0,\delta; D)=:K.$$
	Moreover the axis of $g$ belongs to QC-Hull$(\Lambda(\Gamma))$, therefore it is easy to conclude that nilrad$(\Gamma, x) \leq 2K + 2D$. This shows that nilrad$^+(\Gamma, X)$ is finite, bounded by $2(K + D)$.
\end{ex}
%
%

\begin{ex}[Abelian covers]\label{ex-sub}${}$\\
	\noindent Let $\Gamma$ be a torsionless group of $\sigma$-isometries of a $\delta$-hyperbolic GCB-space $(X,\sigma)$ that is $P_0$-packed at scale $r_0>0$.
	Let $\Gamma_0 =[\Gamma,\Gamma]$ be the commutator subgroup.
	We affirm that 
\vspace{-3mm}
	
	$$\textup{nilrad}^+(\Gamma_0, X) \leq 5 \cdot \textup{nilrad}^+(\Gamma, X)$$ 
	(in particular nilrad$^+(\Gamma_0, X)$ is finite for any quasiconvex-cocompact $\Gamma$).\\
\noindent	Actually, assume that $\textup{nilrad}^+(\Gamma, X) <D$ and let $a,b\in \Gamma$ elements which generate a non-virtually nilpotent  (hence non-elementary) subgroup and  satisfy $d(x,ax)<D$, $d(x,bx)<D$.
	Then also the elements  $a'=a^{-1}[a,b]a$ and $b'=b^{-1}[a,b]b$ generate a non-elementary (hence non-virtually nilpotent) subgroup of $\Gamma$, by the same argument used in the last lines of the proof of Proposition \ref{kapo}. However $a'$ and $b'$ belong to $\Gamma_0$ and both move $x$ less than $5D$, which proves the claim.\\
	Notice that $\Gamma_0 \backslash X$ is not compact provided that the abelianization $\Gamma \slash \Gamma_0$ of $\Gamma$ is infinite.
\end{ex}

Corollary \ref{cor-systole} says that given $g \in\Gamma\setminus \Gamma^\diamond$ and $x \in X$ an upper bound of the displacement of this point by any other element of $\Gamma$ which does not generates with $g$ an elementary subgroup yields
a corresponding  lower bound of  the displacement $d(x,gx)$. Reversing the inequality  \eqref{formula:sys-nil} we obtain:

\begin{lemma}
	\label{Rr-elementary}
	Let $\varepsilon$ be smaller than the generalized Margulis constant  $\varepsilon_0$. \\
	For any given $x\in X_\varepsilon^\diamond$  the group $\Gamma_{R_\varepsilon} (x)$ is elementary  for
\vspace{-3mm}

 $$R_\varepsilon = \frac{1}{H_0} \cdot \log \left(\frac{1}{  \varepsilon \cdot H_0 } \right) $$
\noindent where   $H_0=H_0(P_0,r_0,\delta)$ is the constant given in Corollary \ref{cor-systole}.
\end{lemma}



\subsection{Lower bound for the diastole}\label{sub-diastole}

The estimate of the diastole given in Corollary \ref{cor-diastole} stems  from the application of the  classical  Tits Alternative  
combined with Breuillard-Green-Tao's  generalized Margulis Lemma.
We state here the version allowing torsion, from which Corollary \ref{cor-diastole} easily follows.\\
Recall that the  {\em free diastole} of $\Gamma$ acting on $X$ is 
$\textup{dias}^\diamond ( \Gamma, X)= \sup_{x\in X} \text{sys}^\diamond (\Gamma,x)$,
and that the {\em free $r$-thin subset}  of $X$ is defined as  
$$ X_r^\diamond = \{ x \in X   \hspace{1mm} | \hspace{1mm} \exists  g \in \Gamma\setminus \Gamma^\diamond  \text{ s.t. } d(x,gx) < r \}.$$

\noindent We then have:
\begin{cor}\label{cor-freediastole} Let $(X,\sigma)$ be a $\delta$-hyperbolic  \textup{GCB}-space that is $P_0$-packed at scale $r_0>0$.
	Then for any non-elementary, discrete group of $\sigma$-isometries $\Gamma$  of $X$ we have:
	$$\textup{dias}^\diamond (\Gamma, X)  \geq  \varepsilon_0$$
	where  $\varepsilon_0= \varepsilon_0 (P_0,r_0)$ always denotes the generalized Margulis constant.
\end{cor}


%
%

\begin{proof}[Proof of Corollary \ref{cor-freediastole}]
	Consider the  free $r$-thin subset $X_r^\diamond$.
	We first show that if $r \leq \varepsilon_0(P_0,r_0)$ then $X_r^\diamond$ is not connected.
	By definition $\forall x\in X_r^\diamond$ the group $\Gamma_r(x)$ is   virtually nilpotent and  contains a torsionless element $a$.  \linebreak For any such $x$ we denote by $N_r (x)$ the unique maximal elementary subgroup of $\Gamma$ containing $\Gamma_r (x)$ given by Lemma \ref{M-elementary}. Moreover, since there are only   finitely many $g \in \Gamma$ such that $d(x,gx)< r$,  there exists $\eta > 0$ such that for all $g \in  \Gamma_r (x)$ it holds  $d(x,gx)\leq r - 2\eta$. In particular if   $d(x,x')<\eta$ then $d(x',ax')<r$, so $x'\in X_r^\diamond$ too, and $\Gamma_r(x)\subset \Gamma_r(x')$. This implies that the map $x\mapsto N_r (x)$ is locally constant. So if $X_r^\diamond$ was connected we would have that $N_r (x)$ is a fixed elementary  subgroup $N_r  $ which  does not depend  on  $x\in X_r^\diamond$. We  show now that $N_r  $ is a normal subgroup of $\Gamma$:  indeed  $\forall g \in \Gamma$ and $\forall x\in X_r^\diamond$ we have that $gx \in X_r^\diamond$   and $\Gamma_r(gx)=g\Gamma_r(x)g^{-1}$, therefore $g N_r g^{-1} = N_r $. 
	As $\Gamma$ is non-elementary there exists a non-elliptic isometry $b\in \Gamma$ such that the group $\langle   N_r  , b\rangle$ is non-elementary, by the maximality of $N_r  $. \linebreak
	Now, the group $N_r$ is normal in $\Gamma$ and amenable (since by Corollary \ref{elementary-nilpotent}  it is virtually nilpotent), therefore its cyclic extension $\langle  N_r, b  \rangle$ is amenable. 
 However we know that any non-elementary group contains a free group (cp. \cite{Gro87}) and this is impossible for an amenable group. 
	This shows that $X_r^\diamond$ is not connected and in particular  $X_r^\diamond \neq X$. Therefore there exists a point $x\in X$ such that $d(x,gx)\geq r$ for every $g\in \Gamma^\diamond$.
\end{proof}

\subsection{Compactness and rigidity}
\label{sec-compactness}
In this last section we  investigate the convergence under ultralimits of group actions on $\delta$-hyperbolic GCB-spaces that are $P_0$-packed at scale $r_0$.\\
We denote by GCB$(P_0,r_0)$ the class of pointed GCB-spaces $(X,x,\sigma)$ that are $P_0$-packed at scale $r_0$, and  by GCB$(P_0,r_0,\delta)$ the subclass made of   $\delta$-hyperbolic spaces.
We refer to \cite{CavS20} and  \cite{Cav21survey} for details about ultralimits and the  relation with Gromov-Hausdorff convergence. \\
These classes are closed under ultralimits and compact under pointed Gromov-Hausdorff convergence. For  GCB$(P_0,r_0)$, this property follows exactly as in the proof of Theorem 6.1 in \cite{CavS20}. The closure of  GCB$(P_0,r_0,\delta)$  then follows from the fact that  the $\delta$-hyperbolicity condition, as expressed in \eqref{hyperbolicity}, is clearly stable under ultralimits. 
\vspace{1mm}

	We start with a general construction. Let $(X_n,x_n)$ be any sequence of metric spaces and suppose to have for every $n$ a group of isometries $\Gamma_n$ of $X_n$. We fix a non-principal ultrafilter $\omega$.
	We define a limit group of isometries $\Gamma_\omega$ of $X_\omega$. 
	We say that a sequence of isometries $( g_n )$, with each $g_n \in \Gamma_n$, is {\em admissible} if there exists $M < +\infty $ such that $d(x_n, g_n x_n)\leq M$  for $\omega$-a.e.$(n)$.  Every admissible sequence $( g_n )$ defines a limit  isometry $g_\omega = \omega$-$\lim g_n$ of $X_\omega$ by the formula 
	$$g_\omega (y_\omega) = \omega \mbox{-}\lim(g_n (y_n))$$
	where $y_\omega = \omega$-$\lim y_n$ is a generic point of $X_\omega$, see \cite{CavS20}, Proposition A.5.
	We then define:
	$$\Gamma_\omega := \lbrace \omega\text{-}\lim g_n  \hspace{2mm} | \hspace{2mm} (g_n) \text{ admissible sequence, } g_n \in \Gamma_n \hspace{2mm}\forall  n\rbrace.$$
	The following lemma is straightforward:

\enlargethispage{1\baselineskip}
	\begin{lemma}
		\label{composition}
		The composition of admissible sequences of isometries is an admissible sequence of isometries and the limit of the composition is the composition of the limits.
	\end{lemma}
	\noindent  (Indeed, if  $g_\omega = \omega$-$\lim g_n$, $h_\omega = \omega$-$\lim h_n$ belong to $\Gamma_\omega$ then  their composition belong to $\Gamma_\omega$, as $d(g_n h_n \cdot x_n, x_n)\leq d(g_n h_n \cdot x_n, g_n \cdot x_n) + d(g_n \cdot x_n, x_n) < + \infty$).

	\noindent Analogously one proves that $(\text{id}_n)$ belongs to $\Gamma_\omega$ and defines the identity map of $X_\omega$, and that  if $g_\omega = \omega$-$\lim g_n$ belongs to $\Gamma_\omega$ then also the sequence $(g_n^{-1})$ defines an element of $\Gamma_\omega$, which is the inverse of $g_\omega$.
	
	\noindent So we have a well defined composition law on $\Gamma_\omega$, that is for $ g_\omega   = \omega \mbox{-}\lim g_n$ and  $ h_\omega = \omega \mbox{-}\lim h_n$ we set
	$$g_\omega \, \circ\, h_\omega = \omega\mbox{-}\lim(g_n \circ h_n)$$
	With this operation $\Gamma_\omega$ is a group of isometries of $X_\omega$ and it is called {\em the ultralimit group} of the sequence of groups $\Gamma_n$. Remark that the limit group depends   on the choice of base points $x_n $ in $X_n$ (in particular, sequences of isometries $(g_n)$ whose displacement at $x$ tend to infinity do not appear in the limit group).
	
	Clearly if we have a sequence $(X_n,x_n,\sigma_n)$, where $\sigma_n$ is a GCB on $X_n$ for every $n$, and a sequence of groups of $\sigma_n$-isometries $\Gamma_n$ of $X_n$, then $\Gamma_\omega$ is a set of $\sigma_\omega$-isometries of $X_\omega$, where $\sigma_\omega$ is the GCB defined in Lemma \ref{bicombing-limit}.
\vspace{1mm}

In the following proposition we describe the possible ultralimits of an admissible sequence of isometries in our setting:

\begin{prop}
	\label{ultralimit-isometry}
	 Let $(X_n,x_n,\sigma_n)$   be a sequence of spaces in $\textup{GCB}(P_0,r_0,\delta)$.
	Let  $\omega$ be a non-principal ultrafilter, and  let $(g_n)$ be any admissible sequence of $\sigma_n$-isometries.  
	\begin{itemize}
		\item[(i)] If $g_n$ is of hyperbolic type with $\sigma_n$-axis $\gamma_n$ for $\omega$-a.e.$(n)$  then
		\begin{itemize}
	\item[(i.1)]  if $\omega$-$\lim d(x_n, \gamma_n) < +\infty$ then  $g_\omega$ is  elliptic when   $\omega$-$\lim \ell(g_n) = 0$, and  \linebreak  hyperbolic  with $\sigma_\omega$-axis $\gamma_\omega = \omega$-$\lim \gamma_n$ and  $\ell(g_\omega)= \omega$-$\lim \ell (g_n)$ otherwise;
			\item[(i.2)] if $\omega$-$\lim d(x_n, \gamma_n) = +\infty$ then $g_\omega$  is either elliptic or parabolic.		\end{itemize}
		\item[(ii)] If $g_n$ is parabolic for $\omega$-a.e.$(n)$ then $g_\omega$ is either  elliptic or parabolic.
	\end{itemize}	
\end{prop}

\noindent  Notice that any two $\sigma_n$-axes of a hyperbolic isometry are at uniformly bounded distance from each other, by the $\delta$-hyperbolicity assumption, cp. Lemma \ref{Morse}, so (i) does not depend on the particular choice of $\gamma_n$. Moreover the ultralimit of a sequence of $\sigma_n$-geodesics $\gamma_n$ at uniformily bounded distance from the base points $x_n$ is again a $\sigma_\omega$-geodesic of $X_\omega$ (cp.  Proposition A.5 of \cite{CavS20} and the definition of $\sigma_\omega$ in Lemma \ref{bicombing-limit}).

\begin{ex} \label{exip-para}${}$  Case  (i.2)  can actually occur for the limit $g_\omega$ of a sequence of hyperbolic isometries $g_n$.  
Let for instance   $\Gamma_n $ be a Schottky group of $X=\mathbb{H}^2$ generated by two hyperbolic isometries $a_n, b_n$ with non-intersecting axes. \linebreak
The convex core of the quotient space $\bar X_n = G_n \backslash \mathbb{H}^2$  is a hyperbolic pair of pants, with boundary given by three periodic geodesics $\alpha_n, \beta_n$ and $\gamma_n$.
These geodesics  correspond, respectively,  to the projections of the axes of the elements $a_n$, $b_n$ and $c_n = a_n \cdot b_n$ (up to replacing $b$ with its inverse). Letting the length $\ell(\gamma_n)$ tend to zero (which means  pulling two of the isometry circles of $a$ and $b$ closer and closer), the sequence of hyperbolic isometries  $c_n$ tends to a parabolic isometry, and $\bar X_n$ tends to a surface with one cusp. 
\end{ex}

\begin{proof}[Proof of Proposition \ref{ultralimit-isometry}] ${}$\\
	We start from $g_n$ of hyperbolic type with axis $\gamma_n$ for $\omega$-a.e.$(n)$. \linebreak
	Assume first that $\omega$-$\lim d(x_n, \gamma_n) = C <+\infty$. If $\omega$-$\lim \ell(g_n) = 0$ then any point of   the limit $\sigma_\omega$-geodesic $\gamma_\omega$ is a fixed point of  $g_\omega$, so $g_\omega$ is  elliptic. \linebreak
	Otherwise $\omega$-$\lim \ell(g_n) = \ell >0$ and it is immediate that  $g_\omega$  translates $\gamma_\omega$ by $\ell$, \linebreak hence it is of hyperbolic type with $\sigma_\omega$-axis $\gamma_\omega$.\\
Suppose now that $g_n$ is of hyperbolic type with $\sigma_n$-axis $\gamma_n$ for $\omega$-a.e.$(n)$  and that $\omega$-$\lim d(x_n,\gamma_n) = +\infty$. 	
Let $M_0\geq 0$ be an upper bound for $d(x_n, g_n x_n)$ for every $n$. A direct application of Proposition \ref{Margulis-nonexplicit} gives $\omega$-$\lim \ell(g_n) = 0$, otherwise the distance between $x_n$ and the axis $\gamma_n$ of $g_n$ would be uniformly bounded. 
Suppose that $g_\omega$ is hyperbolic: 
in this case $\ell(g_\omega) = \ell_0$ would be strictly positive. Applying again Proposition \ref{Margulis-nonexplicit} we would find, for $\omega$-a.e.$(n)$, a point $p_n \in X_n$ satisfying
\vspace{-3mm}

$$d(p_n, x_n)\leq K_1(\ell_0, M_0, \delta),\qquad d(p_n,g_n p_n)\leq \frac{\ell_0}{2}.$$
The first condition implies that the sequence $(p_n)$ defines a point $p_\omega$ of $X_\omega$, while the second condition implies that $d(p_\omega, g_\omega p_\omega) \leq \frac{\ell_0}{2}$ which is impossible, so $g_\omega$ is no of hyperbolic type.

\noindent Finally, suppose that $g_n$ is of parabolic type for $\omega$-a.e.$(n)$. If $g_\omega$ was hyperbolic of translation length $\ell_0 > 0$ then arguing as before there would exist  a point $p_n \in X_n$ satisfying
\vspace{-3mm}

	$$d(p_n, x_n)\leq K_1(\ell_0,M_0,\delta),\qquad d(p_n,g_n p_n)\leq \frac{\ell_0}{2}.$$
	This is again a contradiction.
\end{proof}

The next theorem explains how the ultralimit  of a sequence of torsion-free, discrete groups of isometries in our setting can degenerate, that is, when the limit is  non-discrete or  admits elliptic elements:

\begin{theo}
	\label{Dicotomia}
	 Let $(X_n,x_n,\sigma_n)$   be a sequence of spaces in $\textup{GCB}(P_0,r_0,\delta)$, and   $\Gamma_n$  a sequence of torsion-free, discrete groups of $\sigma_n$-isometries of $X_n$. \linebreak
	Let $\omega$ be a non-principal ultrafilter and $\Gamma_\omega$   the limit group of $\sigma_\omega$-isometries of $X_\omega$.   Then one of the following  mutually esclusive possibilities holds:
	\begin{itemize}
		\item[(a)] $\forall L \geq 0$ $\exists r > 0$ such that $\omega\textup{-}\lim d(x_n, (X_n)_{r}) > L$. In this case the group $\Gamma_\omega$ acts discretely on $X_\omega$ and it has no torsion; 
		\item[(b)] $\exists L \geq 0$ such that $\forall r > 0$ it holds $\omega\textup{-}\lim d(x_n, (X_n)_r) \leq L$. In this case the group $\Gamma_\omega$ is elementary (possibly non-discrete).
	\end{itemize}
\end{theo}
\begin{proof}
 	We start from case (a). By Proposition \ref{packingsmallscales} we know that $X_\omega$ is proper. \linebreak 
	Let $g_\omega = \omega$-$\lim g_n$ be an element of $\Gamma_\omega$ and $y_\omega = \omega$-$\lim y_n$ be a point of $X_\omega$. \linebreak 
	By definition of $y_\omega$ there exists $L\geq 0$ such that $d(x_n,y_n)\leq L$ for all $n$. \linebreak
	By assumption, there exists $r$ such that  $d(y_n, g_n y_n )\geq r$    for $\omega$-a.e.$(n)$, if $g_n \neq \text{id}$ for $\omega$-a.e.$(n)$. This implies $d(y_\omega, g_\omega y_\omega )\geq r$,  so    sys$(\Gamma_\omega, y_\omega) \geq r$ \linebreak for all $y_\omega \in \overline{B}(x_\omega, L)$. Since $X_\omega$ is proper we conclude that $\Gamma_\omega$ is discrete. Moreover it is torsion-free: indeed any elliptic element $g_\omega=(g_n)$ of $\Gamma_\omega$ must have a fixed point $y_\omega$, hence,  as just proved,   $g_n$ is the identity for $\omega$-a.e.$(n)$; so   $g_\omega = \text{id}$ necessarily, since   sys$(\Gamma_\omega,y_\omega)$  is strictly positive.	
	
 \noindent We 	now  study   case (b).  In this case for all $r>0$  there exists a point $y_n \in X_n$ with $d(x_n,y_n)\leq L$ and sys$( \Gamma_n, y_n)\leq r$  for all $\omega$-a.e.$(n)$.  
 	Observe that for all $r\leq \varepsilon_0$ the group $\Gamma_{R_r}(y_n)$ is elementary, with $R_r \to +\infty$ when $r\to 0$ by Lemma \ref{Rr-elementary}. 
	Now,   by definition, for every  $g_\omega = \omega$-$\lim g_n$ of $\Gamma_\omega$ there exists $M$ such that $d(x_n, g_n x_n) \leq M$ for $\omega$-a.e.$(n)$, so $d(y_n, g_n y_n)\leq 2L + M \leq R_r$ provided that $r$ is small enough. 
	This implies that $g_n$ belongs,  for $\omega$-a.e.$(n)$, to a fixed elementary subgroup $\Gamma_n' < \Gamma_n$  that does not depend on the element $g_\omega$ under consideration. Then, there are two possibilities: for $\omega$-a.e.$(n)$ either $\Gamma_n'$ is of hyperbolic type or it is of parabolic type. \\
Assume that the isometries in  $\Gamma_n'$ are all hyperbolic  for $\omega$-a.e.$(n)$, so there exists a common $\sigma_n$-axis $\gamma_n$ for all of them. If  $\omega$-$\lim d(x_n, \gamma_n) < + \infty$  and $\omega$-$\lim \ell(g_n) >0$, then we are in case (i.1) of  Proposition \ref{ultralimit-isometry}:  the limit $g_\omega$   is hyperbolic with $\sigma_\omega$-axis $\gamma_\omega = \omega\mbox{-}\lim \gamma_n$, for all $g_\omega \in \Gamma_\omega$, hence this group is elementary (and discrete).
In all the other cases,  Proposition \ref{ultralimit-isometry} implies that  $\Gamma_\omega$ does not contain any hyperbolic isometry, so the group is elementary by Gromov' classification of groups acting on hyperbolic spaces (cp. \cite{Gro87},\cite{DSU17}). 
\end{proof}

%
%
%

We examine now the case when the limit group is discrete:  

\begin{cor}
	\label{Dicotomia-geometrica}
Same assumptions as in Theorem \ref{Dicotomia}.\\
	 Let $\bar X_n=  \Gamma_n \backslash X_n$  be the quotient metric spaces,   let $p_n\colon X_n \to \bar X_n$ the \linebreak projection maps and let  $\bar x_n = p_n(x_n)$. Then:
\begin{itemize}
\item[(i)]  if the groups $\Gamma_n$  are non-elementary then one can always suitably choose the base points $x_n \in X_n$ so  that  case (a) of Theorem \ref{Dicotomia} occurs, hence the ultralimit group $\Gamma_\omega$ is discrete and torsion-free;	
\item[(ii)] if case (a) of Theorem \ref{Dicotomia} occurs then the ultralimit space $\bar X_\omega$ of the sequence $(\bar X_n, \bar x_n)$ is isometric to $\Gamma_\omega \backslash X_\omega$.
\end{itemize}	 
\end{cor}

\begin{proof}
By Corollary \ref{cor-freediastole}, if the groups $\Gamma_n$ are non-elementary it is always possible to choose $x_n \in X_n$ so that the pointwise systole of the $\Gamma_n$ at $x_n$ are uniformly bounded away from zero, i.e. $\text{sys}(\Gamma_n,x_n)\geq \varepsilon > 0$ for every $n$. The fact that case (a) occurs for this choice of the base points is then a direct consequence of  Proposition \ref{lemma-lowerdistance}.(ii).\\ 
To show (ii), notice that the projections $p_n\colon X_n \to  \bar X_n$ form an admissible sequence of $1$-Lipschitz maps and then, by Proposition A.5 of \cite{CavS20}, they yield a limit map $p_\omega\colon X_\omega \to  \bar X_\omega$ defined as $p_\omega ( y_\omega) = \omega$-$\lim p_n(   y_n)$, for $\omega$-$\lim  y_n =  y_\omega$. 
	The map $p_\omega$ is clearly surjective. We want to show that it is $\Gamma_\omega$-equivariant. We fix $g_\omega = \omega$-$\lim g_n \in \Gamma_\omega$ and $y_\omega = \omega$-$\lim y_n \in X_\omega$. Then:
	\vspace{-3mm}
	
	$$p_\omega (\gamma_\omega   y_\omega) = \omega\text{-}\lim p_n(\gamma_n y_n) = \omega\text{-}\lim p_n(y_n) = p_\omega(y_\omega).$$
	Therefore we have a well defined, surjective quotient map $\bar{p}_\omega\colon \Gamma_\omega\backslash X_\omega \to \bar X_\omega$. \linebreak   We will now show  that it is a local isometry. We fix an arbitrary point \linebreak $y_\omega = \omega\text{-}\lim y_n \in X_\omega$,  consider its class $[y_\omega]  \in \Gamma_\omega\backslash X_\omega$ and set $L=d(x_\omega, y_\omega)$. \linebreak By assumption there exists $r$, depending only on $L$, such that sys$(\Gamma_n, y_n) \geq r$ for $\omega$-a.e.$(n)$. In particular the systole of $\Gamma_\omega$ at $y_\omega$ is at least $r$, so the quotient map $X_\omega \to \Gamma_\omega\backslash X_\omega$ is an isometry between $\overline{B}(y_\omega, \frac{r}{2})$ and $\overline{B}([y_\omega],\frac{r}{2})$. \linebreak
	Moreover for $\omega$-a.e.$(n)$ we have that $\overline{B}(p_n(y_n), \frac{r}{2})$ is isometric to $\overline{B}(y_n,\frac{r}{2})$. By Lemma A.8 of \cite{CavS20} we know that  $\omega$-$\lim \overline{B}(p_n(y_n), \frac{r}{2})$ is isometric to $\overline{B}(p_\omega(y_\omega), \frac{r}{2})=\overline{B}(\bar p_\omega(y_\omega), \frac{r}{2})$ and that  $\omega$-$\lim \overline{B}(y_n, \frac{r}{2})$ is isometric to $\overline{B}(y_\omega, \frac{r}{2})$.  Therefore $\overline{B}(\bar p_\omega(y_\omega), \frac{r}{2})$ is isometric to $\overline{B}([y_\omega],\frac{r}{2})$. By Proposition 3.28, Sec.I, of \cite{BH09} we conclude that the map $\bar{p}_\omega$ is a locally isometric covering map. To conclude, it is enough to show that  $\bar{p}_\omega$ is injective. Let $[z_\omega], [y_\omega] \in \Gamma_\omega\backslash X_\omega$. Then we have
	$\bar{p}_\omega([z_\omega]) = \bar{p}_\omega([y_\omega])$ if and only if $p_\omega(z_\omega)=p_\omega(y_\omega)$. This is equivalent to $\omega\text{-}\lim d(p_n (z_n), p_n(y_n)) = 0$ and, as the systole of $y_n$ is uniformly bounded away from zero,  this means   $\omega$-$\lim d(z_n, g_n y_n) = 0$ for some $g_n\in \Gamma_n$ and for $\omega$-a.e.$(n)$. We observe that the sequence $(g_n)$ is admissible, therefore it defines an element $g_\omega = \omega$-$\lim g_n \in \Gamma_\omega$ satisfying $d(z_\omega,g_\omega y_\omega)=0$. This implies that $[z_\omega] = [y_\omega]$ and therefore  $\bar{p}_\omega$ is an isometry.
\end{proof}

Non-elementary ultralimit groups are characterized in the next result:
\begin{theo}
	\label{elementarita}
	Same assumptions as in Theorem \ref{Dicotomia}.
	\begin{itemize}
		\item[(i)] If there exist two sequences of admissible isometries $(g_n),(h_n)$ in $\Gamma_n$ \linebreak of the same type such that $\langle g_n, h_n \rangle$ is non-elementary for $\omega$-a.e.$(n)$ then the group $\langle g_\omega, h_\omega\rangle$ is non-elementary;
		\item[(ii)] the group $\Gamma_\omega$ is non-elementary if and only if there exist two sequences of admissible isometries $(g_n),(h_n)$  such that $\langle g_n, h_n \rangle$ is non-elementary for $\omega$-a.e.$(n)$.
	\end{itemize}
\end{theo}

\noindent Notice that $\Gamma_\omega$ non-elementary implies that it is also discrete and torsion-free, by Theorem \ref{Dicotomia}.

\begin{proof}
	Let $(g_n), (h_n)$ be as in   (i) and let $M \geq 0$ such that for all $n$ it holds $d(x_n,g_n x_n), d(x_n, h_n x_n)\leq M$. 
We first show that  there are no elliptics  in  $\langle g_\omega, h_\omega \rangle$.
 Actually, assume that $f_\omega \in \langle g_\omega, h_\omega \rangle$ is elliptic, with  $f_\omega= \omega\mbox{-}\lim f_n$, for some admissible sequence of isometries $f_n \in  \langle g_n, h_n \rangle$.
 Then,  there would exist a point $y_\omega = \omega$-$\lim y_n$ with $f_\omega y_\omega = y_\omega$. 
 So,  for all $r > 0$ and for $\omega$-a.e.$(n)$ the following conditions would hold, for some $L\geq 0$:
			$$d(x_n,y_n)\leq L, \qquad d(y_n, f_n y_n) \leq r.$$
			The first condition implies that
			$$d(y_n, g_n y_n) \leq d(y_n, x_n) + d(x_n, g_n x_n) + d(g_n x_n, g_n y_n) \leq 2L + M$$
			and similarly for $h_n$.
			If $r$ is small enough we then deduce that $\langle g_n, h_n \rangle$ is elementary by Lemma \ref{Rr-elementary}, a contradiction. \\
Assume now that the elements   $g_\omega, h_\omega$ are both parabolic.
If $\langle g_\omega, h_\omega \rangle$ was elementary then they would have the same fixed point at infinity $z$. 
We then choose $\varepsilon>0$ small enough so that $R_\varepsilon \geq 16\delta + \varepsilon$, where $R_\varepsilon$ is the quantity defined in Lemma \ref{Rr-elementary}. As  $\ell(g_\omega)=\ell( h_\omega)=0$  there exist  points $y_\omega = \omega$-$\lim y_n$, $w_\omega=\omega$-$\lim w_n$ of $X_\omega$ such that $d(y_\omega, g_\omega y_\omega) < \varepsilon$ and $d(w_\omega,h_\omega w_\omega) < \varepsilon$.
			By Lemma \ref{parallel-geodesics} we can also find points $y_\omega' \in [y_\omega, z]$ and $w_\omega' \in [w_\omega, z]$ such that $d(y_\omega', w_\omega') \leq 8\delta$. By convexity and by the triangular inequality we deduce:
			$$d(y_\omega' , g_\omega y_\omega' ) < \varepsilon, \qquad d(y_\omega', h_\omega y_\omega') < 16\delta + \varepsilon \leq R_\varepsilon.$$
			Similar estimates hold for   $g_n$ and $h_n$ for $\omega$-a.e.$(n)$, implying that $\langle g_n, h_n\rangle$ is elementary for $\omega$-a.e.$(n)$, a contradiction. \\
A similar argument works when one element is parabolic, say  $g_\omega$, and the other, $h_\omega$, is hyperbolic.
In this case, if  $\langle g_\omega, h_\omega\rangle$ was elementary,   the fixed point of $g_\omega$ would coincide with   one point  at infinity $z$ of an axis $\eta$ of $h_\omega$. 
We choose $\varepsilon > 0$ so that $R_\varepsilon > \ell_0 + 16\delta$, where $R_\varepsilon$ is again the number given by Lemma \ref{Rr-elementary} and  $\ell_0$ is the minimal displacement   of $h_\omega$. \linebreak
We then take a point $z_\omega = \omega$-$\lim z_n$ of $X_\omega$ such that $d(z_\omega, g_\omega z_\omega) < \varepsilon$,  a point $y_\omega = \omega$-$\lim y_n$ on  $\eta$,  and a point $z_\omega'=\omega$-$\lim z_n' \in [z_\omega, z]$   such that $d(y_\omega, z_\omega') \leq 8\delta$.
			By  convexity of $\sigma_\omega$ (since $g_\omega$ is a $\sigma_\omega$-isometry) and the triangular inequality  we get
			$$d(z_\omega', g_\omega z_\omega') < \varepsilon,\qquad d(z_\omega', h_\omega z_\omega') \leq 16\delta + \ell_0 < R_\varepsilon.$$
		and again	similar estimates hold   for $ g_n, h_n$ for $\omega$-a.e.$(n)$, showing that $\langle g_n, h_n\rangle$ is elementary for $\omega$-a.e.$(n)$, a contradiction.

\noindent It remains to consider the case where both $g_\omega$ and $h_\omega$ are of hyperbolic type.
			Suppose they have the same $\sigma_\omega$-axis. By Lemma \ref{ultralimit-isometry} we know that this $\sigma_\omega$-axis is the ultralimit  of some $\sigma_n$-axis $\gamma_n$ of $g_n$, and of some $\sigma_n$-axis $\eta_n$ of $h_n$ as well; therefore $\omega$-$\lim \gamma_n = \omega$-$\lim \eta_n$. This means that for all $C>0$ and for $\omega$-a.e.$(n)$ the set of points of $\gamma_n$ that are at distance at most $\frac{\varepsilon_0}{37}$ from $\eta_n$ is a subsegment of length at least $C$, where $\varepsilon_0$ is the generalized Margulis constant. \linebreak
			By Proposition \ref{kapo} we conclude that $\ell(g_n) \geq \frac{1}{5}C$. Therefore the sequence $(g_n)$ is not admissible, a contradiction. This implies that $g_\omega$ and $h_\omega$ do not have the same axis, therefore $\langle g_\omega, h_\omega \rangle$ is not elementary. This proves (i).
		  
\vspace{1mm}
 \noindent 	In order to prove (ii), assume first that   $(g_n), (h_n)$  are two admissible sequences such that  $\langle g_n, h_n \rangle$ is not elementary for $\omega$-a.e.$(n)$. Up to replacing $h_n$ with $h_n g_n h_n^{-1}$ we may suppose that $g_n,h_n$ are of the same type, and still admissible. So $\Gamma_\omega$ is not elementary by (i). 
Conversely,	assume that  $\Gamma_\omega$ is not elementary. Then, it contains at least a hyperbolic element $g_\omega$ and, by Theorem \ref{Dicotomia} we know it is discrete.
Therefore, by discreteness and non-elementarity,  there exists another element $h_\omega \in \Gamma_\omega$ such that $\langle g_\omega, h_\omega \rangle$ is not elementary. 
	Up to replacing $h_\omega$ with $h_\omega g_\omega h_\omega^{-1}$ we may again suppose that $h_\omega$ is of hyperbolic type (and   the group genarated by $g_\omega$ and this element remains non-elementary). Hence the $\sigma_\omega$-axis $\gamma_\omega, \eta_\omega$ respectively  of $g_\omega = \omega$-$\lim g_n$ and $h_\omega = \omega$-$\lim h_n$ are not the same. By Lemma \ref{ultralimit-isometry}   the elements $g_n, h_n$ are  hyperbolic for $\omega$-a.e.$(n)$, and have   $\sigma_n$-axes $\gamma_n, \eta_n$ such that $\gamma_\omega=\omega$-$\lim \gamma_n$, and  $\eta_\omega = \omega$-$\lim \eta_n$. Now, if $\langle g_n, h_n \rangle$ was elementary for $\omega$-a.e.$(n)$ then we could choose $\gamma_n = \eta_n$ for $\omega$-a.e.$(n)$ and therefore $\gamma_\omega = \eta_\omega$, a contradiction. This shows that   $\langle g_n, h_n \rangle$ is not elementary for $\omega$-a.e.$(n)$. 
\end{proof}


Recall from the introduction that, for any fixed choice of  parameters $P_0,r_0,\delta, \Delta> 0$, we called 
\vspace{-3mm}

$$\text{GCB}(P_0,r_0,\delta; \Delta)$$
the class of spaces that are quotients of a space $(X,x, \sigma)\in$GCB$(P_0,r_0,\delta)$ by a discrete, torsion-free group of $\sigma$-isometries $\Gamma$  of $X$ with nilrad$^+(\Gamma, X)\leq \Delta$ 
(notice, then, that  the group  $\Gamma$ is non-elementary by assumption).\\
We will now prove that this class is closed under ultralimits, hence compact under pointed, equivariant Gromov-Hausdorff convergence.

\begin{proof}[Proof of Theorem   \ref{compactness-nilpotence}]
Let $(\bar X_n, \bar x_n, \bar \sigma_n)= \Gamma_n \backslash (X_n, x_n, \sigma_n)$ be any sequence in this class, with $ (X_n, x_n, \sigma_n) \in \text{GCB}(P_0,r_0,\delta)$.
As recalled at the beginning of Section \ref{sec-compactness}   the ultralimit   $(X_\omega, x_\omega, \sigma_\omega)$ of the sequence  $(X_n,x_n,\sigma_n)$  is again an element of GCB$(P_0,r_0,\delta)$.
Then, to prove that our class is closed under ultralimits, we need only to show that the bound of the upper nilradius is satisfied also by the limit group $\Gamma_\omega$ acting on $X_\omega$. Indeed it will be enough to apply Corollary \ref{Dicotomia-geometrica} to ensure that the quotient spaces $(\bar X_n, \bar x_n, \bar \sigma_n)$ converge to  $\Gamma_\omega \backslash ( X_\omega, x_\omega, \sigma_\omega)$, that belongs to GCB$(P_0,r_0,\delta;\Delta)$.
\linebreak 
	By the estimate \eqref{sys-uppernil} we know that sys$(X_n,\Gamma_n) \geq s_0(P_0,r_0,\delta, \Delta)$ for all $n$. \linebreak Therefore we are in case (a) of Theorem \ref{Dicotomia} and $\Gamma_\omega$ is a discrete and torsion-free group of isometries of $X_\omega$.  \\
Now, assume first  that  sys$(\Gamma_n, X_n)$ is greater than or equal to the the generalized Margulis constant $\varepsilon_0$ for $\omega$-a.e.$(n)$.  Then sys$(\Gamma_\omega, X_\omega)\geq \varepsilon_0$, so  nilrad$^+(\Gamma_\omega,  X_\omega ) = -\infty\leq \Delta$, and the conclusion holds.\\
Otherwise, sys$(\Gamma_n, X_n) < \varepsilon_0$ for $\omega$-a.e.$(n)$. In this case we take any \linebreak  $y_\omega = \omega$-$\lim y_n$ such that $s = \text{sys}(\Gamma_\omega, y_\omega) <\varepsilon_0$. By the discreteness of $\Gamma_\omega$ there exists $g_\omega = \omega$-$\lim g_n \in \Gamma_\omega$ such that $d(y_\omega, g_\omega y_\omega) = s$. We fix $\varepsilon < \varepsilon_0 - s$ and we deduce that  $d(y_n, g_n y_n) < s + \varepsilon < \varepsilon_0$  for $\omega$-a.e.$(n)$, so nilrad$^+(\Gamma_n, y_n)\leq \Delta$. This means that for all $\varepsilon > 0$ there is $h_n\in \Gamma_n$ such that $d(y_n, h_n y_n)\leq \Delta + \varepsilon$ and $\langle h_n, g_n \rangle$ is not elementary. To conclude, we need to show that $\langle g_\omega, h_\omega \rangle$ is not elementary. Assume the contrary: then,  $h_\omega$ has the same type and  the same fixed points at infinity as $g_\omega$.  
If they were   hyperbolic,  then by Lemma \ref{ultralimit-isometry}   also $g_n,h_n$ would be hyperbolic for $\omega$-a.e.$(n)$, and by Theorem \ref{elementarita} we would obtain that $\langle g_n, h_n \rangle$ is elementary for $\omega$-a.e.$(n)$, a contradiction.\linebreak
On the other hand, if both are parabolic then we have two possibilities: either $g_n,h_n$ are of the same type  for $\omega$-a.e.$(n)$, and  arguing as before   would give again a contradiction; or $g_n$ is hyperbolic and $h_n$ is parabolic for $\omega$-a.e.$(n)$. \linebreak In this last case we consider the elementary group $\langle g_\omega, h_\omega g_\omega h_\omega^{-1}\rangle$ and  apply Theorem \ref{elementarita} as before to deduce that   the group $\langle g_n, h_n g_n h_n^{-1} \rangle$ is  elementary for $\omega$-a.e.$(n)$. 
Therefore 
 $h_n \text{Fix}_\partial (g_n) = \text{Fix}_\partial(h_ngh_n^{-1}) = \text{Fix}_\partial (g_n)$, 
so the fixed point  of $h_n$ coincides with one of the  fixed points of $g_n$, which  contradicts the fact that $\Gamma_n$ is discrete. This shows that $\langle g_\omega, h_\omega \rangle$ is not elementary. 
By the arbitrariness of $\varepsilon$ we then obtain nilrad$^+(\Gamma_\omega, y_\omega) \leq \Delta$.
\end{proof}

We will finally prove Theorem \ref{compactness-abetc}. For this,  we need  to explain how commutators, normal subgroups and quotients behave under ultralimits. \linebreak In the next lemmas, we  assume that a  non-principal ultrafilter  $\omega$ is given: 
\begin{lemma}
	\label{lemma-limit}
	Let $\Gamma_n'$,  $\Gamma_n$ be    isometry groups of pointed  metric spaces $(X_n, x_n)$, with $\Gamma_n' < \Gamma_n$, and let $\Gamma_\omega', \Gamma_\omega$  be the limit groups of isometries of $X_\omega$.  Then:
	\begin{itemize}
		\item[(i)] $ [\Gamma_\omega, \Gamma_\omega] \subseteq [\Gamma_n, \Gamma_n]_\omega  $;
		\item[(ii)] if $\Gamma_n$ is abelian   (resp. $m$-step nilpotent, $m$-step solvable)  for $\omega$-a.e.$(n)$, then $\Gamma_\omega$ is abelian (resp. $m$-step nilpotent, $m$-step solvable);
		\item[(iii)] if $\Gamma_n' \triangleleft \Gamma_n$   for $\omega$-a.e.$(n)$, then   $\Gamma_\omega'$ is a normal subgroup of $\Gamma_\omega$.
		\end{itemize}
\end{lemma}
\begin{proof}
	Let $g_\omega = [a_{\omega,1}, b_{\omega,1}]\cdots [a_{\omega,l}, b_{\omega,l}] \in [\Gamma_\omega, \Gamma_\omega]$. We can write any $a_{\omega,i}$ as $\omega$-$\lim a_{n,i}$ and any $b_{\omega,i}$ as $\omega$-$\lim b_{n,i}$, for admissible sequences  $a_{n,i},  b_{n,i} \in \Gamma_n$. Moreover by Lemma \ref{composition}, since $l$ is finite, we get 
	$$[a_{\omega,1}, b_{\omega,1}]\cdots [a_{\omega,l}, b_{\omega,l}] = \omega\text{-}\lim ([a_{n,1}, b_{n,1}]\cdots [a_{n,l}, b_{n,l}]),$$
	so $g_\omega \in [\Gamma_n, \Gamma_n]_\omega$, showing (i).\\
	Let us prove (ii): if $\Gamma_n$ is abelian  for $\omega$-a.e.$(n)$, then by Lemma \ref{composition} we get directly that also $\Gamma_\omega$ is abelian.
	The nilpotent and the solvable case are proved by induction on $m$. If $m=0$, we are in the abelian case.
	Assume now that the claim holds for $(m-1)$-step nilpotent  (resp. solvable) groups, and let $\Gamma_n$ be $m$-step nilpotent (resp. solvable) for $\omega$-a.e.$(n)$.  The groups $[\Gamma_n, \Gamma_n]$ are $(m-1)$-step nilpotent (resp. solvable)  for $\omega$-a.e.$(n)$; so,  by the induction hypothesis, $[\Gamma_n, \Gamma_n] _\omega$ is $(m-1)$-step nilpotent (resp. solvable). But then,  by (i),   also $[\Gamma_\omega, \Gamma_\omega]$ is $(m-1)$-step nilpotent (resp. solvable) and therefore $\Gamma_\omega$ is $m$-step nilpotent (resp. solvable).\\
	Finally, let us show (iii). The limit group $\Gamma_\omega'$ is contained in $\Gamma_\omega$, by definition.\linebreak
	We take $g_\omega = \omega$-$\lim g_n \in \Gamma_\omega$ and $h_\omega = \omega$-$\lim h_n \in \Gamma_\omega'$. We know that  for $\omega$-a.e.$(n)$ there exists $h_n'\in \Gamma_n'$ such that $g_nh_n = h_n'g_n$, as $\Gamma_n'$ is normal in $\Gamma_n$.  \linebreak Since $(h_n') = (g_n h_n g_n^{-1})$ is an admissible sequence, it defines a limit isometry $h_\omega' \in \Gamma_\omega'$. Moreover by Lemma \ref{composition} we get $h_\omega' = g_\omega h_\omega g_\omega^{-1}$, so $\Gamma_\omega'$ is normal.
	\end{proof}

\begin{lemma}
	\label{lemma-limit-bis}
	 Let $(X_n,x_n,\sigma_n)$ be  pointed  spaces belonging to  $\textup{GCB}(P_0,r_0,\delta)$, and let $\Gamma'_n, \Gamma_n$ be   groups of $\sigma_n$-isometries of $X_n$ with $\Gamma'_n \triangleleft \Gamma_n$ for $\omega$-a.e.$(n)$. \linebreak 
	 Suppose that  the groups $\Gamma_n$ and $X_n$ satisfy condition  (a) of Theorem \ref{Dicotomia}. 
	  Then there is a natural isomorphism between  the ultralimit $Q_\omega$ of the groups $Q_n=\Gamma_n / \Gamma_n'$ acting on the quotient,  pointed spaces $(X'_n, x'_n)= \Gamma'_n \backslash  ( X_n, x_n )$ and the   group $\Gamma_\omega / \Gamma_\omega'$  (observe that these  are isometry groups  of the spaces $ X'_\omega$ and $\Gamma_\omega' \backslash X_\omega$ respectively, which are isometric by Corollary \ref{Dicotomia-geometrica}).\end{lemma}

\begin{proof}
  The groups $Q_n=\Gamma_n / \Gamma_n'$ act naturally on the pointed, quotient spaces  $(X'_n, x'_n)$,  where $x'_n$ is the image of $x_n$ under the projection map $X_n \to  X'_n$,
 and their ultralimit $Q_\omega$  acts  on the ultralimit  $(X'_\omega, x'_\omega)$ of the  $(X'_n, x'_n)$. \linebreak
We define    $\Phi \colon Q_\omega \to \Gamma_\omega / \Gamma_\omega'$ as follows. An element    $q_\omega \in Q_\omega$  is  a sequence of admissible isometries $(q_n=g_n\Gamma_n')$ of $Q_n$ acting  on $X'_n$. As   $\Gamma_n$ and $X_n$  satisfy condition (a) of Theorem \ref{Dicotomia},    we have   $d(q_n  x'_n,  x'_n) \leq L$   for $\omega$-a.e.$(n)$. \linebreak
 We then  choose an admissible sequence $(g_n)$ of isometries of $X_n$ belonging to the class $q_n$, 
and define $\Phi(q_\omega)$ as the class   of $\omega$-$\lim g_n$ modulo $\Gamma'_\omega$.\linebreak It   clearly does not depend on the choice of the admissible representative $(g_n)$, and yields a surjective homomorphism. 
 The injectivity of $\Phi$ follows directly from the assumptions on $\Gamma_n'$. 
 Indeed, if    $q_\omega = \omega$-$\lim q_n$ and  $(g_n)$ is an admissible sequence of isometries of $X_n$ belonging to the classes $(q_n)$,   then   $\Phi(q_\omega) \in \Gamma'_\omega$ implies that there exists an admissible sequence $(g'_n)$ of isometries in $\Gamma_n'$ such that $\omega$-$\lim g'_n = \omega$-$\lim g_n$. In particular  $\forall \varepsilon > 0$ we have \linebreak  $d(g_n x_n, g'_n x_n) \leq \varepsilon$ for $\omega$-a.e.$(n)$. By the assumption (a) of   \ref{Dicotomia} we deduce \linebreak that $g_n = g'_n$,  hence $g_n \in \Gamma_n'$,  for $\omega$-a.e.$(n)$.  Thus $\omega$-$\lim q_n$ is the identity. 
\end{proof}


\begin{proof}[Proof of Theorem \ref{compactness-abetc}]
We have a sequence of pointed spaces $(\bar X_n, \bar x_n)$ with pointed, universal coverings $(X_n, x_n, \sigma_n)$ belonging to GCB$(P_0,r_0,\delta)$, and a sequence of normal coverings $\pi_n: \hat X_n \to  \bar X_n$. This means that $\bar X_n = \Gamma_n \backslash X_n$ and  $\hat X_n = \Gamma_n' \backslash X_n$, where $\Gamma_n, \Gamma_n'$ are discrete, torsionless groups of isometries of $X_n$, and $\Gamma_n'$ are normal subgroups of $\Gamma_n$. We denote by $\bar p_n \colon X_n \to \bar X_n$ and $\hat p_n \colon X_n \to \hat X_n$ the natural universal covering  maps, so  $\pi_n \circ \hat p_n = \bar p_n$ for all $n$.\linebreak
We take a non-principal ultrafilter $\omega$,  fix points $x_n, \hat x_n, \bar x_n$ of $X_n, \hat X_n, \bar X_n$ respectively, with $\hat p_n (x_n) = \hat x_n$ and $\bar p_n(x_n)=\bar x_n$, consider the pointed spaces $(X_n,x_n), (\hat X_n, \hat x_n), (\bar X_n, \bar X_n)$ and the corresponding ultralimits $X_\omega, \hat X_\omega, \bar X_\omega$.\linebreak
We consider the limit groups  $\Gamma_\omega$, $\Gamma_\omega'$ of   $\sigma_\omega$-isometries  of $X_\omega$.
Since the groups $\Gamma_n$ are cocompact with codiameter $\leq D$ we know from Example \ref{ex-coco} that sys$(\Gamma_n, X_n) \geq s_0(P_0,r_0,\delta,D) > 0$. Moreover, since $\Gamma_n' < \Gamma_n$,   the same estimate holds for $\Gamma_n'$. Therefore we are in case (a) of Theorem \ref{Dicotomia}, both for the sequence $(\Gamma_n)$ and for $(\Gamma_n')$. We conclude that $\Gamma_\omega$ and $\Gamma_\omega'$ are  discrete and torsion-free, and moreover $\Gamma_\omega' \triangleleft \Gamma_\omega$ by Lemma \ref{lemma-limit}.(iii).\\
Now, calling $\pi_\omega$ and $ \hat p_\omega, \bar p_\omega$  the ultralimit maps of the sequences $(\pi_n), (\hat p_n), (\bar p_n)$   we have by definition $\pi_\omega \circ \hat p_\omega = \bar p_\omega$. Moreover by Corollary \ref{Dicotomia-geometrica}.(ii) we know that $\hat p_\omega$ and $\bar p_\omega$ are covering maps with automorphism groups $\Gamma_\omega'$ and $\Gamma_\omega$ respectively. It is then straightforward to conclude that $\pi_\omega$ is a covering. \linebreak Since $\Gamma'_\omega$ is normal in  $\Gamma_\omega$, by the classical covering theory   we know that the automorphism group of $\pi_\omega$ coincides with the group $\Gamma_\omega / \Gamma_\omega'$, which is isomorphic to the ultralimit group $Q_\omega$ of the $Q_n= \Gamma_n / \Gamma'_n$, by Lemma \ref{lemma-limit-bis}. \linebreak
Furthermore, as every  $Q_n$ is abelian (resp. $m$-step nilpoten or solvable), then  the same holds for the ultralimit  $Q_\omega$ and for $\Gamma_\omega / \Gamma_\omega'$, by Lemma \ref{lemma-limit}.(ii). 
Finally, we observe that the group $\Gamma_\omega'$ is not elementary. Indeed, if $\Gamma_\omega'$  was elementary  then  it would be cyclic (being a torsionless subgroup of $\Gamma_\omega$ without parabolics), and by   the   exact sequence
$$1 \to \Gamma_\omega' \to \Gamma_\omega \to \Gamma_\omega/\Gamma_\omega' \to 1,$$
with both $\Gamma_\omega'$ and $\Gamma_\omega/\Gamma_\omega'$ solvable, we would deduce that $\Gamma_\omega$ is solvable. This contradicts the fact that $\Gamma_\omega$ contains a free group.
The closure then follows  once again by Corollary \ref{Dicotomia-geometrica}.
\end{proof}

\section{Data availability}
We do not analyse or generate any datasets.

\bibliographystyle{alpha}
\bibliography{tits_bicombing}

\end{document}